\documentclass[11pt]{article}

\usepackage{latexsym}
\usepackage{amssymb}
\usepackage{amsthm}
\usepackage{amscd}
\usepackage{amsmath}
\usepackage{mathrsfs}
\usepackage{graphicx}
\usepackage{hyperref}
\usepackage{shuffle}

\usepackage[all]{xy}
\input xy \xyoption{frame}
\xyoption{dvips}

\usepackage[colorinlistoftodos]{todonotes}
\usetikzlibrary{chains,scopes,decorations.markings}

\theoremstyle{definition}
\newtheorem* {theorem*}{Theorem}
\newtheorem* {conjecture*}{Conjecture}
\newtheorem{theorem}{Theorem}[section]

\theoremstyle{definition}

\newtheorem* {example*}{Example}

\newtheorem{lemma}[theorem]{Lemma}
\theoremstyle{definition}
\newtheorem{definition}[theorem]{Definition}
\theoremstyle{definition}

\newtheorem{proposition}[theorem]{Proposition}
\newtheorem{corollary}[theorem]{Corollary}

\newtheorem{remark}[theorem]{Remark}
\theoremstyle{definition}
\newtheorem {example}[theorem]{Example}
\theoremstyle{definition}

\theoremstyle{definition}

\theoremstyle{definition}

\theoremstyle{definition}

\newcommand{\dense}{\mathsf{dense}}

\def\modu{\ (\mathrm{mod}\ }

\def\({\left(}
\def\){\right)}

\newcommand{\CC}{\mathbb{C}}

\newcommand{\cP}{\mathcal{P}}

\newcommand{\cO}{\mathcal{O}}

\newcommand{\cS}{\mathcal{S}^G_K}
\newcommand{\cI}{I}

\def\cW{\mathcal{W}}

\def\NN{\mathbb{N}}

\def\CC{\mathbb{C}}

\def\ZZ{\mathbb{Z}}
\def\Aut{\mathrm{Aut}}

\def\GL{\mathrm{GL}}

\def\cyc{\operatorname{cyc}}

\newcommand{\supp}{\mathrm{supp}}

\newcommand{\cN}{\mathcal{N}}

\def\barr{\begin{array}}
\def\earr{\end{array}}
\def\ba{\begin{aligned}}
\def\ea{\end{aligned}}
\def\be{\begin{equation}}
\def\ee{\end{equation}}

\def\Cyc{\mathrm{Cyc}}
\def\Fix{\mathrm{Fixed}}

\def\qquand{\qquad\text{and}\qquad}
\def\quand{\quad\text{and}\quad}

\def\quord{\quad\text{or}\quad}

\def\inv{\operatorname{inv}}

\def\I{\cI}

\def\hs{\hspace{0.5mm}}

\def\id{\mathrm{id}}
\def\PP{\mathbb{Z}_{>0}}

\def\ben{\begin{enumerate}}
\def\een{\end{enumerate}}

\def\cE{\mathcal V}

\def\hs{\hspace{0.5mm}}

\newcommand{\fpf}{{\mathsf{fpf}}}

\def\Des{\operatorname{Des}}
\def\NDes{\operatorname{NDes}}

\def\ellhat{\hat\ell}

\newcommand{\std}{\operatorname{std}}

\newcommand{\SO}{\operatorname{SO}}
\renewcommand{\O}{\operatorname{O}}
\newcommand{\Sp}{\operatorname{Sp}}

\newcommand{\rank}{\operatorname{rank}}

\def\arcstart{\ \xy<0cm,-.06cm>\xymatrix@R=.1cm@C=.10cm }
\newcommand{\arcstartc}[1]{\ \xy<0cm,-.15cm>\xymatrix@R=.1cm@C=#1cm}

\def\ellhat{\hat\ell}

\def\Neg{\operatorname{Negated}}

\def\NCSM{\textsf{NCSP}}

\def\bar{\overline}

\def\O{\mathrm{O}}

\newcommand{\W}{W^{\mathsf{B}}}
\newcommand{\WD}{W^{\mathsf{D}}}

\usepackage{mathtools}
\usepackage{tikz-cd}

\definecolor{darkred}{rgb}{0.7,0,0} % darkred color
\newcommand{\defn}[1]{{\color{darkred}\emph{#1}}}

\newcommand{\shD}{\sh_{\mathrm{D}}}
\newcommand{\ccM}{M^G_K}
\newcommand{\ccMD}{M^{(n,n)}_{\DI}}
\newcommand{\cycles}{C}

\newcommand{\rankD}{\rank_{\mathsf{D}}}

\newcommand{\es}{\mathrm{es}}
\newcommand{\nb}{\mathsf{nb}}

\newcommand{\cMM}{\mathsf{Shapes}_{\mathrm{D}}}

\newcommand{\points}{\mathsf{points}}
\newcommand{\APoints}{\mathrm{Points}}
\newcommand{\Points}{\mathrm{Points}_{>0}}
\newcommand{\defequals}{=}
\newcommand{\triv}[1]{\mathsf{triv}(#1^+)}

\newcommand{\cM}{\mathsf{Shapes}}
\newcommand{\sh}{\mathsf{sh}}

\renewcommand{\bot}{\mathsf{gen}}

\newcommand{\gammaDense}{\gamma_{\mathsf{dense}}}
\newcommand{\deltaDense}{\delta_{\mathsf{dense}}}
\newcommand{\epsilonDense}{\epsilon_{\mathsf{dense}}}
\newcommand{\etaDense}{\hat\epsilon_{\mathsf{dense}}}

\newcommand{\des}{\mathrm{des}}
\newcommand{\asc}{\mathrm{asc}}

\newcommand{\rev}{\mathsf{rev}}
\newcommand{\Triv}[1]{\mathrm{Triv}(#1^+)}
\newcommand{\NCSP}{\mathsf{NCSP}}

\newcommand{\bei}{\begin{itemize}}
\newcommand{\eei}{\end{itemize}}

\newcommand{\inc}{\mathsf{embed}}

\newcommand{\cEABrion}{\cE^G_K}
\newcommand{\cABrion}{\cW^G_K}%{\cW_{K}^{G}}
\newcommand{\RSphi}{\psi^G_K}
\newcommand{\phiRS}{\RSphi}
\newcommand{\tphiRS}{\phi^G_K}
\renewcommand{\neg}{\mathsf{neg}}
\newcommand{\fix}{\mathsf{fix}}

\newcommand{\NRes}{\mathrm{NRes}}

\newcommand{\Ad}{\mathrm{Ad}}
\newcommand{\AI}{\mathrm{AI}}
\newcommand{\AII}{\mathrm{AII}}
\newcommand{\AIII}{\mathrm{AIII}}
\newcommand{\BI}{\mathrm{BI}}
\newcommand{\CI}{\mathrm{CI}}
\newcommand{\CII}{\mathrm{CII}}
\newcommand{\DI}{\mathrm{DI}}
\newcommand{\DIII}{\mathrm{DIII}}

\newcommand{\Ifpf}{\I_\fpf}

\newcommand{\cA}{\mathcal{V}}
\newcommand{\cAfpf}{\cA_\fpf}
\newcommand{\cAfpfA}{\cAA_\fpf}

\newcommand{\cAfpfC}{\cAfpf^{\mathsf{C}}}
\newcommand{\cAfpfD}{\cAfpf^{\mathsf{D}}}

\newcommand{\cAD}{\cA^{\mathsf{D}}}

\newcommand{\cAB}{\cA^{\mathsf{B}}}
\newcommand{\cAA}{\cA^{\mathsf{A}}}

\newcommand{\precsimD}{\precsim_{\mathsf{D}}}

\newcommand{\llD}{\ll_{\mathsf{D}}}

\newcommand{\incDIIIEven}{\inc_{\DIII}^{\text{Even}}}
\newcommand{\incDIIIDI}{\inc_{\DIII}^{\DI}}

\newcommand{\regroupB}{\mathsf{regroup}_{\mathrm{B}}}
\newcommand{\regroupD}{\mathsf{regroup}_{\mathrm{D}}}
\newcommand{\NCSPB}{\NC_{\mathrm{B}}}
\newcommand{\NCSPD}{\mathsf{SP}_{\mathrm{D}}}
\renewcommand{\div}{{\mathrm{div}}}

\newcommand{\ddelta}{\bot_{\mathrm{D}}}
\newcommand{\cRD}{\mathcal{R}}

\def\vsp{6pt}

\newcommand{\Pairs}{\mathrm{Pairs}}
\newcommand{\NC}{\mathsf{NC}}
\newcommand{\CycPM}{\Cyc^\pm}
\newcommand{\cED}{\cA}
\usepackage{fullpage}

\numberwithin{equation}{section}
\allowdisplaybreaks[1]
\UseCrayolaColors

\begin{document}
\title{Brion atoms for classical types}
\author{
    Eric Marberg\\
    Department of Mathematics \\
    HKUST \\
    {\tt eric.marberg@gmail.com}
}

\date{}

\maketitle

\begin{abstract}
Let $G$ be a classical group defined over the complex numbers with a Borel subgroup $B$.
Choose a holomorphic involution of $G$ and let $K$ be its set of fixed points.
The group $K$ acts on the flag variety $G/B$ with finitely many orbits
and Brion has derived a general formula for the cohomology classes of the corresponding orbit closures as linear combinations of Schubert classes.
This article provides a uniform description of the sets of Weyl group elements (which we refer to as Brion atoms) that index the terms in this formula.
We also include some applications of our classifications, including a refinement of Brion's cohomology formula.
\end{abstract}

%\tableofcontents

\section{Introduction}

Brion \cite{Brion2001} has derived a general cohomology formula for
the closures of the orbits of a  {symmetric subgroup} 
acting on a complete flag variety. This formula expresses the cohomology class of each orbit closure as a linear combination of Schubert classes
with coefficients that are powers of two.
This article is concerned with the combinatorial problem of describing the terms in this formula
in an explicit and uniform way for all classical types. The rest of this introduction provides 
a brief explanation of Brion's result in order to summarize our main theorems.

\subsection{Symmetric subgroups}

Let $G$ be a connected, reductive algebraic group defined over the complex numbers. Assume $G$ has a holomorphic automorphism $\theta$ that is an involution, meaning that $\theta=\theta^{-1}$. 

Fix a Borel subgroup $B$ containing a maximal torus $T$
and assume that both groups are preserved by $\theta$. 
The subgroup of fixed points 
$ K = G^\theta \defequals  \{ g \in G : g=\theta(g)\}$
acts on the complete flag variety $G/B$ with finitely many orbits \cite{Mat79} and 
the group $K$ is called a \defn{symmetric subgroup} of $G$.

Let $W=N_G(T)/T$ be the Weyl group of $G$. Write
 $S$ for its set of simple generators, $\ell : W\to \NN=\{0,1,2,\dots\}$ for its Coxeter length function, and $w_0 \in W$ for the unique longest element.
 Because $T$ is $\theta$-stable, the involution $\theta \in \Aut(G)$ descends to an automorphism $\theta\in \Aut(W)$ which preserves $S$. 
It will be convenient to introduce the notation 
$ \Theta \defequals  \theta|_W \circ \Ad(w_0)= \Ad(w_0)\circ \theta|_W \in \Aut(W)$ for the composition of $\theta$
with the inner automorphism of $W$ induced by $w_0$.

The groups $G$ and $K$ that can arise from this setup have been completely classified \cite{MO90}.
 Table~\ref{tbl1}
shows the possibilities for $(G,K)$ when $G$ is classical, along with the  automorphism $\Theta \in \Aut(W)$.
By a \defn{classical group}, we mean either the general linear group
 $\GL(n)=\GL(n,\CC)$
 for any positive integer $n$,
the special orthogonal group $\SO(n)=\SO(n,\CC)$ for any integer $n\geq 2$, or the symplectic group $\Sp(n)=\Sp(n,\CC)$ for any even positive integer $n$. 

We identify the respective Weyl groups of $\GL(n)$, $\SO(2n+1)$ and $\Sp(2n)$, and $\SO(2n)$
with the symmetric group $S_{n}$, the signed symmetric group $\W_n$, and the even-signed symmetric group $\WD_n$; see Section~\ref{per-sect} for our precise conventions.
Let $\ast$ and $\diamond$ denote the usual diagram automorphisms of $S_{n}$ and $\WD_n$ (see Section~\ref{per-sect}).

\def\hhline{\\ & &&&&&  \\ [-\vsp]\hline & & &&&& \\ [-\vsp]}
\def\gap{\\[-\vsp]&&&&&&\\}
\begin{table}[h]
\begin{center}
{\small
\begin{tabular}{| l | l | l | l | l |l | l |}
\hline&&&&&& \\[-\vsp]
Type & Parameters &  $G$ & $K=G^\theta$ & $W$ & $\Theta$ & Reference   
\hhline
AI & $n\in \NN$ & $\GL(n)$    & $\O(n)$ & $S_{n}$ & $\id$ &  \cite[\S2.2]{Wyser} 
\gap
AII  & $n\in \NN$ even & 
& $\Sp(n)$ & 
& $\id$ &  \cite[\S2.4]{Wyser} 
\gap
AIII & $n=p+q$ &   
& $ \GL(p)\times \GL(q) $ 
& 
& $\ast$ &  \cite[\S2.1]{Wyser} 
\hhline
BI & $2n+1=p+q$ & $\SO(2n+1)$ & $\mathrm{S}(\O(p)\times \O(q))$  & $\W_n$ & $\id$ &  \cite[\S3.1]{Wyser} \hhline
CI & $n\in \NN$ & $\Sp(2n)$   & $\GL(n)$ & $\W_n$ & $\id$ &  \cite[\S4.2]{Wyser} 
\gap
CII  & $2n=p+q$  with $p$ and $q$ even 
&  %$\Sp(2n)$ 
& $\Sp(p)\times \Sp(q)$ & 
%$\W_n$ 
& $\id$ &  \cite[\S4.1]{Wyser} 
%\\  & with $p$ even & & & & & 
\hhline
DI  & $2n=p+q$ and $k=\frac{|p-q|}{2}$
& $\SO(2n)$ & $ \mathrm{S}(\O(p) \times \O(q))$  & $\WD_n$ & $(\diamond)^k$ &  \cite[\S5.1, \S5.3]{Wyser}  
\gap
DIII & $n\in \NN$ & %$\SO(2n)$ 
& $\GL(n)$ 
& 
%$\WD_n$
& $(\diamond)^n$ &  \cite[\S5.2]{Wyser}  
\gap\hline
\end{tabular}}
\end{center}
\caption{
Symmetric subgroups $K=G^\theta$ of classical groups $G$,
 with Weyl group $W$ and associated automorphism $\Theta = \Ad(w_0)\circ \theta|_W$.
We require $p,q\in \NN$ in all cases. 
%The last column provides references in Wyser's thesis \cite{Wyser} for the parametrization of all $K$-orbits in $G/B$.
Our types follow \cite{MO90}, except that we label both DI/DI$'$ from \cite{MO90} as DI. As in \cite{MO90}, we have no type DII.
}\label{tbl1}
\end{table}

 \subsection{Brion's cohomology formula}\label{brion-intro-sect}
%\subsection{Schubert varieties and Schubert classes}

Besides $K$, the Borel subgroup $B$ also acts on the flag variety $G/B$ with finitely many orbits. 
%By definition,
%each element of $G/B$ is a left coset of the form $gB$ for some $g \in G$, and 
%any $B$-orbit in $G/B$ is a set of such cosets.
%Every $B$-orbit in $G/B$ contains at least one element of the form $wB$ with $w \in N_G(T)$,
%and  if we have two elements $v,w \in N_G(T)$ then the cosets $vB$ and $wB$ can only belong to the same $B$-orbit if $vT = wT$, meaning that $v$ and $w$ represent the same element of  $W = N_G(T) /T$.
%In this sense, 
The finite set of $B$-orbits in $G/B$ is naturally in bijection with  
the Weyl group $W$,
and 
we write $BwB/B$ for the $B$-orbit in $G/B$ corresponding to $w \in W$.
%Explicitly, this is the orbit of any $g \in N_G(T)$ with $gT = w$.

The flag variety $G/B$ is a homogeneous space and so is endowed with the Zariski topology.
The closures $\overline{BwB}/B \supseteq BwB/B$ are called \defn{Schubert varieties}.
We adopt the indexing convention
\be\label{Xw-eq} X_w = \overline{B  w_0w  B}/B\quad\text{for $w \in W$}\ee
where $w_0 \in W$ is the longest element in the Weyl group.
%This notation aligns with the following property:
% the set-wise containment $X_v \subseteq X_w$ holds for two elements $v,w \in W$ if and only if we have $v\leq w$ where $\leq$ denotes the \defn{(strong) Bruhat order} on $W$.

Let $H^*(G/B)$ denote the cohomology ring of the flag variety $G/B$ with integer coefficients.
Each closed subset $X$ of $G/B$ determines a cohomology class denoted $[X]\in H^*(G/B)$.
The corresponding \defn{Schubert classes} $[X_w]$ for $w \in W$
turn out to be a $\ZZ$-basis for $H^*(G/B)$.

Denote the finite set of $K$-orbits in $G/B$ as $\{\cO_\gamma\}_{\gamma \in \Gamma^G_K}$
where $\Gamma^G_K$ is some index set.
For each of the classical types in Table~\ref{tbl1}, there is a standard way of identifying  $\Gamma^G_K$ with an explicit set
of combinatorial objects (specifically, certain matchings with signs attached to isolated vertices) \cite{MO90,Yam97}.
These indexing conventions are reviewed in Section~\ref{index-sect}.

For each $\gamma \in \Gamma^G_K$ let $Y_\gamma = \overline{\cO_\gamma}\subseteq G/B$
be the Zariski closure of the $K$-orbit $\cO_\gamma$.
Then there are unique integers $c_{\gamma w} \in \ZZ$ expressing the associated  cohomology class
$ [Y_\gamma] = \sum_{w \in W} c_{\gamma w} [X_w]$
as a linear combination of Schubert classes.

It is a natural and interesting problem to compute the coefficients $c_{\gamma w}$.
In an abstract sense, this task is completely solved by work of Brion \cite{Brion2001},
which derives a general formula 
\be\label{brion}
 [Y_\gamma] = \sum_{w\in \cABrion(\gamma)} 2^{d_\gamma(w)} [X_{w^{-1}}] \in H^*(G/B)\ee
where $\cABrion(\gamma)$ is a  subset of equal-length elements of $W$ and $d_\gamma : \cABrion(\gamma) \to \{0,1,2,\dots\}$ is a certain map.
To define things precisely, we need to review
the \defn{weak order graph} for $(G,K)$ from \cite{Brion2001}. 

This graph  has vertex set $\Gamma^G_K$
and directed edges labeled by simple generators  $s\in S$. 
The edges are described as follows.
Fix $s \in S$,  write $\alpha$ for the corresponding simple root, and let $P_\alpha$ be the minimal parabolic subgroup of $G$ of type $\alpha$ containing $B$.
Consider the projection \[ \pi_\alpha: G/B \to G/P_\alpha.\]   
For each $\gamma \in \Gamma^G_K$, the set $\pi_\alpha^{-1}(\pi_\alpha(\cO_\gamma))$ contains a unique dense $K$-orbit \cite[\S1.4]{Wyser}.
The weak order graph has an edge 
$\gamma \xrightarrow{s} \delta$
whenever the dense $K$-orbit in $\pi_\alpha^{-1}(\pi_\alpha(\cO_\gamma))$  is $\cO_\delta$ 
for an index $\gamma \neq \delta \in \Gamma^G_K$.
If there is an edge of this form then the restriction of   $\pi_\alpha$ to the closure
$Y_\gamma = \overline{\cO_\gamma}$ has degree $1$ or $2$ \cite[\S1]{Brion2001}. 
The \defn{degree} of the edge $\gamma \xrightarrow{s}\delta$ is the degree of the restricted map  $\pi_\alpha |_{Y_\gamma}$.

Each weak order graph  has a unique sink vertex $\gammaDense \in \Gamma^G_K$,
which indexes the unique $K$-orbit whose closure $Y_{\gammaDense}$ is $G/B$.
Given some $\gamma \in \Gamma^G_K$
consider any path 
\[\gamma=\gamma^0 \xrightarrow{s_1}\gamma^1\xrightarrow{s_2} \cdots \xrightarrow{s_k}\gamma^k= \gammaDense\] in the weak order graph that starts at $\gamma$ and ends at $\gammaDense$. Then $w=s_1s_2\cdots s_k$
 is always a reduced expression for an element of the Weyl group,
 and the set of \defn{Brion atoms} $\cABrion(\gamma)$ in \eqref{brion} consists of   all elements $w$ that arise from paths in this way.
Similarly, the number of degree 2 edges in the chosen path depends only on $w$,
and the map $d_\gamma$ in \eqref{brion} assigns $w$ to this number.

Brion's formula tells us some nontrivial information about the coefficients $c_{\gamma w}$. For example, they are each zero or integral powers of $2$. However, while completely general,
the formula \eqref{brion} is not very explicit.
It is not clear, for example, 
when the Schubert expansion of $[Y_\gamma]$ is multiplicity-free in the sense 
of having all coefficients $c_{\gamma w} \in \{0,1\}$, or when it consists of just a single term.

The goal of this article is to explain a more direct statement of Brion's formula for the classical types in Table~\ref{tbl1},
which will make it straightforward to answer such questions.
Our results build on earlier work  \cite{BurksPawlowski,CJW,CU1,CU2,HM,HMP2} addressing types A, B, and C.
The main novelty (and difficulty) of our present undertaking is to handle 
type D and to handle all classical types in a unified way.

\subsection{Main results}

Fix a pair $(G,K)$ of classical type.
Each $K$-orbit in $G/B$ has the form 
\[\cO = KgB/B\quad\text{for various $g \in G$ with $g^{-1}\theta(g) \in N_G(T)$}.\]
The image $w(\cO)$ of $g^{-1}\theta(g)$ in $W = N_G(T)/T$ is independent of the choice of $g$ \cite[\S0]{RichSpring}.
Define 
\be
\tphiRS : \Gamma^G_K \to W 
\quad
\text{to be the map that sends $\gamma \mapsto w(\cO_\gamma)$}.
\ee
Following \cite[Def.~1.3.6]{Wyser}, we refer to this
as the \defn{Richardson--Springer map}.
It will be convenient to work with a modified version of this definition. Namely:

\begin{definition}\label{rs-def}
Let $\phiRS : \Gamma^G_K \to W 
$
be the map with 
$\phiRS(\gamma) = w_0\cdot \tphiRS(\gamma)^{-1}$ and define
\[ \cI^G_K \defequals  \left\{ \phiRS(\gamma) : \gamma \in \Gamma^G_K\right\}\subseteq W.\]
Then for each $z \in \cI^G_K$ we define an associated set of \defn{extended Brion atoms}
\be\label{cEABrion-eq}
\cEABrion(z) \defequals  \bigcup_{\substack{ \gamma \in \Gamma^G_K \\ \phiRS(\gamma)=z}} \cABrion(\gamma) \subseteq W
\ee
where $\cABrion(\gamma)$ is the set of Brion atoms from \eqref{brion} defined via the weak order graph on $\Gamma^G_K$.
\end{definition}

In each classical type, the sets $\Gamma^G_K$, $\cI^G_K$ and $\cEABrion(z)$ have more explicit descriptions; see Tables~\ref{tbl2}, \ref{rs-image-tbl} and \ref{extended-brion-tbl}.
Our main results involve some additional combinatorial data presented in Sections~\ref{A-shape-sect}, \ref{BC-shape-sect}, and \ref{D-shape-sect}.
Specifically, for each $z \in \cI^G_K$   we will define:
\bei
\item a set $\cM^G_K(z)$ of allowable shapes, consisting of certain matchings (see Table~\ref{shapes-summary-tbl}); and

\item a \defn{shape operator} $\sh^G_K$ which gives a surjective map $\cEABrion(z)\to \cM^G_K(z)$.
\eei
Recall that a \defn{matching} is a finite set of pairwise disjoint sets of size one or two.
In all classical types, the indexing set $\Gamma^G_K$
will consist of certain \defn{clans}, which are matchings
with signs attached to isolated vertices.
If $\gamma \in \Gamma^G_K$ is such a clan
and $z = \phiRS(\gamma)$ then 
every $M \in \cM^G_K(z)$ 
will be a matching on a subset of the isolated vertices of $\gamma$. In this setup, we say that $M$
is \defn{$\gamma$-aligned}
if $M$ does not pair two isolated vertices of $\gamma$ with the same sign; see Definition~\ref{aligned-def}.

The constructions outlined above will let us define 
\be
\cEABrion(z,M) \defequals  \left\{ w \in \cEABrion(z) : \sh^G_K(w) = M\right\}
\quad
\text{for each $M \in \cM^G_K(z)$.}
\ee
We may now state our first main theorem.

\begin{theorem} 
\label{main-thm}
Fix a pair $(G,K)$ of classical type.
Suppose $\gamma \in \Gamma^G_K$ and $z = \phiRS(\gamma) \in \cI^G_K$. 
Then   
\[
\cABrion(\gamma) = \bigsqcup_{\substack{M \in  \cM^G_K(z)
\\ M\text{ is $\gamma$-aligned}}
}  \cEABrion(z,M)
.
\]
Moreover, there is a map $e_z : \cM^G_K(z)\to \NN$ and an integer $\rho^G_K(z) \in  \NN$ such that 
\[d_\gamma(w) = e_z(\sh^G_K(w))
\quand \ell(w) = \rho^G_K(z)\quad\text{for all $w \in \cABrion(\gamma)$.}
\]
\end{theorem}

Formulas for  $d_\gamma(w)$ and $\rho^G_K(z)$ appear in Table~\ref{extended-brion-tbl} (cf. Propositions~\ref{drho-prop} and \ref{dgamma-prop}).
These show that $d_\gamma$ only depends on $z=\phiRS(\gamma)$ outside types BI and CI. The existence of $e_z:\cM^G_K(z)\to \NN$ for these types
is checked in Section~\ref{proof-sect2}.

Thus, to compute all terms in Brion's cohomology formula \eqref{brion},
one just needs to be able to efficiently generate the sets $\cEABrion(z,M)$.
For this task, we shall additionally define
\bei
\item for each $z \in \cI^G_K$ and $M \in \cM^G_K(z)$ a certain  \defn{generator} $\bot^G_K(z,M) \in W$; and
\item a transitive relation $\precsim^G_K$ on the Weyl group $W$.
\eei
Recall that a poset is \defn{graded} if there exists a rank function assigning an integer to each poset element that increases by exactly one
after applying any covering relation.
%We arrive at our second main theorem.

\begin{theorem} 
\label{main-thm2}
Continue to fix a pair $(G,K)$ of classical type. Choose an element $z  \in \cI^G_K$ and a shape $M \in \cM^G_K(z)$.
Then the set of extended Brion atoms for $z$ with shape $M$ is
\[ 
\cEABrion(z,M)
= \left\{ w \in W : \bot^G_K(z,M) \precsim^G_K w\right\}.
\]
Additionally, this finite set is  
a graded poset relative to $\precsim^G_K$.
\end{theorem}

The fact that  $\cEABrion(z,M)$ is a graded poset under $\precsim^G_K$ with a clearly identified minimum is a key point of this theorem, as it makes it efficient 
to generate $\cEABrion(z,M)$
from the distinguished element $ \bot^G_K(z,M)$.
Abstractly, this means that the undirected Hasse diagram is a connected, bipartite graph. Any such graph can be oriented to define a graded poset, and the extra information needed to define this orientation is what our theorem provides for the Brion atom sets. 
 
To give substance to these results, we must first specify a number of relevant definitions for each family of symmetric pairs listed in Table~\ref{tbl1}.
 We distribute this material across Sections~\ref{A-shape-sect}, \ref{BC-shape-sect}, and \ref{D-shape-sect}
 following some preliminaries in Sections~\ref{index-sect}, \ref{order-sect}, and \ref{shape-prelim-sect}.
The proofs of Theorems~\ref{main-thm} and \ref{main-thm2} appear in Sections~\ref{proof-sect1},
\ref{proof-sect2}, and \ref{D-proof-sect}.
Section~\ref{app-sect}  concludes with some applications of our main theorems.

\subsection{Nested descent sets}
Our shapes for Brion atoms 
are formulated to
make
checking the key property of $\gamma$-alignment as simple as possible in each classical type.
However, these constructions are not merely heuristic,
and can be given the following type-independent interpretation.

Fix $z \in \cI^G_K$ and $w \in \cEABrion(z)$.
Choose a reduced expression $w=s_1s_2\cdots s_k$
and for each $1\leq i\leq k$ define $r_i = s_1s_2\cdots s_{i-1}s_i s_{i-1}\cdots s_2s_1$.
 Recall that the \defn{left inversion set} 
$ N_L(w) = \left\{ r_1,r_2,\dots,r_k\right\}$ 
does not depend on the choice of reduced expression.

Define $y_k = \phiRS(\gammaDense)$,
and for each $1\leq j\leq k$ 
define $y_{j-1}$ to be either
 $\Theta(s_j)y_j = y_js_j$ when equality holds or else $ \Theta(s_j)y_js_j$.
Then it can be shown \cite{NDes} that $z=y_0$ and 
that the set
\be \NDes^G_K(w) = \Bigr\{ r_{j} :1\leq j \leq k\text{ with }  \Theta(s_j)y_j = y_js_j \Bigr\}\subseteq N_L(w)\ee
is also  independent of the initial choice of reduced expression.
Moreover, if $v,w \in \cEABrion(z)$, then 
\be\label{sh-ndes-eq}\sh^G_K(v) =\sh^G_K(w)\quad\text{if and only if}\quad
\NDes^G_K(v) = \NDes^G_K(w).\ee
The fascinating properties of the \defn{nested descent sets} $\NDes^G_K(w)$, which make the equivalence asserted in \eqref{sh-ndes-eq} transparent, will be pursued elsewhere \cite{NDes}.

As a final comment, for each $z \in \cI^G_K$ and $M \in \cM^G_K(z)$, consider the sums
\be
\label{Sigma-eq}
\textstyle
\Sigma^G_K(z,M) = \sum_{w \in \cEABrion(z,M)} [X_{w^{-1}}]
=\sum_{\bot^G_K(z,M) \precsim^G_K w \in W}  [X_{w^{-1}}]  \in H^*(G/B).\ee
These classes, or the polynomials obtained by replacing $[X_{w^{-1}}]$ by an appropriate classical type \defn{Schubert polynomial},
may be interesting to study. 
The polynomial versions in types AI and AII recover the \defn{involution Schubert polynomials} from \cite{HMP2,HMP6,WY}, % (see also \cite{BurksPawlowski,WY0}), 
and our main results imply:

\begin{corollary} If $(G,K)$ is of classical type and $\gamma \in \Gamma^G_K$ has $z = \phiRS(\gamma) \in\cI^G_K$, then
\[\textstyle 
 [Y_\gamma] = \sum_{\text{$\gamma$-aligned } M \in \cM^G_K(z)} 2^{e_z(M)} \Sigma^G_K(z,M) \in H^*(G/B).
\]
\end{corollary}

 \subsection*{Acknowledgments}

The author is grateful to Zachary Hamaker and Brendan Pawlowski for several helpful discussions.
This article is based on work supported by the National Science Foundation under grant DMS-1929284 while the author was in residence at the Institute for Computational and Experimental Research in Mathematics in Providence, RI, during the Categorification and Computation in Algebraic Combinatorics semester program.
The author was also partially supported by Hong Kong RGC grants 16304122 and 16304625.

\section{Orbit indexing sets}
\label{index-sect}

This section contains preliminaries 
on Coxeter groups and certain signed matchings called \defn{clans}.
These objects make up the standard indexing sets for 
the $K$-orbits in $G/B$
for each classical type.

\subsection{Classical Weyl groups}\label{per-sect}

We formally define the classical Weyl groups from the introduction as follows.
Let  $S_{n}$ be the group of all permutations of  
$ [n] \defequals  \{1,2,\dots,n\}.$ For this Coxeter group, the set of simple generators is $S= \{t_1,t_2,\dots,t_{n-1}\}$ where 
 $ t_i \defequals  (i,i+1) \in S_{n}$ swaps $i$ and $i+1$.

Let $\W_n$ be the group of all permutations of 
$ [\pm n] \defequals  \{-n,\dots,-2,-1,1,2,\dots,n\}$ that commute with the negation map.
Now the set of simple generators is $S = \{t_0,t_1,t_2,\dots,t_{n-1}\}$ where 
\be \label{t-def}
  t_0 \defequals  (-1,1),\quand t_i \defequals  (-i-1,-i)(i,i+1)\text{ for }i \in [n-1].
\ee
With minor abuse of notation, we use 
 $t_i$ for $i>0$ to denote elements of both $S_n$ and $\W_n$.
 
Finally, let $\WD_n\subseteq\W_n$
consist of all elements $\sigma\in\W_n$
that induce an even number of sign changes when applied to the numbers $1,2,3,\dots,n$.
When $n\geq 2$, this Coxeter group's simple generating set is $S = \{t_{-1}, t_1,t_2,\dots,t_{n-1}\}$
where $t_1,t_2,\dots,t_{n-1}$ are as in \eqref{t-def} and \be t_{-1} \defequals  (-2,1)(-1,2) = t_0 t_1 t_0.\ee
Let $\ast$ and $\diamond$ denote the respective automorphisms of $S_{n}$ and $\WD_n$ (for $n\geq 2$) with the formulas
\be
t_i^\ast = t_{n-i}\text{ (for $i \in [n-1]$)}
\qquand t_{1}^\diamond= t_{-1},\quad  t_{-1}^\diamond =t_1, \quad t_i^\diamond = t_i\text{ (for $i>1$)}.
\ee
When $n=1$, we have $\WD_1=\{1\}$, $S=\varnothing$,
and $\diamond=\id$.

We refer to elements of $\W_n$ as \defn{signed permutations}
and to elements of $\WD_n$ as \defn{even-signed permutations}.
Such elements are completely determined by their \defn{one-line representations},
meaning the sequences $w_1w_2\cdots w_n$ with $w_i = w(i)$. We  write $\overline{a}$ to mean the number $-a$,
so that the elements of $\W_2$ are  $12$, $21$, $\overline{1}2$, $\overline{2}1$, $1\overline{2}$, $2\overline{1}$, $\overline{1}\hs\overline{2}$, and $\overline{2}\hs\overline{1}$, for example.
The longest element $w_0\in W$ is 
\be
w_0 = \begin{cases} 
n\cdots 321 &\text{if }W=S_{n}, \\[-14pt]\\
\overline{1}\hs\overline{2}\hs\overline{3}\hs\cdots\hs\overline{n} &\text{if $W=\W_n$ or if $W=\WD_n$ with $n$ even,} \\[-14pt]\\
{1} \overline{2}\hs\overline{3}\hs\cdots\hs\overline{n}
& \text{if $W=\WD_n$ with $n$ odd}.
\end{cases}
\ee

Write $\cI(W) \defequals  \{ z \in W : z=z^{-1}\}$ for the set of involutions in $W$.
Recall from Table~\ref{tbl1} that our choice of $(G,K)$ determines an involution $\Theta \in \Aut(W)$.
Define the set of \defn{twisted involutions} 
$ \cI_\Theta(W)\defequals  \{z\in W : \Theta(z)=z^{-1}\}.$
If $\Theta = \Ad(g) : w\mapsto gwg^{-1}$ for some $g=g^{-1}\in W$, then
\be\label{twist-inv-eq} 
\cI_\Theta(W) = \{ gz :z \in \cI(W)\}= \{zg :z \in \cI(W)\}.\ee
This   applies   when $W=S_n$ and $\Theta = \ast = \Ad(w_0)$.
In type D, we have the similar identity 
\be\label{twist-inv-eq2}
\cI_\diamond(\WD_n) = \Bigl\{ t_0z :z \in \cI(\W_n)\setminus\cI(\WD_n)\Bigr\}= \Bigl\{zt_0 :z \in \cI(\W_n)\setminus\cI(\WD_n)\Bigr\}.
\ee

\subsection{Clans and clan alignment}\label{clans-sect}

Our conventions in this section follow \cite{MO90,McTr,Wyser,Yam97}.
A \defn{partition} of a set $X$ is a set of pairwise disjoint sets $\cP$ with union $X$. 
The elements of $\cP$ are called its \defn{blocks}.
A \defn{matching} is a set partition whose blocks all have size one or two.
A matching is \defn{perfect} if it has no singleton blocks.

\begin{definition}
Choose a finite set $X \subset\ZZ$.
A \defn{clan} is a triple $\gamma = (S_+, S_-, M)$ where $S_+$ and $S_-$ are finite disjoint subsets of $X$
and $M$ is a perfect matching on $X \setminus (S_+\sqcup S_-)$.
We refer to $X$ as the \defn{base set} of $\gamma$ 
and to the difference $|S_+| - |S_-| \in \ZZ$ as the \defn{type} of $\gamma$.
\end{definition}

Suppose $\gamma=(S_+,S_-,M)$ is a clan with base set $X$.
A \defn{one-line representation} of $\gamma$ is any map $c : X \to  \{ \pm \} \sqcup \ZZ $, written $i\mapsto c_i$,
such that $c_i = +$ if and only if $i \in S_+$, $c_i = -$ if and only if $i \in S_-$,
and $c_i = c_j \in \ZZ$ for $i\neq j$ if and only if $\{i,j\}\in M$.
If $X = \{i_1< i_2< \dots<i_n\}$
then we identify $c$ with the sequence $(c_{i_1}, c_{i_2},\dots,c_{i_n})$.
When $X\subseteq \ZZ$ we define 
\be
\APoints(\gamma) = S_+ \sqcup S_-
\quand \Points(\gamma)= \APoints(\gamma) \cap \PP.
\ee

The one-line representation of $\gamma$ is not unique, and to recover $\gamma$ from $c$ 
we must know the base set $X$. However,
distinct clans with the same base set do not share any one-line representations.

\begin{example}
The 14 distinct clans with base set $\{1,2,3\}$ are represented by 
\[ 
\ba
&(1,1,+),&& (1,1,-),&& (1,+,1),&& (1,-,1),&& (+,1,1),&& (-,1,1),&&(+,+,+), \\
&(+,+,-),&&(+,-,+),&&(-,+,+),&&(+,-,-),&&(-,+,-),&&(-,-,+),&&(-,-,-).
\ea
\]
\end{example}

A block of a set partition is \defn{trivial} if it has exactly 
one element. A set partition is \defn{trivial} if it has only trivial blocks. Given a set partition $\Lambda$ and a set $S$ let $\Lambda \cap S = \{ B\cap S : B\in \Lambda\}\setminus \{\varnothing\}$.

\begin{definition}\label{aligned-def}
If $\gamma = (S_+,S_-,M)$ is a clan
then a matching 
 is \defn{$\gamma$-aligned} when it has 
no nontrivial blocks that are subsets of either $S_+$ or $S_-$.
\end{definition}

Equivalently, 
a matching $\Lambda$ is $\gamma$-aligned if $\Lambda \cap S_+$ and $\Lambda\cap S_-$ are both trivial.

\begin{example}
If $\gamma =(S_+,S_-,M)= (+,+,-,-,+,+) $ has base set $\{1,2,3,4,5,6\}$ then
 we have
$S_+ = \{1,2,5,6\}$,
  $S_- =\{3,4\}$, and $M=\varnothing$, so the matchings
$\Bigl\{\{1\}, \{2,3\}, \{4,5\}, \{6\}\Bigr\}
$
and
$
\Bigl\{ \{ 1,3\},\{2\}, \{4,6\}, \{5\}\Bigr\}
$
are $\gamma$-aligned, but $\Bigl\{\{1,6\}, \{2,3\}, \{4,5\}\Bigr\}$ is not.
\end{example}

\subsubsection{Standard clans}

The following kind of clans will index the $K$-orbits in $G/B$ for type AIII.

\begin{definition}
A \defn{standard $(p,q)$-clan} for $p,q\in\NN$ is a clan with base set $[p+q]$ and type $p-q$.
\end{definition}

If $\gamma=(S_+,S_-,M)$ is a standard $(p,q)$-clan then let $\pi_\gamma \in \I(S_{p+q})$
be the involution with $\pi_\gamma(i) =i$ for each $i \in S_+\sqcup S_-$
and with $\pi_\gamma(i)=j$ and $\pi_\gamma(j)=i$ for each $\{i,j\} \in M$.

\begin{example}\label{gamma-ex}
Fix $p,q \in \NN$. Then $p+q = |p-q| +2m$ for some $m \in \NN$ and we define 
$\gammaDense^{(p,q)}$ to be the standard $(p,q)$-clan with one-line representation given by
\be
\label{pq-eq}\ba
(1,2,3,\dots,m,+,+,\dots,+,m,\dots,3,2,1)&\quad\text{if }p\geq q,\text{ or }
\\
(1,2,3,\dots,m,-,-,\dots,-,m,\dots,3,2,1)&\quad\text{if }p\leq q.
\ea
\ee
If $n=p+q$ and $w_0 \in S_{n}$ is the longest element and $\gamma=\gammaDense^{(p,q)}$ then we have
\[ 
 \ba
 \pi_\gamma & = (1,n)(2,n-1)\cdots(m,n-m+1) \in \I(S_{n}), 
\\
w_0\cdot \pi_\gamma &=(m+1,n-m)(m+2,n-m-1)(m+3,n-m-2)\cdots \in \I_\ast(S_{n}).
\ea
\]
\end{example}

\subsubsection{Symmetric and skew-symmetric clans}

If $\gamma = (S_+, S_-, M)$ is a clan with 
one-line representation 
$(c_1,c_2,\dots,c_n)$
then its \defn{reversal} is the clan
with one-line representation
$\gamma^\rev = (c_n,\dots,c_2,c_1).$
The \defn{conjugate} of $\gamma$ is  the clan
$\overline{\gamma} = (S_-, S_+, M).$

\begin{definition}
A \defn{symmetric $(p,q)$-clan} for $p,q\in \NN$ is a clan $\gamma=\gamma^\rev$ of type $p-q$ with base set 
\[
X = \begin{cases} 
[\pm n] & \text{if $p+q=2n$ is even},
\\
[\pm n]\sqcup \{0\} &\text{if $p+q=2n+1$ is odd}.
 \end{cases}
\]
A \defn{skew-symmetric $(n,n)$-clan} is a clan $\gamma=\overline{\gamma^\rev}$ of type $0$ with base set $[\pm n]$.
\end{definition}

Suppose $\gamma$ is a
skew-symmetric $(n,n)$-clan or a
 symmetric $(p,q)$-clan with $p+q \in \{2n,2n+1\}$. Then $\gamma$
may be expressed in one-line notation as 
either
\[ 
(c_{-n},\dots,c_{-2}, c_{-1}, c_1,c_2,\dots,c_n)
\quord
(c_{-n},\dots,c_{-2}, c_{-1}, c_0, c_1,c_2,\dots,c_n).
\]
Define $\sigma_\gamma \in \W_n$ to be the signed involution with $\sigma_\gamma(i) = i$ for each $i \in [\pm n]$
with $c_i \in \{\pm\}$ and with $\sigma_\gamma(i) = j$ and $\sigma_\gamma(j)=i$ whenever $i,j \in [\pm n]$ are such that $c_i=c_j\in \ZZ$.

\begin{example}
Fix $p,q,n \in \NN$ with $p+q \in \{2n,2n+1\}$. Then $p+q = |p-q| +2m$ for some $m \in \NN$ and we define 
$\deltaDense^{(p,q)}$ to be the symmetric $(p,q)$-clan with the same one-line representation as in \eqref{pq-eq}
(but now with a different base set).
Notice that the $(n,n)$-clan 
\[
 \deltaDense^{(n,n)} = (\varnothing,\varnothing, \{\{-i,i\} : i \in [n]\}) = (1,2,\dots,n,n,\dots,2,1)
 \] is both symmetric and skew-symmetric. If $k=n-m$ and $\delta =  \deltaDense^{(p,q)} $ then 
 \[
 \sigma_\delta = 12\cdots k \overline{(k+1)(k+2)\cdots n}\in \W_n
 \quand
 \overline{\sigma_\delta} = \overline{12\cdots k} {(k+1)(k+2)\cdots n}\in \W_n.
 \]
 Here, for any $\sigma\in \W_n$, we let $\overline{\sigma}\in \W_n$ be the element sending $i\mapsto -\sigma(i)$ for $i \in[\pm n]$.
\end{example}

\begin{definition}
A symmetric $(p,q)$-clan $\gamma=(S_+,S_-,M)$ is \defn{strict} if $i+j\neq0$ for all $\{ i,j\} \in M$.
\end{definition}

\begin{example}
Fix even integers $p,q\in \NN$. Then $p+q=|p-q|+2m$ for some $m \in2\NN$, and 
 we define 
$\epsilonDense^{(p,q)}$ to be the (strict) symmetric $(p,q)$-clan with one-line representation given by
\[\ba
(1,2,3,4,5,6,\dots,m-1,m,+,+,\dots,+,m-1,m,\dots,5,6,3,4,1,2)&\quad\text{if }p\geq q,\text{ or}\\
(1,2,3,4,5,6,\dots,m-1,m,-,-,\dots,-,m-1,m,\dots,5,6,3,4,1,2)&\quad\text{if }p\leq q.
\ea
\]
\end{example}

In types AI and AII
it is customary to index the $K$-orbits in $G/B$
by the respective sets of involutions $\I(S_n)$ and $\Ifpf(S_n)= \{ w\in \I(S_n) : w(i)\neq i \text{ for all }i \in [n]\}$  \cite{MO90,RichSpring,Wyser,Yam97}. However, to give uniform statements
we will still define $\Gamma^G_K$  as a set of clans in these cases. Specifically: 
\begin{definition} A symmetric $(n,n)$-clan $\gamma = (S_+,S_-,M)$ is \defn{involutive} 
if $S_+=S_-=\varnothing$
and every block $\{a<b\} \in M$ has $a<0<b$.
\end{definition}

The map $\gamma \mapsto \overline{\sigma_\gamma}$
is clearly a bijection from involutive $(n,n)$-clans to $\I(S_n)$,
if we view $S_n \subset \W_n$ as a subgroup in the obvious way.
Restricting to strict elements gives a bijection to $\Ifpf(S_n)$.

\subsubsection{Clans for type DIII}

In the following definition,
let $\binom{[n]}{2}$ denote the set of 2-element subsets of $[n]$.
\begin{definition}\label{h-def}
A skew-symmetric $(n,n)$-clan  $\gamma=(S_+,S_-,M)$ is \defn{even-strict} if
\[
\text{$i+j\neq0$ for all $\{ i,j\} \in M$}
\quand
\text{the quantity
$ h(\gamma) \defequals  \Bigl|S_+ \cap [n]\Bigr| + \Bigl| M \cap \tbinom{[n]}{2}\Bigr|$
is even.}
\]
\end{definition}

\begin{example}
When $n$ is even, let $\etaDense^{(n,n)}= \epsilonDense^{(n,n)}$ be the skew-symmetric $(n,n)$-clan 
\[ 
(1,2,3,4,5,6,\dots,n-1,n,n-1,n,\dots,5,6,3,4,1,2).
\] 
When $n$ is odd, let $\etaDense^{(n,n)}$ be the skew-symmetric $(n,n)$-clan 
\[ 
(1,2,3,4,5,6,\dots,n-2,n-1,+,-,n-2,n-1,\dots,5,6,3,4,1,2).
\] Then $\etaDense^{(n,n)}$ is even-strict for either parity of $n$. Write $\gamma = \etaDense^{(n,n)}$. Then when  $n$  is even we have 
 \[
 {\sigma_\gamma}=
 \bar 2\hs \bar 1\hs \bar4\hs \bar3\hs \bar6\hs \bar5 \cdots \overline{n} \hs\hs \overline{n-1}
 \qquad\text{so}
 \qquad
  \overline{\sigma_\gamma} = t_1t_3t_5\cdots t_{n-1}
  \in \I(\WD_n)
 \]
 while if $n$ is odd then 
 \[
{\sigma_\gamma}=
 1\bar 3\hs \bar 2\hs \bar5\hs \bar4\hs \bar7\hs \bar6 \cdots \overline{n} \hs\hs \overline{n-1}
  \qquad\text{so}
 \qquad
 t_0 \cdot \overline{\sigma_\gamma} = t_2t_4t_6\cdots t_{n-1} \in \I_\diamond(\WD_n).
 \]
\end{example}

\def\hhline{\\ & &&&  \\ [-\vsp]\hline & & && \\ [-\vsp]}
\def\gap{\\[-\vsp]&&&&\\}
\begin{table}[h]
\begin{center}
{\small
\begin{tabular}{| l | l | l | l | l |}
\hline&&&& \\[-\vsp]
Type & Parameters &  Indexing set $\Gamma^G_K$ for $K$-orbits in $G/B$  & Dense orbit & $\phiRS$ 
\hhline
AI & $n\in \NN$ & $\Gamma_{\AI}^{(n)} \defequals $ involutive $(n,n)$-clans    & $\deltaDense^{(n,n)}$ & $\gamma \mapsto \overline{\sigma_\gamma}$
\gap
AII  & $n\in \NN$ even & $\Gamma_{\AII}^{(n)} \defequals$ strict involutive $(n,n)$-clans  & $\epsilonDense^{(n,n)}$ & $\gamma \mapsto \overline{\sigma_\gamma}$
\gap
AIII & $n=p+q$ & $\Gamma_{\AIII}^{(p,q)}\defequals $  standard $(p,q)$-clans   & $ \gammaDense^{(p,q)}$ & $\gamma \mapsto w_0 \pi_\gamma$
\hhline
BI & $2n+1=p+q$ & $\Gamma_{\BI}^{(p,q)} \defequals $  symmetric $(p,q)$-clans & $ \deltaDense^{(p,q)}$ & $\gamma \mapsto \overline{\sigma_\gamma}$ \hhline
CI & $n\in \NN$ & $\Gamma_{\CI}^{(n)} \defequals  $ skew-symmetric $(n,n)$-clans   & $\deltaDense^{(n,n)}$ & $\gamma \mapsto \overline{\sigma_\gamma}$
\gap
CII  & $2n=p+q$ % with $p$ even 
&  $\Gamma_{\CII}^{(p,q)} \defequals  $ strict symmetric $(p,q)$-clans & $ \epsilonDense^{(p,q)}$ & $\gamma \mapsto \overline{\sigma_\gamma}$ 
%\\   & with $p$, $q$ even &  & &
\hhline
DI  & $2n=p+q$ %with $p+n$ even 
&  $\Gamma_{\DI}^{(p,q)} \defequals  $  symmetric $(p,q)$-clans & $ \deltaDense^{(p,q)}$ & $\gamma \mapsto 
(t_0)^k  \overline{\sigma_\gamma}$ 
%\\  & with $k=\frac{|p-q|}{2}$  & & &
\gap
DIII & $n\in \NN$ & $\Gamma_{\DIII}^{(n)} \defequals $ even-strict skew-symmetric $(n,n)$-clans & $\etaDense^{(n,n)}$ & $\gamma \mapsto  (t_0)^n \overline{\sigma_\gamma}$
\gap\hline
\end{tabular}}
\end{center}
\caption{
Indexing sets for $K$-orbits in $G/B$ following \cite{Wyser}.
The last two columns provide the index for the unique dense $K$-orbit and the
formula for the Richardson--Springer map.
In type CII both $p$ and $q$ must be even. In type DI the value of $k$ in the last column is $k= \frac{|p-q|}{2}$. 
}\label{tbl2}
\end{table}

\subsection{Orbit parametrizations}

The definitions in the preceding sections are related
to the $K$-orbits in $G/B$ in the following way.

\begin{proposition}[\cite{MO90,Yam97}]
For each pair $(G,K)$ of classical type listed in Table~\ref{tbl1}, the
finite set of $K$-orbits in $G/B$ is in bijection with the corresponding indexing set
 $\Gamma^G_K $  specified in Table~\ref{tbl2}. 
\end{proposition}

\begin{remark}
Rather than writing $\Gamma^{\GL(n)}_{\GL(p)\times \GL(q)}$ in type AIII
or $\Gamma^{\Sp(2n)}_{\GL(n)}$ in type CI,
we use the abbreviations $\Gamma_{\AIII}^{(p,q)}$ and $\Gamma_{\CI}^{(n)}$.
We extend this convention to all classical types in the obvious way, as well as to various other notations depending on $(G,K)$  such as
$\phiRS$, $\cI^G_K$, $\cABrion(\gamma)$, and $\cEABrion(z)$.
\end{remark}

The references \cite{MO90,Yam97} give an explicit way of labeling each $K$-orbit 
in $G/B$ in classical type by an index $\gamma \in \Gamma^G_K$.   
We define $\cO_\gamma \subseteq G/B$ to be the $K$-orbit associated to $\gamma \in \Gamma^G_K$
by this \defn{standard parametrization}.

Our goal now is to translate all relevant information about the weak order graph on $K$-orbits into 
statements about the indexing sets $ \Gamma^G_K$.
Starting this process,  Table~\ref{tbl2} provides the index of the unique dense $K$-orbit
in each classical type, along with an explicit formula for the Richardson--Springer map $
\phiRS : \Gamma^G_K \to W 
$ from Definition~\ref{rs-def}.

The image of $
\phiRS 
$
 is always a  set of twisted involutions $\cI^G_K \subseteq \cI_\Theta(W)$.
We can describe this set  more explicitly using the following 
 permutation statistics.
If $w$ is a permutation in $S_n$ or signed permutation in $\W_n$ then let $\Fix(w) \defequals  \{ i \in [n] : w(i)=i\}$ denote its  set of positive fixed points, and set $\fix(w)=|\Fix(w)|$.
If $w \in \W_n$ then define 
$\ell_0(w) \defequals  |\{ i \in [n] : w(i) < 0\}|$ along with
\be
 \Neg(w) \defequals  \{i \in [n] : w(i)=-i\}
 \quand
  \neg(w) \defequals  |\Neg(w)|.
\ee
Denote the subsets of \defn{fixed-point-free} involutions in each Weyl group $W \in \{S_n,\W_n,\WD_n\}$ 
by \[\Ifpf(W) = \{w =w^{-1} \in W : \fix(w)  = 0\}.\]

It will   be useful to identify the image of the index $\gamma_{\dense}$ of the unique dense $K$-orbit in $G/B$
under $\phiRS$. Fix integers $0\leq k\leq n$. When $n-k$ is even, let
$i = \frac{n-k}{2}$, $j=\frac{n+k}{2}+1$, and  
 \be\label{omega-eq}
 \omega_k^n \defequals  (i+1,j-1)(i+2,j-2)(i+3,j-3)\cdots (\lfloor\tfrac{n}{2}\rfloor,\lceil\tfrac{n}{2}\rceil+1) \in S_n.\ee
  Notice that $\omega_k^n=1 $ if $k \leq 1$ and  $\omega_k^n=w_0$ if $k=n$.
 Next, let 
 \be\label{sigmakn-eq}
 \sigma_k^n \defequals  \overline{1}\hs\overline{2}\cdots \overline{k}(k+1)(k+2)\cdots n \in \W_n
 \quand
 \hat \sigma_k^n \defequals (t_0)^k \sigma_k^n\in \WD_n.
\ee
Finally, when $n-k$ is even, define 
\be\label{sigmakn-fpf-eq}
\sigma_{k\times \fpf}^n \defequals \sigma_k^n\cdot t_{k+1}\cdot t_{k+3}\cdot t_{k+5}\cdots t_{n-1}
=
\overline{1}\hs\overline{2}\cdots \overline{k}(k+2)(k+1)\cdots n(n-1) \in \I_\fpf(\W_n).\ee

\begin{proposition}
For each pair $(G,K)$ of classical type listed in Table~\ref{tbl1}, the corresponding subset $\cI^G_K \subseteq \I_\Theta(W)$
and the element $\phiRS(\gamma_{\dense}) \in \cI^G_K$
are as
 specified in Table~\ref{rs-image-tbl}.
\end{proposition}

\begin{proof}
It is easy to compute $\phiRS(\gamma_{\dense})$ from the information in Table~\ref{tbl2}.
Our formulas for $\cI^G_K$ in types A, B, and C likewise follow directly from 
the descriptions of $\Gamma^G_K$ and $\phiRS$ in Table~\ref{tbl2}.
What we wish to show is also straightforward in type DI,
using \eqref{twist-inv-eq2} and the fact
that if $2n=p+q$
 then $\frac{|p-q|}{2} = \frac{|p - (2n-p)|}{2} = |p-n|$ has the same parity as $p+n$.
 
Type DIII requires more explanation.
 Suppose $\gamma=(S_+,S_-,M)$ is a strict skew-symmetric $(n,n)$-clan.
 Then $\overline{\sigma_\gamma}$ has 
 no fixed points since $i+j\neq 0$ for all $\{i,j\} \in M$.
If the set $S=S_+ \sqcup S_-$ is nonempty and $x = \max(S)$
then let $
\tilde \gamma \defequals (S_+ \ominus \{x,-x\},   S_- \ominus \{x,-x\}, M)$
where $\ominus$ denotes the symmetric set difference,
and if $S = \varnothing$ then define $\tilde\gamma = \gamma$. 

We have $\gamma =\tilde\gamma$ if and only if the signed permutation $\sigma_\gamma$ has no fixed points, which is equivalent to $z=\overline{\sigma_\gamma}$ having no negated points.
In this case $h(\gamma) = |M\cap \binom{[n]}{2}|$ is equal to $\frac{1}{2} \ell_0(z)$.
Hence, if $n$ is even and  $z \in \Ifpf(\WD_n)$ has $\neg(z)=0$ then 
$z \in \cI_{\DIII}^{(n)}$ if and only if $ \ell_0(z)$ is divisible by four.

On the other hand, if $z=\overline{\sigma_\gamma} $
has at least one negated point (which must occur if $n$ is odd) then exactly one of
$\gamma $ or $ \tilde\gamma$ is even-strict skew-symmetric since  $ h(\gamma) = h(\tilde\gamma) \pm 1$.
As $\sigma_\gamma = \sigma_{\tilde \gamma}$
it follows when $n$ is even that any  $z \in \Ifpf(\WD_n)$ with $\neg(z)>0$ is in $\cI_{\DIII}^{(n)}$.

This discussion characterizes $\cI_{\DIII}^{(n)}$ when $n$ is even, and shows that the set is equal to the expression in  Table~\ref{rs-image-tbl}.
When $n$ is odd, every  $z \in \Ifpf(\W_n)$ has $\neg(z)>0$ and  $\ell_0(z)\notin 2\ZZ$, so occurs as 
$\overline{\sigma_\gamma}$ for some even-strict skew-symmetric $(n,n)$-clan $\gamma$.
Therefore the elements 
$t_0\cdot \overline{\sigma_\gamma}$ that make up  $\cI_{\DIII}^{(n)}$ when $n$ is odd
consist of all even-signed permutations $ z\in  \WD_n$ with $ t_0z \in \Ifpf(\W_n)$.
\end{proof}

\def\hhline{\\ & & &  \\ [-\vsp]\hline & & & \\ [-\vsp]}
\def\gap{\\[-\vsp]&&&\\}
\begin{table}[h]
\begin{center}
{\small
\begin{tabular}{| l | l | l | l |}
\hline&&& \\[-\vsp]
Type & Parameters & Image $\cI^G_K$  of $\RSphi : \Gamma^G_K \to \I_\Theta(W)$ & $\phiRS(\gamma_{\dense}) \in \cI^G_K$ 
\hhline
AI & $n\in \NN$ & $\cI_{\AI}^{(n)} \defequals  \cI(S_{n}) $   & $1$  
\gap
AII  & $n\in \NN$ even & $\cI_{\AII}^{(n)} \defequals  \Ifpf(S_{n})$ & $t_1t_3t_5\cdots t_{n-1}$ 
\gap
AIII & $n=p+q$ with  $k= |p-q| $ & $\cI_{\AIII}^{(p,q)}\defequals \left\{z\in \cI_\ast(S_{n}) : \fix(w_0z) \geq k\right\}$   & $\omega_k^{n}$ \hhline
BI & $2n+1=p+q$ with $k=\frac{|p-q|-1}{2}$ & $\cI_{\BI}^{(p,q)} \defequals  \left\{ z \in\cI(\W_n) : \neg(z)\geq k\right\}$  %$n=\frac{p+q-1}{2}$ and 
& $\sigma^n_k$
\hhline
CI & $n\in \NN$ & $\cI_{\CI}^{(n)} \defequals  \cI(\W_n)$ & $1$    
\gap
CII  & $2n=p+q$ with $k=\frac{|p-q|}{2}$ &  $\cI_{\CII}^{(p,q)} \defequals  \left\{ z \in \Ifpf(\W_n): \neg(z)\geq k \right\}$   
  & $\sigma_{k\times \fpf}^n$ 
%  \\& with $p$ even & &
\hhline
DI  & $2n=p+q$ with $k=\frac{|p-q|}{2}\in2\ZZ$ &  $\cI_{\DI}^{(p,q)} \defequals  \left\{ z \in \cI(\WD_n): \neg(z)\geq k \right\}$ 
 &  $\hat \sigma_k^n =
  \sigma_k^n$   
%\\ & with $p+n$ even & &
\gap
  & $2n=p+q$ with $k=\frac{|p-q|}{2}\notin2\ZZ$ & $\cI_{\DI}^{(p,q)} \defequals   \left\{ z \in \cI_\diamond(\WD_n) : \neg(t_0z)\geq k \right\}$
& $\hat \sigma_k^n = 
t_0  \sigma_k^n$ 
%\\ & with $p+n$ odd & &
\gap
DIII & $n\in \NN$ even & $\cI_{\DIII}^{(n)} \defequals 
 \left\{ z \in \Ifpf(\WD_n) : \neg(z)>0 \text{ or }  \ell_0(z) \in 4\NN\right\}$  & $t_1t_3t_5\cdots t_{n-1}$
 \gap
   & $n\in \NN$ odd & $\cI_{\DIII}^{(n)} \defequals 
 \left\{ z\in  \WD_n : t_0z \in \Ifpf(\W_n)\right\}$ & $t_2t_4t_6\cdots t_{n-1}$
\gap\hline
\end{tabular}}
\end{center}
\caption{Images of the Richardson--Springer map for classical groups of rank $n$
as parametrized in Table~\ref{tbl1}. In the last column $\gamma_{\dense} \in \Gamma^G_K$ denotes the index
of the unique dense $K$-orbit in $G/B$.
}\label{rs-image-tbl}
\end{table}

 \section{Weak order graphs}\label{order-sect}
 
 We maintain the conventions of the previous section.
 Here, our main goal is to present a characterization of the weak order graph
 that is uniform across all classical types.

\subsection{Demazure products}\label{dem-sect}

There is a unique associative operation $\circ : W \times W \to W$ called the \defn{Demazure product}
satisfying
\be
w\circ s =\begin{cases} ws&\text{if }\ell(ws) =\ell(w)+1 \\
w&\text{if }\ell(ws) =\ell(w)-1
\end{cases}
\quand
s\circ w =\begin{cases} sw&\text{if }\ell(sw) =\ell(w)+1 \\
w&\text{if }\ell(sw) =\ell(w)-1
\end{cases}
\ee
for all $s \in S$ and $w \in W$, where $\ell : W \to \NN$ is the Coxeter length function \cite[\S7.1]{Humphreys}.
The weak order graph on $\Gamma^G_K$ is closely related to this operation, and so we review some of its properties here.

If $s_1,s_2,\dots,s_k \in S$ then we have
$
\ell(s_1\circ s_2 \circ \cdots \circ s_k) \leq k
$
with equality if and only if $s_1s_2\cdots s_k$ is a reduced expression.
This means that for any $v,w \in W$ we have
$\ell(v\circ w) \leq \ell(v) + \ell(w),$ 
and if equality holds in this identity then $v\circ w=vw$ while otherwise
there exists a shorter element $u \in W$ with $\ell(u)<\ell(v)$ and $v\circ w = u\circ w$.
 We also have
\be
(v\circ w)^{-1} = w^{-1} \circ v^{-1}\quand \Theta(v\circ w) = \Theta(v) \circ \Theta(w) \quad\text{for all $v,w\in W$.}
\ee
Finally, if $s \in S$ and $z \in \I_\Theta(W)$ then  
  the exchange condition  (see \cite[Lem.~3.4]{H2}) implies that
\be\label{szs-eq} \Theta(s)\circ z \circ s = \begin{cases}
\Theta(s) z s &\text{if }\ell(zs) > \ell(z)\text{ and }\Theta(s)z \neq zs \\
zs  &\text{if }\ell(zs) > \ell(z)\text{ and } \Theta(s)z= zs \\
z  &\text{if }\ell(zs) < \ell(z).
\end{cases}
\ee
We will need a more explicit statement of this formula 
for each classical Weyl group.
 
\subsubsection{Demazure conjugation in type A}

Fix $i \in [n-1]$ and let $\ell $ be the Coxeter length function $S_{n} \to \NN$. Then
$
 \ell(w)  \defequals \inv(w)
$ where
\be
 \inv(w) \defequals |\{ (i,j) \in [n]\times [n] : i<j\text{ and }w(i)>w(j)\}|.
 \ee 
Thus if $w \in S_{n}$  then $\ell(wt_i) >\ell(w)$ if and only if $w(i)<w(i+1)$.
If $z \in \I(S_{n})$ then 
\be
t_i \circ z \circ t_i = \begin{cases}
z \cdot t_i &\text{if $z(i)=i$ and $z(i+1)=i+1$} \\
z &\text{if $z(i)>z(i+1)$} \\
t_i \cdot z\cdot t_i &\text{otherwise}.
\end{cases}
\ee
If $w_0 \in S_{n}$ is the longest element
and $z \in \I_\ast(S_{n})$ then 
\be
t_i^\ast \circ z \circ t_i = \begin{cases}
z\cdot t_i &\text{if $z(i)=w_0(i+1)$ and $z(i+1)=w_0(i)$} \\
z &\text{if $z(i)>z(i+1)$} \\
t_{n-i} \cdot z\cdot  t_i &\text{otherwise}.
\end{cases}
\ee

\subsubsection{Demazure conjugation in types B and C}

Fix $i \in [n-1]$ and let $\ell $ be the Coxeter length function of $\W_{n}$. 
Then $
 \ell(w)  \defequals \tfrac{\inv^\pm(w) + \ell_0(w)}{2}
$ 
where
\be
\ba
 \inv^\pm(w) &\defequals |\{ (i,j) \in [\pm n]\times [\pm n] : i<j\text{ and }w(i)>w(j)\}|\quand
 \\
 \ell_0(w) &\defequals |\{ i \in [n] : w(i) < 0\}|.
 \ea
 \ee
Thus if $w \in \W_n$ then 
$
\ell(wt_0) > \ell(w)\text{ $\Leftrightarrow$ }w(1) >0
$
and
$
\ell(wt_i) >\ell(w)\text{ $\Leftrightarrow$ }w(i)<w(i+1)
.$
If  $z \in \I(\W_n)$ and we set $z(0)=0$ then the next formula  holds for all $i \in \{0,1,2,\dots,n-1\}$: 
\be\label{BC-dem-eq}
t_i \circ z \circ t_i = \begin{cases}
z \cdot t_i &\text{if $z(i)=i$ and $z(i+1)=i+1$, or} \\
 &\text{if $z(i)=-i-1$ and $z(i+1)=-i$} \\
z &\text{if $z(i)>z(i+1)$} \\
t_i \cdot z\cdot t_i &\text{otherwise}.
\end{cases}
\ee

\subsubsection{Demazure conjugation in type D}\label{d-conj-sect}

The length function $\ell : \WD_{n} \to \NN$
 has the formula $ \ell(w)  \defequals \tfrac{\inv^\pm(w) - \ell_0(w)}{2}$,
 and if $i \in [n-1]$ then
\be\label{d-descent-eq}
\ell(wt_{-1}) > \ell(w)\text{ $\Leftrightarrow$ }{-w(2)}<w(1)
\quand
\ell(wt_i) >\ell(w)\text{ $\Leftrightarrow$ }w(i)<w(i+1)
.\ee
Let $[a,b] = \{ i \in \ZZ: a\leq i \leq b\}$.
Then the right hand side of \eqref{BC-dem-eq} gives the formula for both 
\[
t_i \circ z \circ t_i \text{ when }(i,z) \in [n-1]\times \I(\WD_n)
\quand t_i^\diamond \circ z \circ t_i\text{ when }(i,z) \in [2,n-1]\times  \I_\diamond(\WD_n).
\]
When $z \in \I(\WD_n)$, this fact lets us compute 
\be
t_{-1} \circ z \circ t_{-1} = (t_1 \circ z^\diamond \circ t_1)^\diamond = \begin{cases}
z \cdot t_0 \cdot t_1 \cdot t_0 &\text{if $z(1)=1$ and $z(2)=2$, or} \\
 &\text{if $z(1)=2$ and $z(2)=1$} \\
z \cdot t_0\cdot t_1 \cdot t_0\cdot t_1 &\text{if $z(1)=-1$ and $z(2)=2$} \\
z &\text{if $z(1)+z(2)<0$} \\
t_0\cdot t_1 \cdot t_0\cdot z\cdot t_0\cdot t_1 \cdot t_0 &\text{otherwise}.
\end{cases}
\ee
Similarly, if $z \in \I_\diamond(\WD_n)$ then  
\be
t_{1}^\diamond \circ z \circ t_{1} = \begin{cases}
 z \cdot t_1 &\text{if $z(1)=-1$ and $z(2)=2$, or} \\
 &\text{if $z(1)=-2$ and $z(2)=1$} \\
z \cdot t_0\cdot t_1 \cdot t_0\cdot t_1 &\text{if $z(1)=1$ and $z(2)=2$} \\
z &\text{if $z(1)>z(2)$} \\
t_0\cdot t_1 \cdot t_0\cdot z\cdot t_1 &\text{otherwise},
\end{cases}
\ee
and since $t_{-1}^\diamond \circ z \circ t_{-1}  = (t_{1}^\diamond \circ z^\diamond \circ t_{1} )^\diamond$ we likewise have
\be
t_{-1}^\diamond \circ z \circ t_{-1} = \begin{cases}
 z \cdot t_0\cdot t_1 \cdot t_0 &\text{if $z(1)=-1$ and $z(2)=2$, or} \\
 &\text{if $z(1)=2$ and $z(2)=-1$} \\
z \cdot t_0\cdot t_1 \cdot t_0\cdot t_1 &\text{if $z(1)=1$ and $z(2)=2$} \\
z &\text{if $z(1)+z(2)<0$} \\
t_1\cdot z\cdot t_0\cdot t_1 \cdot t_0 &\text{otherwise}.
\end{cases}
\ee

\subsection{Weak order}\label{weak-order-sect}

Suppose $\gamma=(S_+,S_-,M)$ and $\delta=(T_+,T_-,N)$ are clans with the same base set $X$.
We say that the two clans are
 \defn{sign-equivalent} and write $\gamma\sim\delta$
if there is an order-preserving bijection
\[
f : S_+ \sqcup S_- \to T_+\sqcup T_-
\quad\text{such that $f(S_+) = T_+$ and $f(S_-) = T_-$.}
\]
We say that $\gamma $ \defn{is contained in} $\delta$ and write $\gamma\subseteq \delta$
if $S_+\subseteq T_+$ and $S_- \subseteq T_-$.
For any set 
 $Y\subseteq X$ let 
\be
\gamma \ominus Y = ((S_+\setminus Y) \sqcup (S_-\cap Y), (S_-\setminus Y) \sqcup (S_+\cap Y),M).
\ee
In words, this is the clan obtained from $\gamma$ by toggling the signs on all isolated points in $Y$.
Write 
\[S_+(\gamma) = S_+,\quad S_-(\gamma) =S_-, \quand M(\gamma)=M.\]

When $w$ is an element of a classical Weyl group $W$, let  $C(w)$ be its set of nontrivial cycles,
formally defined as the family of sets with at least two elements that are single orbits under the natural action of $w$.
For example, we have 
\[ C(s)  =
\begin{cases}
\{ \{ i,i+1\}\} &\text{if $s=t_i \in S_{n}$ for $i \in [n-1]$} \\
 \{ \{i,i+1\}, \{-i-1,-i\}\} &\text{if $s=t_i \in \W_n$ for $i\in [n-1]$} \\
\{ \{-1,1\}\} & \text{if $s=t_0 \in \W_n$} \\
 \{ \{-2,1\}, \{-1,2\}\} &\text{if $s=t_{-1} \in \WD_n$}
 \end{cases}
 \]
 and $C(\overline 3 1 \overline 5\overline 4 2) = \{ \{ 1,-3,5,2\}, \{-1,3,-5,-2\}, \{-4,4\}\}$.
 We now have this uniform combinatorial description of the weak order graph for each classical symmetric pair $(G,K)$:

\begin{theorem}\label{weak-order-thm}
Assume $(G,K)$ is of classical type.
Let the $K$-orbit closures in $G/B$ have the standard parametrization by 
the sets  $\Gamma^G_K$  
in Table~\ref{tbl2}. Then 
\[\gamma\xrightarrow{s}\delta
\qquad\text{for a given choice of $\gamma,\delta \in \Gamma^G_K$ and $s\in S\subset W$}\]
 is an edge in the corresponding weak order graph if and only if 
\ben
\item[(a)] one has $\phiRS(\gamma)=\Theta(s) \circ  \phiRS(\delta) \circ s \neq  \phiRS(\delta)$;

\item[(b)] if $\cycles(s) \not\subseteq M(\delta)$ then $s$ restricts to 
bijections 
\[S_-(\gamma)\leftrightarrow S_-(\delta)\quand S_+(\gamma)\leftrightarrow S_+(\delta),\quad\text{where we set $s(0)=0$ in type BI};
\]
\item[(c)] if $\cycles(s) \subseteq M(\delta)$ then all of the following hold:
\bei
\item we have $M(\gamma) = M(\delta)\setminus \cycles(s)$;
\item if $s=t_0$ in type BI then
$\delta\ominus\{0\}\subseteq\gamma$ and in all other cases $\delta\subseteq\gamma$;
and
\item when $s=t_i$
neither $S_+(\gamma)$ nor $S_-(\gamma)$ contains both $i$ and $|i|+1$.
\eei
\een
Define  $z =  \phiRS(\delta)$.
When present in the weak order graph, the edge $\gamma\xrightarrow{s}\delta$ is doubled  if and only if
\bei
\item we have $\{i,i+1\}\subseteq  \Fix(z)$ when $s=t_i$ for $i\neq 0$ in types AI, BI, or CI; 

\item we have $\{|i|,|i|+1\}\subseteq \Fix(\varepsilon z)$ when $s=t_i$ for $i\neq 0$ in type DI
and  $\varepsilon=(t_0)^{|p-q|/{2}}$;   
or

\item we have $1\in  \Fix(z)$ when $s=t_0$ in type BI.

\eei
\end{theorem}

Notice that
if $\gamma$ and $\delta$ are involutive $(n,n)$-clans in type AI or AII then
 only condition (a) is relevant
 since   (c) never applies and  (b) holds vacuously.
 
\begin{proof}
Wyser's thesis
\cite[\S2-\S5]{Wyser} provides a comprehensive description of the $K$-orbit parametrization and
weak order graph in all classical types  A, B, C, and D.
Apart from our superficial reindexing convention in types AI and AII, 
we use the same orbit parametrization, and the theorem follows by carefully
comparing Wyser's description of the weak order with the formulas for $\phiRS$ and
 Demazure conjugation in the previous section.
\end{proof}

\begin{example}\label{weak-ex-bc}
Suppose $\gamma \xrightarrow{t_0} \delta$ is an edge in the weak order graph when $ W=\W_n$.
\ben
\item[(1)] Suppose $C(t_0) = \{(-1,1)\} $ is contained in $ M(\delta)$ in type BI. Then the conditions
\[
M(\gamma) = M(\delta)\setminus \cycles(t_0),\quad
\delta\ominus\{0\}\subseteq \gamma, \quand |S_+(\gamma)\cap \{0,1\}|= |S_-(\gamma)\cap \{0,1\}|=1
\]
mean that if we write $\delta = (c_{-n},\dots,c_{-2},c_{-1},c_0,c_1,c_2,\dots,c_n)$ in one-line notation,
then 
\[(c_{-1},c_0,c_1) = (a,\pm,a)\quad\text{for some $a \in \ZZ$}\]
and 
$\gamma=(c_{-n},\dots,c_{-2},\pm,\mp,\pm,c_2,\dots,c_n).$
In this case, the edge is doubled.

\item[(2)] Suppose $(G,K)$ is of type CI and $C(t_0)   \subseteq M(\delta)$. Then the base set $[\pm n]$ of $\gamma$ and $\delta$ does not contain $0$,
so 
we are only required to have
$
M(\gamma) = M(\delta)\setminus \cycles(t_0)
$
and
$ 
\delta \subseteq \gamma
$.
Hence, if 
\[\delta = (c_{-n},\dots,c_{-2},c_{-1},c_1,c_2,\dots,c_n)\] is a one-line representation,
then 
$c_{-1}=c_1 \in \ZZ$   
and 
$\gamma= (c_{-n},\dots,c_{-2},\pm,\mp,c_2,\dots,c_n).$
In this case, the edge is not doubled.

\item[(3)] In type CII we never have $C(t_0)   \subseteq M(\delta)$ since every clan in $\Gamma_{\CII}^{(p,q)}$ is strict.
\een
\end{example}

\begin{example}
Suppose $\gamma \xrightarrow{t_{-1}} \delta$ is an edge in the weak order graph when $ W=\WD_n$
and   $(G,K)$ is of type DI  or DIII.
If $C(t_{-1}) = \{(-1,2),(-2,1)\} $ is contained in $M(\delta)$ 
then the conditions
\[
M(\gamma) = M(\delta)\setminus \cycles(t_{-1}),\quad
\delta \subseteq \gamma, \quand |S_+(\gamma)\cap \{-1,2\}|= |S_-(\gamma)\cap \{-1,2\}|=1
\]
mean that if we write $\delta = (c_{-n},\dots,c_{-2},c_{-1},c_1,c_2,\dots,c_n)$ in one-line notation, then 
\[ (c_{-2},c_{-1},c_1,c_2) = (a,b,a,b)\quad\text{for some }a,b \in \ZZ\]
and  a one-line representation of the (symmetric or even-strict skew-symmetric) clan $\gamma$ is  
\[
\begin{cases}
(c_{-n},\dots,c_{-3},\pm,\mp,\mp,\pm,c_3,\dots,c_n) &\text{in type DI} \\
(c_{-n},\dots,c_{-3},\mp,\mp,\pm,\pm,c_3,\dots,c_n) &\text{in type DIII}.
\end{cases}
\]
A weak order edge $\gamma \xrightarrow{t_{-1}} \delta$ is never doubled if $C(t_{-1})\subseteq M(\delta)$
since then $\{1,2\}\notin M(\delta)$. However, the weak order graph in type DI 
may contain doubled edges of the form 
\[(c_{-n},\dots,c_{-3}, a,b,b,a,c_3,\cdots,c_n) \xrightarrow{t_{-1}} (c_{-n},\dots,c_{-3}, a,a,b,b,c_3,\dots,c_n)\quad\text{when $a,b \in \ZZ$}.\]
\end{example}

Assume $(G,K)$ is classical and fix $z \in \cI^G_K$.
Then there is a matching
 and a set such that \be
\ccM( z) = M(\gamma)
\quand
\cS(z) = S_+(\gamma) \sqcup S_-(\gamma)
\quad
\text{for all $\gamma \in \Gamma^G_K$ with $z= \phiRS(\gamma)$.}
\ee
Explicitly, we see from Table~\ref{tbl2} that 
\[ \ccM(z) = \begin{cases}
C(w_0 z) &\text{in type AIII} \\
C(\overline {\varepsilon z}) &\text{in all other types}
\end{cases}
\]
along with 
\[
\cS(z) = 
\begin{cases}
\varnothing & \text{in types AI and AII} \\
\{ i \in [n] :   z(i) = n+1-i\} &\text{in type AIII} \\
\{ i \in [\pm n] : z(i) =-i\}\sqcup\{0\} &\text{in type BI} \\
\{ i \in [\pm n] : \varepsilon z(i) =-i\} &\text{in types CI, CII, DI, and DIII},
\end{cases}
\]
where we
define $\varepsilon$ to be $ (t_0)^{|p-q|/2}$ in type DI,
$ (t_0)^n$ in type DIII, and $1$ in all other cases.
Recall that a \defn{set partition} is a set of nonempty, pairwise disjoint sets (called \defn{blocks}).
It follows that \be
\label{sp-obs}
\ccM(z) \sqcup \{ \{ a\} : a \in \cS(z)\}\ee is a set partition of $[n]$ in type AIII, of $ [\pm n]\sqcup \{0\}$ in type BI,  and of $[\pm n]$ in all other types.
In the following lemma
let $\supp(w) = \{ i : w(i) \neq i \}$ for $w \in W$. 
%Continue to assume $(G,K)$ is not of type AI or AII.

\begin{lemma}\label{MS-lem}
Suppose $z \in \cI^G_K$ and $s \in S\subset W$. 
Define $\tau = \Theta(s) \circ z \circ s \in \I_\Theta(W)$.
Then:
\ben

\item[(a)] If $C(s) \subseteq \ccM(z)$ then we have $\tau = zs=\Theta(s)z$ along with  
\[\ccM(\tau) = \ccM(z) \setminus C(s)\quand \cS(z) \sqcup \supp(s) = \cS(\tau).\]

\item[(b)] If $C(s) \not\subseteq \ccM(z)$ then $s$ restricts to a bijection  
$\cS(z) \leftrightarrow \cS(\tau)$.

\een
\end{lemma}

\begin{proof}
When $C(s) \subseteq \ccM(z)$,
it follows using Tables~\ref{tbl1} and \ref{tbl2} that
 $zs = \Theta(s) z$.
The descent conditions in Section~\ref{dem-sect} show that $\ell(z) < \ell(zs)$ so
we conclude from \eqref{szs-eq} that $\tau =  zs = \Theta(s) z$.

When $C(s) \subseteq \ccM(z)$ it is clear by definition that $\ccM(zs) = \ccM(z) \setminus C(s)$ 
and so our observation about \eqref{sp-obs} being a set partition implies that we have $ \cS(z) \sqcup \supp(s) = \cS(zs)$.
This proves part (a).
Part (b) can be shown by a case-by-case calculation using the formulas  in Section~\ref{dem-sect}.
\end{proof}

Continue to assume $(G,K)$ is not of type AI or AII.
For $v,w \in W$ and $s \in S$ we write 
\be\label{wsv-eq}
w \xrightarrow{s} v\quad\text{if}\quad w= \Theta(s) \circ v\circ s\neq v .
\ee
Now write $z_{\dense} = \phiRS(\gammaDense)$ and consider a sequence 
\be\label{path-eq}
P = \( z =z^0 \xrightarrow{ s_1} z^1 \xrightarrow{ s_2}  \cdots \xrightarrow{ s_m} z^m=z_{\dense} \)
\quad\text{where each $s_i \in S$.}
\ee
The associativity of the Demazure product implies that $s_1s_2\cdots s_m$ is a reduced expression.
We wish to determine when $P$ \defn{lifts} to a path in weak order for $(G,K)$ in the sense that 
there are clans $\gamma^0,\gamma^1,\dots,\gamma^m \in \Gamma^G_K$ with $\phiRS(\gamma^i) = z^i$  
such that 
$
\gamma^0 \xrightarrow{ s_1} \gamma^1 \xrightarrow{ s_2}  \cdots \xrightarrow{ s_m} \gamma^m=\gammaDense
$
are all edges in the weak order graph.

To this end, we associate a set partition to $P$ in the following way.
First, let 
\be
w=s_1s_2\cdots s_m \in W
\quand
 \Lambda_{m+1}(P) = \left\{ \left\{   w(a) \right\} :  a \in \cS(z_\dense) \right\}
 \ee
with the convention that $w(0) = 0$.
Next, for each $i \in [m]$  let 
$v=s_1s_2\cdots s_{i-1}$ and define 
\be\label{Lambda_i-eq}
\Lambda_i(P)
=\begin{cases}
\left\{ \left\{v(a), v(b)\right\} : \{a,b\} \in C (s_i)\right\} &\text{if $C(s_i) \subseteq \ccM(z^{i})$}\\
\varnothing &\text{if $C(s_i) \not\subseteq \ccM(z^{i})$}.
\end{cases}
\ee
The following property is essentially immediate from Lemma~\ref{MS-lem} by induction on $m$.

\begin{proposition}\label{Lambda-prop}
The union $ \bigcup_{i=1}^{m+1} \Lambda_i(P)$
 is disjoint and gives a set partition of $\cS(z)$.
\end{proposition}
 
Outside type BI we omit the blocks from $\Lambda_{m+1}(P)$ to define
\be
\Lambda(P) =  \Lambda_1(P) \sqcup \Lambda_2(P) \sqcup  \cdots \sqcup  \Lambda_m(P).
\ee
In type BI with parameters $p+q=2n+1$, if this union contains $l$ symmetric blocks 
\be\label{BI-a-eq}
\{-a_1,a_1\}, \{-a_2,a_2\}, \dots, \{-a_l,a_l\}
\quad\text{with $0<a_1<a_2<\dots<a_l$,}\ee then we 
form $\Lambda(P) $ from $ \Lambda_1(P) \sqcup\Lambda_2(P) \sqcup  \cdots \sqcup  \Lambda_m(P)$
by replacing all of these blocks with one set
\[
\{-a_{l},\dots,-a_2,-a_1,0,a_1,a_2,\dots,a_{l}\}.
\]

\begin{example}\label{lift-ex}
In type BI with $(p,q)= (5,4)$ or in type CI with $n=4$, a valid choice of $P$ is
\[
P= \(\ \ba 
z= \overline{1234} &\xrightarrow{t_0}1\overline{234} 
 \xrightarrow{t_1} \overline{1}2\overline{34} 
  \xrightarrow{t_2} \overline{12}3\overline{4}
   \xrightarrow{t_3} \overline{123}4 
   \xrightarrow{t_1} \overline{213}4
   \\& \xrightarrow{t_0} 21\overline{3}4
    \xrightarrow{t_1}  12\overline{3}4 
    \xrightarrow{t_2}1\overline{2}34 
    \xrightarrow{t_1}\overline{1}234 
\xrightarrow{t_0}   1234 =z_\dense
\ea\ \),
\]  
which has length $m=10$. Then $\Lambda_{11}(P)=\{\{0\}\}$ in type BI and $\Lambda_{11}(P) =\varnothing$ in type CI, while
\bei
\item $\Lambda_{10}(P) = \{ \{-4,4\}\}$ since 
\[C(t_0) =\{\{-1,1\}\}\subseteq \ccM(1234) =C(\overline{1234})= \{ \{-1,1\}, \{-2,2\},\{-3,3\},\{-4,4\}\},\]
\item $\Lambda_9(P)=\Lambda_8(P)=\Lambda_7(P)=\Lambda_6(P) = \varnothing$,
\item $\Lambda_5(P) = \{ \{-3,-2\},\{2,3\}\}$ since 
\[C(t_1)=\{\{-2,-1\},\{1,2\}\} \subseteq \ccM(\overline{213}4) =C(213\overline{4}) = \{ \{-2,-1\},\{1,2\},\{-4,4\}\},\]
\item $\Lambda_4(P)=\Lambda_3(P)=\Lambda_2(P) = \varnothing$, and
\item $\Lambda_{1}(P) = \{ \{-1,1\}\}$ since 
$C(t_0) =\{\{-1,1\}\}\subseteq \ccM(1\overline{234}) = C(\overline{1}234) = \{ \{-1,1\}\}.$
\eei
This computation gives
\[ \Lambda(P) = \begin{cases}
 \{   \{-4,-1,0,1,4\},  \{-3,-2\},\{2,3\} \}&\text{in type BI} \\
 \{  \{-4,4\},  \{-3,-2\},\{2,3\}, \{-1,1\}\}&\text{in type CI}.
 \end{cases}
 \]
The sequence $P$ lifts to multiple paths in the weak order graphs for either type;
see Figure~\ref{lift-fig}.
\end{example}

\begin{figure}
{\[
\begin{tikzcd}
 (1,2,3,4,+,4,3,2,1)  \\ 
 (1,2,3,+,-,+,3,2,1) \arrow[u, Rightarrow, "t_0"] \\ 
(1,2,+,3,-,3,+,2,1) \arrow[u, "t_1"] \\
(1,+,2,3,-,3,2,+,1) \arrow[u, "t_2"]  \\
(1,+,2,3,-,2,3,+,1) \arrow[u, Rightarrow, "t_1"] \\ 
(1,+,2,2,-,3,3,+,1) \arrow[u, "t_0"] \\ 
(1,+,+,-,-,-,+,+,1)  \arrow[u, "t_1"]  \\
(+,1,+,-,-,-,+,1,+)  \arrow[u, "t_3"]  \\
(+,+,1,-,-,-,1,+,+)  \arrow[u, "t_2"] \\
(+,+,-,1,-,1,-,+,+)  \arrow[u, "t_1"] \\
(+,+,-,-,+,-,-,+,+) \arrow[u, Rightarrow, "t_0"]
\end{tikzcd}
\qquad\qquad\qquad\qquad
\begin{tikzcd}
 (1,2,3,4,4,3,2,1)  \\ 
 (1,2,3,+,-,3,2,1) \arrow[u, "t_0"] \\ 
(1,2,+,3,3,-,2,1) \arrow[u, "t_1"] \\
(1,+,2,3,3,2,-,1) \arrow[u, "t_2"] \\
(1,+,2,3,2,3,-,1) \arrow[u, Rightarrow, "t_1"]   \\ 
(1,+,2,2,3,3,-,1) \arrow[u, "t_0"]  \\ 
(1,+,+,-,+,-,-,1) \arrow[u, "t_1"] \\
(+,1,+,-,+,-,1,-) \arrow[u, "t_3"]   \\
(+,+,1,-,+,1,-,-)  \arrow[u, "t_2"]  \\
(+,+,-,1,1,+,-,-) \arrow[u, "t_1"]\\
(+,+,-,-,+,+,-,-) \arrow[u, "t_0"] 
\end{tikzcd}
\]}
\caption{Paths in weak order for types BI (left) and CI (right)
lifting the sequence $P$ in Example~\ref{lift-ex}.
The path in type BI has 3 doubled arrows while the path in type CI only has one.}\label{lift-fig}
\end{figure}
 
Recall from Section~\ref{brion-intro-sect} that if $\gamma \in\Gamma^G_K$ then $\cABrion(\gamma)$
is the set of all Weyl group elements of the form $s_1\cdots s_m$
such that $\gamma \xrightarrow{s_1} \cdots \xrightarrow{s_m} \gammaDense$ 
is a path in the weak order graph on $\Gamma^G_K$.

\begin{lemma}\label{weak-order-lem}
Continue to assume   $(G,K)$ is of classical type but not of type AI or AII. Define $P$ as in \eqref{path-eq}
and suppose $\gamma \in \Gamma^G_K$ has $z=\phiRS(\gamma)$. Then the following properties are equivalent:
\ben
\item[(a)] The sequence $P$ lifts to a path from $\gamma$ to $\gammaDense$  in the weak order graph for $(G,K)$.
\item[(b)] The Weyl group element $w =s_1s_2\cdots s_m$ belongs to $\cW^G_K(\gamma)$.
\item[(c)] Whenever $a<b$ are consecutive elements of some block in $ \Lambda(P)$ it holds that  
\[|S_+(\gamma) \cap \{a,b\}| = |S_-(\gamma) \cap \{a,b\}| = 1.\]
\een
\end{lemma}
 
\begin{proof} Property (a) implies property (b) by the definition of $\cW^G_K(\gamma)$.
Conversely, setting $s\cdot \gamma =\delta$ when the weak order graph has an edge of the form $\gamma \xrightarrow{s}\delta$ and defining $s\cdot \gamma = \gamma$ when no such edge exists 
determines an action of the monoid $(W,\circ)$ on $\Gamma^G_K$
\cite[\S0]{RichSpring}.
It follows that if $w\in \cW^G_K(\gamma)$, so that some choice of $P$  
corresponding to a reduced expression for $w$ lifts to a path in the weak order graph,
then all such paths $P$ lift in this way. Hence (b) also implies (a).

It remains to show the equivalence of   (a) and (c).
Fix a one-line representation $\gamma = (\gamma_1,\gamma_2,\dots)$.
Note that we can only have $0 \in S_+(\gamma) \sqcup S_-(\gamma)$ when $(G,K)$ is of type BI.
Furthermore, if we are in type BI  and $a_1,a_2,\dots ,a_l$ are defined as in \eqref{BI-a-eq},
then we must have $\{ \{-a_j,a_j \}\} = \Lambda_{i_j}(P)$ for an increasing sequence of indices $1\leq i_1 < i_2 <\dots < i_l \leq m$
by \cite[Lem.7.4]{HM}.

If $(G,K)$ is of type CI then every block $\{a,b\} \in \Lambda(P)$ with $a+b=0$ has 
$|S_+(\gamma) \cap \{a,b\}| = |S_-(\gamma) \cap \{a,b\}| = 1$ since $\gamma$ is skew-symmetric.
Similarly, if $(G,K)$ is of type CII, DI, or DIII
then there are no blocks $\{a,b\} \in \Lambda(P)$ with $a+b=0$. This holds
in type CII since the matching $\ccM(z_\dense)$ has no blocks $\{a,b\}$ with $a+b=0$,
so it follows from the formulas in Section~\ref{dem-sect} that the same is true of $\ccM(z^i)$ for all $i \in [m]$. 
Meanwhile, in types DI and DIII
we can never have $a+b=0$ for $\{a,b\} \in C(s_i)$
so our claim again holds.

Given these observations, by comparing Theorem~\ref{weak-order-thm} and Lemma~\ref{MS-lem} with Example~\ref{weak-ex-bc},
we see that  $P$ fails to lift to a path in the weak order graph if and only if
for some $i \in[m]$ 
there is a block $\{a,b\} \in \Lambda_i(P)$ with $a+b\neq0$ and either
 $\{a,b\}\subseteq S_+(\gamma)$ or $\{a,b\}\subseteq S_-(\gamma)$, or if in type BI 
the  signs $\gamma_0,\gamma_{a_1}$, $\gamma_{a_2}$, \dots, $\gamma_{a_l}$
do not alternate when $a_1,a_2,\dots,a_l$ are defined as in \eqref{BI-a-eq}.
These obstructions are avoided precisely when condition (c) holds.
\end{proof}

\subsection{Brion atoms}
Recall from \eqref{cEABrion-eq} that each $z \in \cI^G_K$ then has an associated set of \defn{extended Brion atoms} $\cEABrion(z) $, which is just the union of $\cABrion(\gamma)$ over all preimages $\gamma$ for $z$ under
the Richardson--Springer map from Definition~\ref{rs-def}.
Using Theorem~\ref{weak-order-thm} we may describe $\cEABrion(z)$ more explicitly in all classical types. This will be a useful intermediate step 
on the way to doing the same for the sets  $\cABrion(\gamma)$.

It follows from \eqref{szs-eq} by induction on length that 
every $z \in \I_\Theta(W)$ 
has the form $z=\Theta(w)\circ  w^{-1}$ for some $w \in W$.
For  $z \in \I_\Theta(W)$ we may therefore define
\begin{subequations}
\begin{align}
 \ellhat_\Theta(z) &= \min\left\{\ell(w) : w\in W\text{ and } \Theta(w)\circ  w^{-1} = z\right\}
 \quad\text{along with}
\\
\cA_\Theta(z) &= \left\{ w \in W : \ell(w)  = \ellhat_\Theta(z)\text{ and } \Theta(w)\circ  w^{-1} = z\right\}.
\end{align}
\end{subequations}
We use the abbreviation  $\cA(z) \defequals \cA_\id(z)$.
Results in  \cite[\S3]{H2} imply that if $s\in S$ then
\be\label{szs-length-eq}
 \Theta(s) \circ z \circ s\neq z \quad\text{if and only if}\quad
\ellhat_\Theta(\Theta(s) \circ z \circ s) = \ellhat_\Theta(z) + 1.
\ee
Fix elements $y,z \in \I_\Theta(W)$ and consider the more general set 
\be\cA_\Theta(y,z) =\left\{ \text{minimal-length elements  
$ w \in W$ with $\Theta(w)\circ y \circ w^{-1} = z$}\right\}.\ee
We abbreviate $\cA(y,z) \defequals \cA_\id(y,z)$.
Unlike $\cA_\Theta(z)$ the set $\cA_\Theta(y,z)$ may be empty.
We omit the proof of the following basic algebraic lemmas,
which can be derived from \cite[\S2]{HMP2}.

\begin{lemma}\label{dem-lem0}
It holds that
$
\cA_\Theta(y,z) = \left\{ w \in W : \ell(w)  = \ellhat_\Theta(z)-\ellhat_\Theta(y)\text{ and } \Theta(w)\circ y\circ   w^{-1} = z\right\}.
$
\end{lemma}

%\begin{proof}
%Suppose $w \in W$ has $\Theta(w)\circ y\circ   w^{-1} = z$. Write $w=s_1\circ s_2\circ \cdots \circ s_{\ell(w)}$
%where each $s_i \in S$.
%Omitting any single factor $s_i$ from this Demazure product yields an element $v\in W$ that is shorter than $w$.
% It follows from \eqref{szs-length-eq}
%that $\ellhat_\Theta(z) = \ellhat_\Theta(\Theta(w)\circ y\circ   w^{-1}) \leq \ellhat_\Theta(y) + \ell(w)$
%and   if this inequality is strict then there exists a shorter element $v \in W$ with $\Theta(v)\circ y\circ   v^{-1} = z$ so $w \notin \cA_\Theta(y,z) $.
%\end{proof}

%We note another basic lemma concerning the set $\cA_\Theta(y,z)$.

\begin{lemma}\label{dem-lem}
Let $u \in W$ and $v \in \cA_\Theta(y)$. 
Then $u \in \cA_\Theta(y,z)$ if and only if $uv \in \cA_\Theta(z)$ and $\ell(uv) = \ell(u) + \ell(v)$.
When this holds, $u\mapsto uv$ is a bijection $\cA_\Theta(y,z) \to \{ w \in \cA_\Theta(z) : \ell(wv^{-1}) = \ell(w)-\ell(v)\}.$
\end{lemma}

%\begin{proof}
%If $u \in \cA_\Theta(y,z)$ then $\Theta(u\circ v) \circ (u\circ v)^{-1} = \Theta(u) \circ y \circ u^{-1} = z$
%so
%\[
%\ellhat_\Theta(z) \leq \ell(u\circ v)  \leq \ell(u)+\ell(v) = \(\ellhat_\Theta(z)-\ellhat_\Theta(y)\) + \ellhat_\Theta(y)= \ellhat_\Theta(z).
%\]
%In this case we   have $\ell(u\circ v)  = \ell(u)+\ell(v)=\ellhat_\Theta(z)$ and $uv = u\circ v \in \cA_\Theta(z)$.
%Conversely, if $uv \in \cA_\Theta(z)$ and $\ell(uv) = \ell(u) + \ell(v)$ then $uv = u\circ v$ 
%so it holds that 
%\[\Theta(u)\circ y \circ u^{-1} = \Theta(u)\circ \Theta(v) \circ v^{-1} \circ u^{-1} = \Theta(uv) \circ (uv)^{-1} = z\]
%along with $\ell(u) = \ell(uv) - \ell(v) = \ellhat(z) - \ellhat(y)$, and therefore $u \in \cA_\Theta(y,z)$ by Lemma~\ref{dem-lem0}.
%\end{proof}

We now present several definitions that are specific to classical Weyl groups.
Recall the elements $\omega^n_k \in S_n$, $\sigma^n_k \in \W_n$, 
$\hat\sigma^n_k=(t_0)^k \sigma^n_k \in \WD_n$, and $\sigma^n_{k\times \fpf} \in \W_n$ from \eqref{omega-eq}, \eqref{sigmakn-eq}, and \eqref{sigmakn-fpf-eq}.

\begin{definition}\label{AA-def}
Assume that $z \in W$ and $k$ is an integer with $0\leq k\leq n$.
\bei
\item[(a)] When $W=S_{n}$ and $n-k$ is even let $\cAA_\ast(z:k) \defequals  \cA_\ast(\omega_k^{n},z)$.

\item[(b)] When $W=\W_n$ let $\cAB(z:k) \defequals  \cA(\sigma_{k}^n,z)$.

\item[(c)] When $W=\WD_n$ let 
$\cAD(z:k) \defequals 
\begin{cases}
\cA( \hat\sigma^n_k,z)&\text{if $k$ is even}
\\
\cA_\diamond( \hat\sigma^n_k,z)&\text{if $k$ is odd}.
\end{cases}$

\eei
\end{definition}

\begin{definition}\label{AA-fpf-def}
Again assume that $z \in W$ and $k$ is an integer with $0\leq k\leq n$. 
\bei
\item[(a)]
When $W=S_{n}$ and $n$ is even   let  $\cAfpfA(z) \defequals  \cA(t_1\cdot t_3\cdot t_5\cdots t_{n-1},z)$. 

\item[(b)]
When $W=\W_n$ and $n-k$ is even let $\cAfpfC(z:k) \defequals  \cA(  \sigma_{k\times \fpf}^n ,z)$.

\item[(c)]
When $W=\WD_n$  let $\cAfpfD(z) \defequals 
\begin{cases} \cA (t_1 \cdot t_3  \cdot t_5\cdots t_{n-1}, z) &\text{if $n$ is even} \\ 
\cA_\diamond (t_2\cdot t_4\cdot  t_6\cdots t_{n-1}, z)
&\text{if $n$ is odd}.
\end{cases}$

\eei
\end{definition}

These definitions are at the service of the following statement.

\begin{proposition}\label{drho-prop}
In classical type,
the sets $\cEABrion(z)\subseteq W$ for $z \in \cI^G_K$ are as specified in Table~\ref{extended-brion-tbl},
and there exists $\rho^G_K(z)\in \NN$
with $\ell(w) =\rho^G_K(z)$ for all $w \in \cEABrion(z)$, which is also specified in Table~\ref{extended-brion-tbl}.
\end{proposition}

\begin{proof}
It is clear from Theorem~\ref{weak-order-thm} and Table~\ref{rs-image-tbl}
that each instance of $\cEABrion(z)$ is contained in the set identified in Table~\ref{extended-brion-tbl}.
While it is not infeasible to show directly
that each of these containments is equality, this will also follow from our proof of Theorem~\ref{main-thm} in later sections.

Specifically, if we use the formulas in Table~\ref{extended-brion-tbl} to define $\cEABrion(z)$,
then 
to deduce \eqref{cEABrion-eq} from Theorem~\ref{main-thm}
it will suffice
to observe that for any   $z \in \cI^G_K$ and $M \in \cM^G_K(z)$
there exists some $\gamma \in \Gamma^G_K$ with $\phiRS(\gamma)=z$ for which $M  $ is $\gamma$-aligned.
This claim will be evident from the definitions in Sections~\ref{A-shape-sect}, \ref{BC-shape-sect}, and \ref{D-shape-sect}.

Each set $\cEABrion(z)$ has the form $\cE_\Theta(y,z)$ for a certain element $y \in \cI^G_K$,
 and so it follows from Lemma~\ref{dem-lem0} that all elements of this set have the same length
 $\rho^G_K(z) = \ellhat_\Theta(z) - \ellhat_\Theta(y).$
Hultman's results \cite[Thm.~4.8]{H1} imply that $\ellhat_\Theta$ is the rank function of $\I_\Theta(W)$ 
 under Bruhat order.
  When $\Theta$ is the identity map, 
 this rank function was explicitly computed by Incitti \cite[Thm.~4.1]{Incitti2}.
 Incitti's identities \cite[Props.~3.1, 3.3, and 3.4]{Incitti2}
 lead directly to the given formulas for 
 $\rho^G_K(z)$ in all classical types except AIII, DI when $k=\frac{|p-q|}{2}$ is odd, and DIII when $n$ is odd. In these cases $\Theta$ is not the identity map.

The formula desired in type AIII can be derived using the identity
$ \ellhat_\ast(z) = \ellhat_\id(w_0) - \ellhat_\id(w_0 z)$ \cite[Cor.~3.9]{HMP2}
along with Incitti's type A formula for $\ellhat_\id$ \cite[Props.~3.1]{Incitti2}.
To handle the remaining type D cases, 
fix $z \in \I_\diamond(\WD_n)$ and define 
$
 z^\oplus = \overline{z_1} z_2z_3\cdots z_n \overline{n+1} \in \I(\WD_{n+1}).
$
One can check the following identities using the background in Section~\ref{d-conj-sect} plus \eqref{szs-length-eq}:
 \be
 \ell(z^\oplus) = \ell(z) + 2n, \qquad \ellhat_{\id}(z^\oplus) = \ellhat_\diamond(z) + n + 1,
 \qquand
 \kappa(z^\oplus) = \kappa(t_0z) + 1
 \ee
 where we define $\kappa(z)= |\{ i : -i \leq z(i) < i\}|$. 
 From here, one can derive the remaining formulas in Table~\ref{extended-brion-tbl} for $\rho^G_K(z)$ 
using Incitti's result that  $\ellhat_{\id}(z^\oplus) = \frac{\ell(z^\oplus)+\kappa(z^\oplus)}{2}$  \cite[Prop.~3.4]{Incitti2}.
\end{proof}

The following nontrivial observation characterizes the coefficients $2^{d_\gamma(w)}$ appearing in \eqref{brion}.

\begin{proposition}\label{dgamma-prop}
In classical type, 
for each $z \in \cI^G_K$
there exists a map $d_z : \cEABrion(z) \to \NN$   such that   
\[d_\gamma(w) = d_z(w)\quad\text{whenever $\gamma \in \Gamma^G_K$
has  $z=\phiRS(\gamma)$ and
$w \in \cABrion(\gamma)$.}\]
These maps have the formulas specified in Table~\ref{extended-brion-tbl}.
\end{proposition}

\begin{proof}
Let $\gamma \in \Gamma^G_K$. 
It is clear from Theorem~\ref{weak-order-thm} 
that $d_\gamma = 0$ outside types AI, BI, CI, and DI.
In types BI and CI the map $d_\gamma$ has the formula in Table~\ref{extended-brion-tbl},
which depends only on $z=\phiRS(\gamma)$, by \cite[Thms.~8.2 and 8.7]{HM}. (The cited theorems have an error in their formulas for $d_\gamma$ but this is corrected in the {\tt arXiv} version of \cite{HM}.)

For the remaining cases, let $z = \phiRS(\gamma)$
and define $\cyc(\gamma)$ to be 
$|\{ i \in [n] : z(i) < i\}|$ and 
$\frac{1}{2} \left|\left\{ i \in [\pm n] : (t_0)^kz(i) < i\right\}\right|$
in
 types AI and DI, respectively, where in the type DI case $k=\frac{|p-q|}{2}$.
By comparing Theorem~\ref{weak-order-thm} with the formulas for Demazure conjugation in Section~\ref{dem-sect},
we see that if $\gamma \rightarrow \delta$ is an edge in the weak order graph for these types
then   $\cyc(\gamma)  =\cyc(\delta)+1$ when the edge is doubled and $\cyc(\gamma)  =\cyc(\delta)$ otherwise.
As   $d_{\gamma_{\dense}}=0$, the map
 $d_\gamma $ is constant on $\cABrion(\gamma)$ with value $\cyc(\gamma)- \cyc(\gamma_{\dense})$. 
 We get the desired formulas by evaluating $\cyc(\gamma_{\dense})$ using Table~\ref{rs-image-tbl}.
\end{proof}

\def\hhline{\\ & & && \\ [-\vsp]\hline & & &&  \\ [-\vsp]}
\def\gap{\\[-\vsp]&&&&\\}
\begin{table}[h]
\begin{center}
{\small
\begin{tabular}{| l | l | l | l | l |}
\hline&&&& \\[-\vsp]
Type & Parameters &  $\cEABrion(z)$ for $z \in  \cI^G_K $   & $d_z(w)$ for $w \in \cEABrion(z)$ 
 & $\rho^G_K(z)$  for $z \in  \cI^G_K $
\hhline
AI &  $n \in \NN$  & $\cA_{\AI}^{(n)}(z) \defequals  \cA(z) $  &  $|\{ i \in [n] : z(i) < i\}|$ & $\frac{\ell(z)+\kappa(z)}{2} $
\gap
AII  & $n \in \NN$ even & $\cA_{\AII}^{(n)}(z) \defequals  \cAfpfA(z)$   & $0$ & $\frac{\ell(z)-\kappa(z)}{2} $
\gap
AIII & $n=p+q$ & $\cA_{\AIII}^{(p,q)}(z) \defequals  \cAA_{\ast}(z:k)$% for $k=|p-q|$
 & $0$  & $\frac{\ell(z)-\kappa(w_0z)}{2}  +  \frac{n-k^2}{4}$
\hhline
BI &  $2n+1=p+q$ %and $k=\frac{|p-q|-1}{2}$ 
& $\cA_{\BI}^{(p,q)}(z) \defequals  \cAB(z:k)$ % for $k=\frac{|p-q|-1}{2}$
 & $|\{i \in [n] : 0 < z(i) < i\}| + \ell_0(w)$ & $\frac{\ell(z)+\kappa(z)}{2} - \binom{k+1}{2}$
\hhline
CI & $n \in \NN$ & $\cA_{\CI}^{(n)}(z) \defequals  \cA(z)$    & $|\{i \in [n] :  z(i) < i\}| - \ell_0(w)$ & $\frac{\ell(z)+\kappa(z)}{2} $
\gap
CII  & $2n=p+q$ % with $p$ even %and $k=\frac{|p-q|}{2}$ 
&  $\cA_{\CII}^{(p,q)}(z) \defequals  \cAfpfC(z:k)$ %for $k=\frac{|p-q|}{2}$ 
& $0$ & $\frac{\ell(z)+\kappa(z)}{2} -  \frac{n+k^2}{2}$ 
%\\   & with $p$ even   &    & &
\hhline
DI  & $2n=p+q$ % with $p+n$ even %and $k=\frac{|p-q|}{2}$
& $\cA_{\DI}^{(p,q)}(z) \defequals  \cAD(z:k)$ %for $k=\frac{|p-q|}{2}$ 
& $\frac{1}{2} \left|\left\{ i \in [\pm n]: \varepsilon_k z(i)<i\right\}\right| - \frac{k}{2}$ & $\frac{\ell(z)+\kappa(\varepsilon_k z)}{2} - \frac{k^2}{2}$
\gap
DIII & $n\in\NN$  & $\cA_{\DIII}^{(n)}(z) \defequals  \cAfpfD(z)$ & $0$ & $\frac{\ell(z)+\kappa(\varepsilon_nz)}{2}  - \frac{n}{2}$
\gap\hline
\end{tabular}}
\end{center}
\caption{Brion atoms for symmetric subgroups of classical groups.
The parameter $k$ is defined as in Table~\ref{rs-image-tbl} to be $k = |p-q|$ in type AIII, $k=\frac{|p-q|-1}{2}$ in type BI, and $k=\frac{|p-q|}{2}$ in all other relevant types.
%Outside types BI and CI the values of $d_z(w)$ in the fourth column depend only on $z \in \cI^G_K$ and are independent of $w$.
In the last two columns we set $\varepsilon_k = (t_0)^k \in \W_n$ and  $\kappa(z) = |\{ i : -i \leq z(i) < i\}|$.}\label{extended-brion-tbl}
\end{table}

 \section{Shape preliminaries}\label{shape-prelim-sect}

Sections~\ref{A-shape-sect}, \ref{BC-shape-sect}, and \ref{D-shape-sect} will 
explain the definitions used in Theorems~\ref{main-thm} and \ref{main-thm2}
for pairs $(G,K)$ of classical type.
This section sets up some notation to streamline this material, along with some 
general properties that will be used to derive the main theorems in type D. 
 
 \subsection{Deduplication}

A \defn{word} is a finite sequence of integers $w=w_1w_2\cdots w_m$.
Given such a sequence, let $[[w]]$ be the subword formed by removing each repeated letter after 
its first appearance, going left to right. 

When   $\{w_1,w_2,\dots,w_m\}=[n]$, we interpret $[[w]]$ as an element of $S_n$
written in one-line notation.
Similarly, when $\{w_1,w_2,\dots,w_m\}$ has size $n$ and contains exactly one of $i$ or $-i$ for each $i \in [n]$,
we interpret $[[w]]$ as an element of $\W_n$.
Thus
$[[ 2\overline{1}33542534]] = 2\overline{1}354 \in \W_5.$

Now suppose $P$ is a finite subset of $\ZZ\times \ZZ$.
Write  
$P = \{(a_1,b_1),(a_2,b_2),\dots,(a_m,b_m)\}$ where $b_i<b_{i+1}$ or ($b_i=b_{i+1}$ and $a_i<a_{i+1}$) for all $i \in[m-1]$.
Then define
\be
[[P]]_{\des} \defequals  [[b_1a_1b_2a_2\cdots b_ma_m]]
\quand
[[P]]_{\asc} \defequals  [[a_1b_1a_2b_2\cdots a_mb_m]].
\ee
For example, if we set
\be \Cyc(z) = \{ (a,z(a)) \in [n]\times [n] : a \leq z(a)\}
\quad\text{for $z \in S_n$}\ee
then 
$\Cyc(52431)  = \{(2,2), (3,4),(1,5)\}$, 
$
[[\Cyc(52431)]]_{\des} = 24351
$,
and
$
[[\Cyc(52431)]]_{\asc} = 23415.
$

\subsection{Word relations}\label{word-rel-sect}

We use the term \defn{word relation} to mean a reflexive relation on finite sequences of integers with the same length.
There are two families of word relations 
that we shall use repeatedly to formulate Theorem~\ref{main-thm2}.
These relations are indexed by an integer parameter $k \in\NN$.
\begin{definition} Let $\overset{k}\precsim$ be the transitive closure of the word relation with 
\be\label{precsim-eq}
u B C A v  \overset{k}\precsim u C A B v
\ee
whenever $u$ and $v$ are words with $\ell(u)\in k+\NN$ and $A,B,C\in \ZZ$ have $A<B<C$. 
\end{definition}

\begin{definition}
Let $\overset{k}\precapprox$ be the transitive closure of the word relation with 
\be\label{precapprox-eq}
u B C A D v \overset{k}\precapprox u A D B C v
\ee
whenever $u$ and $v$ are words with $\ell(u)\in k+2\NN$ and $A,B,C,D\in \ZZ$ have $A<B<C<D$. 
\end{definition}

Both of these relations are partial orders, since they are sub-relations of lexicographic order and reverse lexicographic order, respectively.
The partial order $\overset{k}\precapprox$ is clearly graded, as follows by
 considering the rank function that sends a word $w$
to the number of inversions in   $w_{k+2}w_{k+4}w_{k+6}\cdots $.

 \subsection{Nested descents}
 
 The shape operators for each classical type will involve the following construction used previously in 
 \cite[\S3]{HM} and \cite[\S6-\S7]{Mar2019}.
 An index $i\in[n-1]$ is a \defn{descent} of a word $w=w_1w_2\cdots w_n$ if $w_i >w_{i+1}$.

\begin{definition}\label{ndes-def} The \defn{nested descent set} of $w$ is the set of pairs given recursively by the formula
\be
\NDes(w) = \begin{cases} 
\varnothing &\text{if $w$ has no descents}, \\
\{ (w_i, w_{i+1})\} \cup \NDes(w_1\cdots w_{i-1} w_{i+2}\cdots w_n) &\text{if $i$ is the first descent of $w$}.
\end{cases}
\ee
Likewise, define the \defn{nested residue} of $w$ to be the set
\be
\NRes(w) = \begin{cases} 
\{w_1,w_2,\dots,w_n\} &\text{if $w$ has no descents}, \\
\NRes(w_1\cdots w_{i-1} w_{i+2}\cdots w_n) &\text{if $i$ is the first descent of $w$}.
\end{cases}
\ee 
When $w \in \W_n$ we evaluate $\NDes(w)$ and $\NRes(w)$ by replacing $w$
with  $w_1w_2\cdots w_n$.
\end{definition}

Write $\precsim$ for $\overset{0}\precsim$ from Section~\ref{word-rel-sect}.
A word $w=w_1w_2\cdots w_n$ 
has a \defn{consecutive $321$-pattern} if $w_i>w_{i+1}>w_{i+2}$ for some $i \in[n-2]$.
A word is a \defn{partial permutation} if its letters are all distinct.
If $w$ is a partial permutation then we say that it is a partial permutation of the set $\{w_1,w_2,\dots,w_n\}$.

\begin{definition} A set of 
partial permutations is a
 \defn{well-nested family} if it is closed under both directions of $\precsim$ and
 none of its elements have any consecutive $321$-patterns.
 \end{definition}

For example the set %of 5-letter partial permutations
$
\{45362, \  45623, \ 46253,\   53462,\   53624, \  56234,\  62453,\  62534\}
$ of 5-letter partial permutations
is closed under $\precsim$ and forms a well-nested family.

 Given a word $w =w_1w_2\cdots w_n$ let
 $\Des(w) = \{ (w_i, w_{i+1}) : i \in [n-1]\text{ with }w_i>w_{i+1}\}$.

\begin{lemma}\label{nest-lem}
Suppose $E$ is a well-nested family of partial permutations.
Choose  $(b,a) \in \ZZ\times \ZZ$ with $a<b$
 and let $F$ be the set of words formed by removing $ba$ from each $w \in E$ with $ (b,a) \in \Des(w)$. 
 Then $F$ is also well-nested.
\end{lemma}

\begin{proof}
Suppose $w \in E$ has $(b,a) \in \Des(w)$.
Form $\hat w$ from $w$ by removing the letters $a$ and $b$.
Since $w$ has no consecutive $321$-patterns, $\hat w$ can only have a consecutive $321$ pattern if
$w$ has a consecutive subword of the form $ZYbaX$ or $ZbaYX$ where $X<Y<Z$.
In the first case we cannot have $Y>b$ or $a>X$ so $ZYXba \precsim ZYbaX$
and in the second case we cannot have $Z>b$ or $a>Y$ so $baZYX \precsim ZbaYX$.
Either way we see that if $\hat w$ has a consecutive $321$-pattern then 
there is a word $v \precsim w$ with a consecutive $321$-pattern, which contradicts
the assumption that $E$ is well-nested.

To show that $F$ is well-nested, it remains to show that this set is closed under $\precsim$.
Write $\sim$ for the symmetric closure of this relation.
Suppose $X<Y<Z$  
and
$Y ba ZX$
is a consecutive subword of some $w \in E$.
Then $Y<b$ and $a<Z$ since $w$ has no consecutive $321$-patterns. If $a<Y$ then 
$Yba ZX \sim   ba YZX \sim ba ZXY$
and if $Y < a$ then $X<a$ so
 $Yba ZX \sim   YbZXa \sim YZX ba \sim ZXY ba$
 as we cannot have $b>Z>X$ since $E$ is well-nested. 
It follows that replacing the subword $YbaZX$ in $w$ by $YZX$ or by $ZXY$ yields two elements of $F$.

A similar argument shows that if 
$Y Z ba X$ is a consecutive subword of some $w \in E$
then replacing this subword by $YZX$ or by $ZXY$ yields two elements of $F$.
We conclude that if $u\in F$ and $u \precsim v$ then $v \in F$.
Symmetric reasoning shows that if $v\in F$ and $u \precsim v$ then $u \in F$.
\end{proof}

\begin{proposition}\label{ndes-prop}
Suppose $w=w_1w_2\cdots w_n$ is a partial permutation that belongs to some well-nested family.
If $j \in [n-1]$ is any descent of $w$ and $\hat w= w_1\cdots w_{j-1} w_{j+2}\cdots w_n$ then 
\[ \NDes(w) = \{(w_j,w_{j+1})\} \sqcup \NDes(\hat w)
\quand
\NRes(w) = \NRes(\hat w)
.
\]
\end{proposition}

This proposition shows that when $w$ is in a well-nested family of partial permutations,
we can inductively compute $\NDes(w)$ and $\NRes(w)$ by choosing any descent at each step.

\begin{proof}
Suppose $i \in [n-1]$ is minimal with $w_i>w_{i+1}$ and $i\neq j \in[n-1]$ is another index with $w_j>w_{j+1}$.
Because $w$ has no consecutive $321$-patterns, we must have $i+1<j$.
Form $w^i$ from $w$ by removing $w_i$ and $w_{i+1}$, form $w^j$ by instead removing $w_j$ and $w_{j+1}$,
and form $w^{ij}$ from $w$ by removing all four letters $w_i, w_{i+1}, w_j,w_{j+1}$.

By definition $\NDes(w) = \{(w_i,w_{i+1})\} \sqcup \NDes(w^i)$ and by induction via Lemma~\ref{nest-lem}
we have $\NDes(w^i) = \{(w_j,w_{j+1})\} \sqcup \NDes(w^{ij})$.
Since by definition $\NDes(w^j) = \{(w_i,w_{i+1})\}\sqcup \NDes(w^{ij})$,
we have $\NDes(w) =  \{(w_j,w_{j+1})\} \sqcup\NDes(w^j)$ as claimed.

Similarly, we have $\NRes(w) = \NRes(w^i)$ and $\NRes(w^j) = \NRes(w^{ij})$ by definition 
while $\NRes(w^i) = \NRes(w^{ij})$ by induction  via Lemma~\ref{nest-lem},
so $\NRes(w)=\NRes(w^j)$.
\end{proof}

\begin{lemma}\label{ijkl-lem}
Suppose
$w=w_1w_2\cdots w_n$ belongs to a well-nested family
and $i,j,k,l \in [n]$ are indices such that 
$(w_i, w_j), (w_k,w_l) \in \NDes(w)$. If $w_i < w_k$ and $w_j<w_l$, 
then $i<j<k<l$.
\end{lemma}

\begin{proof}
We have $i <j$ and $k<l$ by the definition of $\NDes(w)$.
Let $X = \{w_1,w_2,\dots,w_n\}$ and $w_a = \min(X)$ and $w_b = \max(X)$.
We may assume that $a=j$ since otherwise after removing $w_a$ and also $w_{a-1}$ when $a>1$
we can deduce that $j<k$ by induction using Lemma~\ref{nest-lem}.
In turn, we may also assume that $b=k$ since otherwise after removing $w_b$ and also $w_{b+1}$ when $b<n$
we can likewise deduce that $j<k$ by induction.

In this case we must also have $a-1=i$ and $b+1=l$.
Now, if we do not have $j<k$ then $k=l-1<l<i<i+1=j$,
but then  there is a word $v$ with $v\precsim w$
that has a consecutive subword of the form $w_k w_l w_i w_j$. 
If $w_l > w_i$ then this word has a consecutive $321$-pattern and if $w_l<w_i$
then $w_k w_l w_i w_j \precsim w_k w_i w_j w_l$ and the second word has a consecutive $321$-pattern.
Either situation contradicts the assumption that $w$ belongs to a well-nested family.
\end{proof}

 We mention one other general proposition related to well-nested families.

\begin{proposition}\label{Pdes-prop}
Suppose
$E$ is a well-nested family of partial permutations.
Then $E$ is a graded poset relative to $\precsim$.
Also, if $w \in E$ then 
the lower set bounded above by $w$ under $\precsim$ has a unique
 minimal element given by $[[P]]_\des$
for
$P = \{ (a,b) : (b,a) \in \NDes(w)\}\sqcup \{(c,c) : c \in \NRes(w)\}.$
\end{proposition}

\begin{proof}
Given $w \in E$ let $w_R$ be the subword 
retaining only the letters in $\{ b : (a,b) \in \NDes(w)\}$
and let $w_L$ be the complementary subword.
For example, if $w=45362$ then $\NDes(w) =\{(5,3),(6,2)\}$ so $w_R=32$ and $w_L = 456$.
To see that $\precsim$ is a graded partial order on $E$,
consider the rank function $\rank(w) = \inv(w_L) - \inv(w_R)$
where $\inv(a_1a_2\cdots a_n) = |\{ (i,j) \in [n]\times [n] : i<j\text{ and }a_i>a_j\}|$. 
It is easy to see that if $v\precsim w$ is a covering relation then $\rank(w) = \rank(v)+1$.

Define 
$P = \{(a_1,b_1),(a_2,b_2),\dots,(a_k,b_k)\}$ where $a_i \leq b_i$ and $b_1<b_2<\dots<b_k$.
Form $\hat w$ by removing $a_k$ and $b_k$ from $w$. 
As $E$ is well-nested and $b_k$ is the largest letter of $w$, we either have $w = \hat w b_k$ when $a_k=b_k$,
or it holds that $w_iw_{i+1} =b_k a_k$ for some index $i$ in which case 
one can check that $\hat w b_k a_k \precsim w$.
In view of Proposition~\ref{ndes-prop} we may assume by induction that $[[\hat P]]_\des \precsim \hat w$ for   $\hat P = P\setminus\{(a_k,b_k)\}$.
It follows that $[[P]]_\des \precsim w$. Finally, $[[P]]_\des$ is clearly minimal under $\precsim$.
\end{proof}

\begin{corollary}
Suppose
$v$ and $w$ are partial permutations that are both contained in the same well-nested family.
Then $v$ and $w$ belong to the same equivalence class for the symmetric closure of $\precsim$
if and only if $\NDes(v) = \NDes(w)$ and $\NRes(v)=\NRes(w)$.
\end{corollary}

\begin{proof}
The ``only if'' direction follows from Proposition~\ref{ndes-prop}.
Conversely, if $\NDes(v) = \NDes(w)$ and $\NRes(v)=\NRes(w)$ then $ [[P]]_\des \precsim v$ and $[[P]]_\des \precsim w$ 
for the same $P$ as in Proposition~\ref{Pdes-prop}.
\end{proof}

Using \cite{Mar2019}, we can  classify all well-nested families of partial permutations $w$.
 If  $w=w_1w_2\cdots w_n$ then define $\std(w) \in S_n$
to have $\std(w)(i) < \std(w)(j)$ if and only if $w_i<w_j$ for each $i,j\in[n]$.
A well-nested family is \defn{minimal} if it is a single equivalence class under the symmetric closure of $\precsim$.
\begin{corollary}
Let 
 $E$ be a  set of partial permutations that all have the same letter set.
Then $E$ is a minimal well-nested family if and only if   $\std : E \to \cA(z)$ is a bijection for some $z\in\I(S_n)$.
\end{corollary}
\begin{proof}
The map $\std : E \to S_n$ is a poset isomorphism onto its image relative to $\precsim$, and any $\cA(z)$ for $z \in \I(S_n)$ is a minimal well-nested family by the discussion in \cite[\S6-\S7]{Mar2019}.
This proves the ``if'' direction.
Conversely, when $E$ is a minimal well-nested family and $w=w_1\cdots w_n \in E$, let $\phi $ be the order-preserving
bijection $\{w_1,\dots,w_n\} \to [n]$, and define
$
\textstyle
z = \prod_{(a,b) \in \NDes(w)} (\phi(a),\phi(b)) \in \I(S_n).
$
Then $\std $ is a bijection $E \to \cA(z)$  \cite[Thm.~6.14]{Mar2019}.
\end{proof}

\subsection{Matchings}\label{matching-sect}

Recall that 
a \defn{matching} on a totally ordered finite set $Y$ is a collection of pairwise disjoint subsets of size one or two,
called \defn{blocks}, whose union is $Y$. A matching is \defn{perfect} if it has no blocks of size one,
and \defn{noncrossing} if it does not have blocks $\{a,c\}$ and $\{b,d\}$ with $a<b<c<d$.
A matching on $Y$ 
(or a block of such a matching)
 is \defn{symmetric} if it is fixed by the reverse permutation of $Y$.

\begin{definition}
Given a finite set of positive integers $X$, let $\NCSP(X)$ be the set of  matchings on the disjoint union $X\sqcup -X$
that are noncrossing, symmetric, and perfect.
For $k \in \NN$ let
\[
 \NCSP(X: { k}) = \{ M\in\NCSP(X) \text{ with exactly $k$ symmetric blocks}\}.
 \]
\end{definition}

The size of $\NCSP(X)$ is the binomial coefficient
$ | \NCSP(X)| = \binom{m}{\lfloor m/2\rfloor}
$ for $m=|X|$
and its elements are in bijection with many combinatorial objects; see \cite[A001405]{OEIS} and \cite[A053121]{OEIS}.

\begin{remark}\label{ncsp-gen-rmk}
All matchings  $M\in \NCSP(X: { k})$ can be generated by the following algorithm.
Set $X_0 = X$ and $m=|X|$. If $m-k$ is odd then $\NCSP(X: { k})$ is empty. Otherwise, for each $i=1,2,\dots,(m-k)/2$ choose   two consecutive elements $a_i<b_i$ in $X_{i-1}$
and form $X_i = X_{i-1} \setminus\{a_i,b_i\}$. Then let 
$M$ consist of the pairs $\{a_i,b_i\}$ and $\{-a_i,-b_i\}$ along with $\{-c,c\}$ for each $c \in X_{(m-k)/2}$.
\end{remark}

A block of size one in a matching $M$ is called \defn{trivial}.
Write $\mathrm{Triv}(M) = \{ c : \{c\} \in M\}$ and let $\mathsf{triv}(M) = | \mathrm{Triv}(M)|$ be the number of trivial blocks in $M$.
For any set $S$ define 
\be
M\cap S = \{ B \cap S : B \in M\} \setminus \{\varnothing\}
\quand
M^+ = M\cap \PP.
\ee
Form $M^\div$ from $M \in \NCSP(X)$ by dividing each symmetric block $\{\pm c\}$ into two blocks $\pm\{ c\}$.
These constructions give, for example,
\[
\ba 
\{ \{-3,-2\},\{-1,1\},\{2,3\}\}^+ &= \{\{1\}, \{2,3\}\},
 \\
\{ \{-3,-2\},\{-1,1\},\{2,3\}\}^\div &= \{\{-3,-2\},\{-1\}, \{1\}, \{2,3\}\}.
\ea
\]
We will frequently compose these definitions to consider the set
\[
\Triv{M} = \left\{ b \in \PP : \{b\} \in M\text{ or }\{a,b\} \in M\text{ for some }a \notin \PP\right\}.\]

Continue to fix a finite set of positive integers $X=\{x_1<x_2<\dots<x_m\}$ and let $k$ be an integer with $0\leq k \leq m=|X|$ such that $m-k$ is even.
Then $\NCSP(X:k)$ contains the matching 
\be\label{Mmin-eq}
M_{\min}(X:k) \defequals
\Bigl\{ \pm \{x_{1},x_{2}\}, \dots, \pm \{x_{m-k-1},x_{m-k}\}\Bigr\} \sqcup \Bigl\{ \{\pm x_{m-k+1}\},  \dots, \{\pm x_m\} \Bigr\}
\ee
Given an element $N \in \NCSP(X:k)$ we write $M\lessdot N$ if $M$ can be formed from $N$
by either
\bei
\item replacing blocks $\{\pm x_i\}$, $\pm \{x_{i+1},x_p\}$ with $i+1<p$
by $\pm\{x_{i},x_{i+1}\}$, $\{\pm x_p\}$; or 

\item replacing blocks $\pm\{x_i,x_{q}\}$, $\pm\{x_{i+1},x_p\}$ with $i+1<p<q$
by $\pm\{x_i,x_{i+1}\}$, $\pm\{x_{p},x_{q}\}$
when no block $\{x_a,x_b\} \in N$ exists with $x_a < x_i< x_b$.
\eei
Notice that the result of either operation is another element $M\in  \NCSP(X:k)$. 

\begin{proposition}\label{lessdot-prop}
The reflexive and transitive closure of $\lessdot $ is a partial order on $\NCSP(X:k)$
with unique minimal element $M_{\min}(X:k)$.
\end{proposition}

\begin{proof}
Let $M,N\in\NCSP(X:k)$
and define $\nb(M)$ to be the number of subsets $\{\{a,d\}, \{b,c\}\}\subset M$
with $0<a<b<c<d$ or $0<-d=a<b<c$. 
Now observe that (1) if $M\lessdot N$ then $\nb(M) < \nb(N)$,
 (2) if  $\nb(N) = 0$ then $N = M_{\min}(X:k)$, and
(3) if $\nb(N)>0$ then some $M$ exists with $M\lessdot N$.
\end{proof}

The material discussed above will be of use in 
the next three sections, where we explain the constructions involved in Theorems~\ref{main-thm} and \ref{main-thm2}
for each instance of $(G,K)$ in classical type. We have summarized this information in Table~\ref{shapes-summary-tbl} for convenience.

\def\hhline{\\ & & & \\ [-\vsp]\hline & & &  \\ [-\vsp]}
\def\gap{\\[-\vsp]&&&\\}
\begin{table}[h]
\begin{center}
{\small
\begin{tabular}{| l | l | l | l |}
\hline&&& \\[-\vsp]
Type & Parameters &  $\cM^G_K(z)$ for $z \in  \cI^G_K $   & $\precsim^G_K$
\hhline
AI &  $n \in \NN$  & $  \{\varnothing\} $  &  $\overset{0}\precsim$
\gap
AII  & $n \in \NN$ even & $  \{\varnothing\} $   & $\overset{0}\precapprox$
\gap
AIII & $n=p+q$ & $\{ M^+ :M \in  \NCSP(X:k)\}$ for $X=\Fix(w_0z)$ and $k=|p-q|$ & $\precsim_{\AIII}^{(p,q)}$ 
\hhline
BI &  $2n+1=p+q$ %and $k=\frac{|p-q|-1}{2}$ 
& $ \NCSPB(X:k)$ for $X=\Neg(z)$ and $k=\frac{|p-q|-1}{2}$ & $\overset{k}\precsim$
\hhline
CI & $n \in \NN$ & $  \NCSP(X)$ for $X=\Neg(z)$    & $\overset{0}\precsim$
\gap
CII  & $2n=p+q$ % with $p$ even %and $k=\frac{|p-q|}{2}$ 
&  $ \left\{ M^{\div} : M \in \NCSP(X:k)\right\}$  for $X=\Neg(z)$  and $k=\frac{|p-q|}{2}$ & $\overset{k}\precapprox$
% \\   & with $p$ even   &    & 
\hhline
DI  & $2n=p+q$ % with $p+n$ even %and $k=\frac{|p-q|}{2}$
& $  \left\{ M^{\div} : M \in \NCSP(X:k)\right\}$  for $X=\Neg\((t_0)^kz\)$ and $k=\frac{|p-q|}{2}$ & $\precsim_{\DI}^{(p,q)}$
\gap
DIII & $n\in\NN$ even  & $\left\{ M\in\NCSPD(X):  \triv{M} \equiv \ell_0(z) \modu 4)\right\}$  for $X=\Neg(z)$  & $\overset{0}\precapprox$
\gap
 & $n\in\NN$ odd  & $ \NCSPD(X)$  for $X=\Neg(t_0z)$ & $\overset{1}\precapprox$
\gap\hline
\end{tabular}}
\end{center}
\caption{Combinatorial data used in Theorems~\ref{main-thm} and \ref{main-thm2}.
The notations $\NCSP(X)$, $\NCSP(X:k)$,
$M^+$, and $M^\div$ are defined in Section~\ref{matching-sect}, while $\NCSPB(X:k)$ and $\NCSPD(X)$ are defined in Sections~\ref{BI-sect} and \ref{D34-sect}.
The partial orders in the last column are defined in Sections~\ref{word-rel-sect}, \ref{A3-sect} and \ref{D12-sect}.
}\label{shapes-summary-tbl} \end{table}

\section{Shapes in type A}\label{A-shape-sect}

We start by defining all terms appearing in Theorems~\ref{main-thm} and \ref{main-thm2} for the cases when $G=\GL(n)$  and $W=S_{n}$.
In this setting, our main theorems are essentially equivalent to results in \cite{BurksPawlowski,CJW,HMP2},
although in type AIII some explanation is required to derive this; see Section~\ref{proof-sect1}.

\subsection{Types AI and AII}\label{tAI-sect}

Fix $z \in \I(S_{n})$ and $w \in S_{n}$.
Recall that 
$
\cI_{\AI}^{(n)} = \I(S_{n})
$
and
$
\cI_{\AII}^{(n)} = \Ifpf(S_{n}) $ when $n$ is even.
In types $\AI$ and $\AII$ the set of matchings associated to $z$ and the shape operator are trivial:

\begin{definition}[Shapes for types $\AI$ and $\AII$] For $z$ and $w$ as above:
\bei
\item (\emph{Possible shapes}) Let $\cM_{\AI}^{(n)}(z)=\cM_{\AII}^{(n)}(z)=\{\varnothing\}$.

\item (\emph{Shape operator})
Let $\sh_{\AI}^{(n)}(w) =\sh_{\AII}^{(n)}(w) =\varnothing$.

\eei
\end{definition}

Recall that we define $\Cyc(z) = \{ (a,b) \in [n]\times[n] : a \leq b = z(a)\}$.

\begin{definition}[Generation for types $\AI$ and $\AII$]
\label{genA12-def}
 \
\bei
\item (\emph{Generators})
Let $ \bot_{\AI}^{(n)}(z,\varnothing) = [[\Cyc(z)]]_{\des} $ and $ \bot_{\AII}^{(n)}(z,\varnothing) = [[\Cyc(z)]]_{\asc} $.

\item (\emph{Order})
Define $\precsim_{\AI}^{(n)}$ and $\precsim_{\AII}^{(n)}$ to be respectively $\overset{0}\precsim$ and $\overset{0}\precapprox$ restricted to $S_{n}$.

\eei
\end{definition}

\subsection{Type AIII}\label{A3-sect}

Fix  $p,q \in \NN$ with $n=p+q$ and set $k=|p-q|$.
Recall that $w_0 =n\cdots 321 \in S_{n}$.
Choose 
$z \in \I_{\AIII}^{(p,q)}$ so that $z\in \I_\ast(S_{n})$ 
is any element for which $w_0 z \in \I(S_{n})$ has at least $k$ fixed points in $[n]$.
Suppose 
\[\gamma\in \Gamma_{\AIII}^{(p,q)}=\{\text{standard $(p,q)$-clans}\}
\quad\text{has  $z=w_0\pi_\gamma = \psi_{\AIII}^{(p,q)}(\gamma)$}.\]

\begin{definition}[Shapes for type $\AIII$]
\label{AIII-shape-def}
 Let $X=\Fix(w_0z) = \APoints(\gamma)$. Then:
\bei
\item (\emph{Possible shapes})
Let $\cM_{\AIII}^{(p,q)}(z)=\left\{ M^+ : M \in \NCSP(X: { k})\right\}$ be the set of noncrossing matchings of $X$ with exactly $k$ trivial blocks, such that no trivial block $\{b\}$ has $a<b<c$ for any nontrivial block $\{a,c\}$.

\item (\emph{Shape operator})
If 
$w= b_1b_2\cdots b_j c_1c_2\cdots c_k a_j \cdots a_2a_1 \in S_{n}$ in one-line notation
then let \[\sh_{\AIII}^{(p,q)}(w)= \Bigl\{
\{a_i,b_i\} : i \in [j]\text{ with }a_i<b_i
\Bigr\}\sqcup \Bigl\{
\{c_1\}, \{c_2\}, \dots,\{c_k\}
\Bigr\}.\]

\eei
\end{definition}

If $M$ is a matching on the set of fixed points 
for an involution $y \in \I(S_n)$, then let
$\Pairs(y,M)$ be the set of pairs $(a,b) \in [n]\times[n]$ with  
either
$
a<b$ and $\{a,b\}\in M$, or $ a> b=y(a)$.

\begin{definition}[Generation for type $\AIII$]\label{AIII-shape-def2}   Fix $M \in \cM_{\AIII}^{(p,q)}(z)$.
\bei
\item (\emph{Generators})
Let
$\bot_{\AIII}^{(p,q)}(z,M) = b_1b_2\cdots b_j c_1c_2\cdots c_k a_j\cdots a_2a_1 \in S_{n}$ where
\[
\ba
\mathrm{Triv}(M) &= \{c_1<c_2<\dots<c_k\}\quand
\\
\Pairs(w_0z,M) &= 
\{ (a_1,b_1),(a_2,b_2),\dots,(a_j,b_j)\}\text{ with } b_1<b_2<\dots<b_j.
\ea
\]

\item (\emph{Order})
Define $\preceq_{\AIII}^{(p,q)}$ to be the transitive closure of the word relation with \[u_1  B_1 B_2  v   A_2 A_1  u_2  \mathbin{\preceq_{\AIII}^{(p,q)}}  u_1  B_2 B_1  v  A_1  A_2  u_2\] whenever 
$u_1$, $u_2$, and $v$ are subwords with $\ell(u_1)=\ell(u_2)$ and $\ell(v) \geq k=  |p-q|$ and 
$A_1, A_2, B_1, B_2\in \ZZ$ are letters satisfying
$A_1<A_2$ and $B_1<B_2$.

\eei
\end{definition}

\begin{remark}\label{AIII-rmk}
If  $z = \psi_{\AIII}^{(p,q)}(\gamma)$ for a clan of the form $\gamma =  (T_+,T_-,N)$
then $\Pairs(w_0z,M)$ consists of exactly the pairs $(a,b) \in [n]\times [n]$
with
 $\{a<b\} \in M$ or $\{b<a\}\in N$.
\end{remark}

\subsection{Proofs of the main theorems}\label{proof-sect1}

Here, we explain the proofs of Theorems~\ref{main-thm} and \ref{main-thm2} in type A  when $W=S_{n}$.
In this well-studied case, most of what we want to show can be deduced from results in \cite{BurksPawlowski,CJW,HMP2}.

\begin{proof}[Proof of Theorems~\ref{main-thm} and \ref{main-thm2}  in types $\AI$ and $\AII$]
In these types, Theorem~\ref{main-thm} is trivial and Theorem~\ref{main-thm2} is equivalent to  
\cite[Thms.~6.10 and 6.22]{HMP2} and \cite[Props.~6.14 and 6.24]{HMP2}.
We highlight some notational differences: the sets 
 $\cA_{\AI}^{(n)}(z)$ and $\cA_{\AII}^{(n)}(z)$ give $\{w^{-1} : w \in \mathcal{A}(z)\}$ and $\{w^{-1} : w \in \mathcal{A}_{\mathsf{FPF}}(z)\}$ in \cite{HMP2},
while $\bot_{\AI}^{(n)}(z,\varnothing)$ and $\bot_{\AII}^{(n)}(z,\varnothing)$  
are the inverses of $\hat 1(z)$ and $\hat 1_{\mathsf{FPF}}(z)$ in \cite{HMP2}.
 \end{proof}

A few more details are needed to address type $\AIII$.
For the rest of this section 
fix $p,q\in \NN$ with $p+q=n$ and set $k=|p-q|$. Choose 
a twisted involution $z \in \I_{\AIII}^{(p,q)}$.

\begin{lemma}\label{AIII-lem0}
Each $M\in\cM_{\AIII}^{(p,q)}(z)$ is  $\gamma$-aligned
for some $\gamma \in \Gamma_{\AIII}^{(p,q)}$ with $ \psi_{\AIII}^{(p,q)}(\gamma)=z$.
\end{lemma}

\begin{proof}
Given $M$, the desired   $\gamma = (T_+,T_-,N)$ is constructed as follows.
Define $N$ to be the perfect matching whose blocks are the $2$-cycles of $w_0 z$.
For each $\{a,b\} \in M$ with $0<a<b$, add $a$ to $T_+$ and $b$ to $T_-$.
Then add all $|p-q|$ elements of $\Triv{M}$ to $T_+$ when $p>q$ or to $T_-$ when $p<q$.
\end{proof}

\begin{lemma}\label{AIII-lem1}
Suppose $w \in  \cA_{\AIII}^{(p,q)}(z)$ has one-line representation
\[
 w=b_1b_2\cdots b_j c_1c_2\cdots c_k a_j\cdots a_2a_1\quad\text{where }j = \tfrac{n-k}{2}= \min\{p,q\}.
 \]
Then $a_i > a_{i+1}$ whenever $i \in [j-1]$ is such that $b_i > b_{i+1}$. Also, it holds that $c_1<c_2<\dots<c_k$.
\end{lemma}

\begin{proof}
Recall that  $\cA_{\AIII}^{(p,q)}(z) = \cA_\ast(\omega,z)$
for $\omega=\omega_k^{n} \in S_{n}$   as in \eqref{omega-eq}.
Notice also that
\[
t_{i} \circ \omega \circ t_{n-i} = t_{n-i} \circ \omega \circ t_i  =
\begin{cases}
\omega &\text{if }j < i <j+k\\
 \omega \cdot t_i \cdot t_{n-i} &\text{if $1\leq i<j$ or $j+k<i\leq n-1$}.
 \end{cases} 
\]
Using this identity, it follows that $c_1<c_2<\dots<c_k$ since otherwise 
for some index $j<i<j+k$ the permutation $\pi = wt_i$ would satisfy $\ell(\pi) = \ell(w)-1$ and $w = \pi \circ t_i$,
in which case
\[z = w^\ast \circ \omega \circ w^{-1} = \pi^\ast\circ t_{n-i} \circ \omega \circ t_i \circ \pi^{-1}  = \pi^\ast \circ \omega \circ \pi^{-1} ,
\]
contradicting the minimality of $\ell(w)$ in the definition of $ \cA_\ast(\omega,z)$.
Similarly, if there is an index $i \in [j-1]$ with $a_i < a_{i+1}$ and  $b_i > b_{i+1}$,
then
the permutation $\pi = wt_i$ would satisfy $\ell(\pi) = \ell(w)-1$, $w = \pi \circ t_i$, and $\pi = \pi \circ t_{n-i}$,
in which case the computation
\[\ba
z = w^\ast \circ \omega \circ w^{-1} &= \pi^\ast\circ t_{n-i} \circ \omega \circ t_i \circ \pi^{-1}
= \pi^\ast\circ t_{i} \circ \omega \circ t_{n-i} \circ \pi^{-1}
\\&  =
(\pi\circ t_{n-i})^\ast \circ \omega \circ (\pi\circ t_{n-i})^{-1}  = \pi^\ast \circ \omega \circ \pi^{-1} 
\ea
\] would again
contradict the minimality of $\ell(w)$.
\end{proof}

We now choose a $(p,q)$-clan
$\gamma \in \Gamma_{\AIII}^{(p,q)}$ with
$z=w_0\pi_\gamma=\psi_{\AIII}^{(p,q)}(\gamma)$.

\begin{lemma}[\cite{BurksPawlowski}] \label{AIII-lem2}
Fix $u,v \in S_{n}$ with $u \preceq_{\AIII}^{(p,q)} v$. 
Then $u \in \cW_{\AIII}^{(p,q)}(\gamma)$ if and only if $v \in \cW_{\AIII}^{(p,q)}(\gamma)$.
\end{lemma}

\begin{proof}
We may assume that the relation is a cover.
Then $\ell(v) = \ell(u)$ and $v=u t_it_{n-i} $ for some $1 \leq i  <\min\{p,q\}$,
so the lemma follows directly from \cite[Thm.~29]{BurksPawlowski}
after noting that the set $\mathcal{A}(\gamma)$ in \cite{BurksPawlowski} is   formed by inverting all elements of $\cW_{\AIII}^{(p,q)}(\gamma)$.
% $\left\{ w^{-1} : w\in \cW_{\AIII}^{(p,q)}(\gamma)\right\}$.
\end{proof}

\begin{proof}[Proof of Theorems~\ref{main-thm} and \ref{main-thm2}  in type $\AIII$]
The relation $\preceq_{\AIII}^{(p,q)}$ is a graded partial order on $S_{n}$, 
since we can use as a rank function the number of 
inversions among the first $\min\{p,q\}$ letters in the one-line representation of a permutation.

Lemmas~\ref{AIII-lem1} and \ref{AIII-lem2} imply that $\cW_{\AIII}^{(p,q)}(\gamma)$ is a disjoint union of upper sets in $S_{n}$ under the order $\preceq_{\AIII}^{(p,q)}$,
where each such upper set contains a unique minimal element of the form 
$
w_{\min}= b_1b_2\cdots b_j c_1c_2\cdots c_k a_j \cdots a_2a_1
$
where $b_1<b_2<\dots<b_j$ and $c_1<c_2<\dots<c_k$.
Can, Joyce, and Wyser give an explicit algorithm in \cite[Thm.~2.22]{CJW} to construct   $\cW_{\AIII}^{(p,q)}(\gamma)$.
Inspecting this algorithm alongside Remark~\ref{AIII-rmk} shows that the permutations that arise as $w_{\min}\in\cW_{\AIII}^{(p,q)}(\gamma)$
are exactly the elements $\bot_{\AIII}^{(p,q)}(z,M)$ for  shapes $M\in \cM_{\AIII}^{(p,q)}(z)$ that are $\gamma$-aligned.
Thus 
\[\textstyle
\cW_{\AIII}^{(p,q)}(\gamma) = \bigsqcup_{\text{$\gamma$-aligned }M \in \cM_{\AIII}^{(p,q)}(z)}  \left\{ w \in S_{n} : \bot_{\AIII}^{(p,q)}(z,M) \preceq_{\AIII}^{(p,q)} w\right\}.
\]
By Lemma~\ref{AIII-lem0}, we
then also have
\[\textstyle
\cA_{\AIII}^{(p,q)}(z) = \bigsqcup_{M \in \cM_{\AIII}^{(p,q)}(z)}  \left\{ w \in S_{n} : \bot_{\AIII}^{(p,q)}(z,M) \preceq_{\AIII}^{(p,q)} w\right\}.
\]
The sets in these unions are the same as  $\left\{ w \in \cA_{\AIII}^{(p,q)}(z) : \sh_{\AIII}^{(p,q)}(w) = M\right\}$ since $ \sh_{\AIII}^{(p,q)}$ is constant on any $\preceq_{\AIII}^{(p,q)} $-interval and
recovers $M$ from $ \bot_{\AIII}^{(p,q)}(z,M)$.
Theorems~\ref{main-thm} and \ref{main-thm2}  therefore follow.
\end{proof}

 \section{Shapes in types B and C}\label{BC-shape-sect}
 
This section explains the constructions used in Theorems~\ref{main-thm} and \ref{main-thm2} 
  when $W=\W_{n}$.
This material is largely expository as in these types 
our   theorems follow from earlier results in \cite{HM}.

\subsection{Cycle sets}
Let $z=z^{-1} \in \I(\W_n) $. Define
$\CycPM(z)$ to be the set of pairs $(a,z(a))$ for $a \in [\pm n]$ with either 
\[
0 < a \leq z(a)
\quord
-z(a) < a < 0 < z(a).
\]
The nontrivial cycles of $z$
  have the form $(-a,a)$ with $a>0$ or come in pairs $(-b,-a)(a,b)$ or $(-b,a)(-a,b)$ with $0<a<b$.
The set $\CycPM(z) $ only retains the cycles $( a,b)$ and $(-a,b)$
that arise in the second two kinds of pairings, along with $(a,a)$ for each $0<a \in \Fix(z)$. Thus 
\be
\left|\left \{ (a,b) \in \CycPM(z) : a< 0\right\}\right| = \tfrac{\ell_0(z) - \neg(z)}{2}.
\ee

\begin{example}
If $n=6$ and $z = 1\overline{2} 6\overline{54}3 = (-2,2)(-6,-3)(3,6)(-5,4)(-4,5) \in \I(\W_n)$
then
\[
\CycPM(z) = \{ (3,6),\hs (-4,5),\hs (1,1)\}.
\]
\end{example}

Given any partition $M$ of a set of integers, we additionally define
\be\label{cyc-z-m-eq}
\CycPM(z,M) = \CycPM(z)\sqcup \{ (-b,a) : \{a,b\} \in M \text{ with }0<a<b\}.
\ee
Every  $(a,b) \in  \CycPM(z,M)$ has $a\leq b$ but we recover
$\CycPM(z) =\{ (a,b) \in  \CycPM(z,M) : |a| \leq b\}$.

\subsection{Type BI}\label{BI-sect}

Fix  $p,q \in \NN$ with $p+q=2n+1$ and set $k=\tfrac{|p-q|-1}{2}$.
Suppose 
$z \in \I_{\BI}^{(p,q)} $, so that $z \in\cI(\W_n) $ is any involution with  $\neg(z)\geq k$.
Choose a clan
\[\gamma\in \Gamma_{\BI}^{(p,q)}=\{\text{symmetric $(p,q)$-clans}\}\quad\text{and assume that  $z=\overline{\sigma_\gamma} = \psi_{\BI}^{(p,q)}(\gamma)$.}\]
Finally let  $X=\Neg(z) = \Points(\gamma)$.

We require a variation of $ \NCSP(X)$ to describe the possible shapes for this type.
Suppose $M \in \NCSP(X)$ has symmetric blocks $\{\pm c_1\}$, $\{ \pm c_2\}$, \dots, $\{ \pm c_l\}$ where $0=c_0<c_1<c_2<\dots<c_l$
and $k\leq l$. Form a matching  $\regroupB^k(M)$
on $\{0\}\sqcup X\sqcup -X$ from $M$ by the following steps:
\ben
\item[(1)] Replace $\{\pm c_i\}$ for $i \in [l-k]$ by the asymmetric block $\{-c_{i-1},c_i\}$.
\item[(2)] Divide $\{\pm c_i\}$ for $i>l-k$ into two trivial blocks $\{-c_i\}$ and $\{c_i\}$.
\item[(3)] Add the trivial block $\{-c_{l-k}\}$.
\een
For example,
applying $
\regroupB^1$ to 
$\Bigl\{
\{-1,1\}, \{-2,2\}, \{-5,5\},\{-6,6\}, \{-4,-3\}, \{3,4\}
\Bigr\}$
gives
\[
\Bigl\{
\{0,1\}, \{-1,2\}, \{-2,5\}, \{-6\}, \{6\}, \{-4,-3\}, \{3,4\}, \{-5\}
\Bigr\}.
\]

\begin{remark}
Notice that we have $\Triv{\regroupB^k(M)} = \Triv{M}$. Moreover, 
as $\gamma$ is symmetric, $\regroupB^k(M)$ is $\gamma$-aligned if and only if
 $ S_+(\gamma)$ and $S_-(\gamma)$ never 
contain two consecutive elements of the set
$\{-c_{l-k},\dots,-c_2,-c_1,0,c_1,c_2,\dots,c_{l-k}\}$
or of any block $\{a,b\} \in M$ with $0<a<b$.
\end{remark}

As $\regroupB^k(M)$ 
is still noncrossing but no longer always symmetric and perfect, we 
define
\be
\NCSPB(X:k) = \left\{ \regroupB^k(M) : M \in \NCSP(X:l)\text{ for any }l\geq k\right\}.
\ee
Below,   let $w \in \W_{n}$ be an arbitrary signed permutation.

\begin{definition}[Shapes for type $\BI$]
\label{BI-shape-def}
 Continue to let $X=\Neg(z) = \Points(\gamma)$. Then:
\bei
\item (\emph{Possible shapes})
Let $\cM_{\BI}^{(p,q)}(z)= \NCSPB(X:k)$. 
 
\item (\emph{Shape operator})
Suppose $j \in \NN$ and $0=a_0<a_1<a_2<\dots<a_j$ are such that
\[
\{ a_1,a_2,\dots,a_j \} = \{-c: 0>c\in  \NRes( w_{k+1}w_{k+2}\cdots w_n)\}.\]
Then let $\sh_{\BI}^{(p,q)}(w)$ 
be the
matching whose nontrivial blocks are
\bei
\item[]  $\{a,-b\}$ and $\{-a,b\}$ for each $(a,b) \in \NDes( w_{k+1}w_{k+2}\cdots w_n)$ with $0<a<-b$, and

\item[]  $\{-a_{i-1}, a_i\}$ for each $i\in[j]$,
\eei
and whose trivial blocks are
$\{-a_j\}$, $ \{w_1\}$, $ \{-w_1\}$, $ \{w_2\}$, $ \{-w_2\}$, \dots, $ \{w_k\}$, $ \{-w_k\}$.

\eei
\end{definition}

\begin{definition}[Generation for type $\BI$]  Fix $M \in \cM_{\BI}^{(p,q)}(z)$.
\bei
\item (\emph{Generators}) 
If $\Triv{M} = \{ a_1<\dots<a_j<c_1<\dots<c_k\}$ then 
let  $\bot_{\BI}^{(p,q)}(z,M) \in \W_n$ be the signed permutation
whose one-line representation is the concatenated word $ uv $ where 
\[
u = c_1c_2\cdots c_k \overline{a_j\cdots a_2a_1}
\quand  v = [[\CycPM(z,M)]]_{\des}.
\]

\item (\emph{Order})
Define $\precsim_{\BI}^{(p,q)}$ to be the partial order  $\overset{k}\precsim$ restricted to $\W_{n}$.

\eei
\end{definition}

\subsection{Type CI}\label{CI-sect}

Choose elements $w \in \W_{n}$ 
and
$z \in \I_{\CI}^{(n)} =\cI(\W_n)$.
Suppose 
\[\gamma \in \Gamma_{\CI}^{(n)}=\{\text{skew-symmetric $(n,n)$-clans}\}\quad\text{has}\quad z=\overline{\sigma_\gamma} = \psi_{\CI}^{(n)}(\gamma).\]

\begin{definition}[Shapes for type $\CI$] Let $X=\Neg(z)=\Points(\gamma)$. Then:
\bei
\item (\emph{Possible shapes})
Let $\cM_{\CI}^{(n)}(z)=\NCSP(X)$.

\item (\emph{Shape operator})
Let $\sh_{\CI}^{(n)}(w)$    be the perfect matching  containing
\bei
\item[]  $\{a,-b\}$ and $\{-a,b\}$ for each $(a,b) \in \NDes(w)$ with $0<a<-b$, and
\item[]  $\{-a,a\}$ for each $a \in \NRes(w)$ with $a<0$.
\eei

\eei
\end{definition}

\begin{definition}[Generation for type $\CI$]
\label{CI-gen-def}
  Fix $M \in \cM_{\CI}^{(n)}(z)$.
\bei
\item (\emph{Generators})
If  $\Triv{M} = \{ c_1>c_2>\dots>c_l\}$ then
let  $\bot_{\CI}^{(n)}(z,M) =uv\in \W_n$ 
where 
\[
u =   \overline{c_1c_2\cdots  c_l}\quand v = [[\CycPM(z,M)]]_{\des}.
\]

\item (\emph{Order})
Define ${\precsim_{\CI}^{(n)}} $ to be the partial order  ${\overset{0}\precsim}$ restricted to $\W_{n}$.

\eei
\end{definition}

\subsection{Type CII}

Fix even integers $p,q \in 2\NN$ with $p+q=2n$ and set $k=\tfrac{|p-q|}{2}$.
Choose any
\[z \in \I_{\CII}^{(p,q)} = \Bigl\{ y \in \Ifpf(\W_n): \neg(y)\geq k\Bigr\}. 
\]
Suppose $\gamma\in \Gamma_{\CII}^{(p,q)} =\{\text{strict symmetric $(p,q)$-clans}\}$ has $z=\overline{\sigma_\gamma} = \psi_{\CII}^{(p,q)}(\gamma)$.

 \begin{definition}[Shapes for type $\CII$] Let $X=\Neg(z)=\Points(\gamma)$. Then:
\bei
\item (\emph{Possible shapes}) 
Let $\cM_{\CII}^{(p,q)}(z)
 = \left\{ M^\div : M \in \NCSP(X:k)\right\}
$.

\item (\emph{Shape operator}) 
If $w \in \W_{n}$ has one-line representation 
\[w= a_1a_2\cdots a_k b_1c_1b_2c_2\cdots b_jc_j\quad\text{where }k = \tfrac{|p-q|}{2}\text{ and }j=\tfrac{n-k}{2}\]
then let $\sh_{\CII}^{(p,q)}(w)$    be the matching whose blocks are
\bei
\item[]  $\{-a_i\}$ and $\{a_i\}$ for each $i \in [k]$, and 
\item[] $\{b_i,-c_i\}$ and $\{c_i,-b_i\}$ for each $i \in [j]$ with $0<c_i<-b_i$.
\eei

\eei
\end{definition}

This definition uses the notation $M^\div$ for matchings $M \in \NCSP(X)$ from Section~\ref{matching-sect}.

 \begin{definition}[Generation for type $\CII$] Fix $M \in \cM_{\CII}^{(p,q)}(z)$.
 \bei
\item (\emph{Generators}) 
If $\Triv{M} = \{ a_1<a_2<\dots<a_k\}$ then let $\bot_{\CII}^{(p,q)}(z,M)=uv \in \W_n$
where 
\[
u =   a_1a_2\cdots  a_k\quand v = [[\CycPM(z,M)]]_{\asc}.
\]

\item (\emph{Order}) 
Define $\precsim_{\CII}^{(p,q)} $ to be the partial order  ${\overset{k}\precapprox}$ restricted to $\W_{n}$.

\eei
\end{definition}

\subsection{Proofs of the main theorems}\label{proof-sect2}

When $W=\W_n$, 
Theorems~\ref{main-thm} and \ref{main-thm2} follow from results in \cite{HM}
that characterize the sets formed by inverting all elements of 
$\cEABrion(z)$ and $\cABrion(\gamma)$.
We explain the details of this translation below.

\begin{proof}[Proof of Theorems~\ref{main-thm} and \ref{main-thm2}  in types $\BI$, $\CI$, and $\CII$]
Theorem~\ref{main-thm} follows from
\bei
\item \cite[Thm.~8.7]{HM} in type $\BI$, using \cite[Lem.~7.6]{HM} to define $\mathcal{A}_k(z,M)$; 
\item \cite[Thm.~8.2]{HM} in type $\CI$,
and 
\item \cite[Thm.~8.4]{HM} in type $\CII$,  using \cite[Lem.~7.10]{HM} to define $\mathcal{A}_k^{\mathsf{FPF}}(z,M)$.
\eei 
The formula for $\cEABrion(z,M)$ in Theorem~\ref{main-thm2} is equivalent to  \cite[Cor.~5.9, Prop.~7.8, and Prop.~7.13]{HM}
for types $\CI$, $\BI$, and $\CII$, respectively.
It remains to justify why $\precsim_{\BI}^{(p,q)} $, $\precsim_{\CI}^{(n)} $, and $\precsim_{\CII}^{(p,q)} $ make  
$\cA_{\BI}^{(p,q)}(z)$, $\cA_{\CI}^{(n)}(z)$, and $\cA_{\CII}^{(p,q)}(z)$ into graded posets.
This is clear in type $\CII$ since ${\overset{k}\precapprox}$ is always a graded partial order.
In types $\BI$ and $\CI$ the desired  property holds by \cite[Cor.~6.3 and Lem.~7.6]{HM}.

In view of Table~\ref{extended-brion-tbl},
we must check in types BI and CI that there exists
a map $e_z : \cM^G_K(z)\to \NN$ as in Theorem~\ref{main-thm}.
It suffices to show that if $w \in \cA^G_K(z)$ for $z \in \cI^G_K$
in these two types then $\ell_0(w)$ can be recovered from $z$ and $M=\sh^G_K(w)$. This is straightforward when $w=\bot^G_K(z,M)$ 
and therefore follows for any relevant $w$ since the relation $\precsim^G_K$ for types BI and CI leaves $\ell_0(w)$ unchanged.
\end{proof}
 
 \section{Shapes in type D}\label{D-shape-sect}
 
In this section, finally, we define all of the notation used in Theorems~\ref{main-thm} and \ref{main-thm2} 
  when $W=\WD_{n}$.
We will postpone the derivation of Theorems~\ref{main-thm} and \ref{main-thm2} in these types
 to Section~\ref{D-proof-sect}.

\subsection{Type DI}\label{D12-sect}

Fix $p,q\in \NN$ with $2n=p+q$ and let $k=\frac{|p-q|}{2}$.  
  Let
\[
z\in \cI_{\DI}^{(p,q)} = 
\begin{cases}
 \Bigl\{ y \in \cI(\WD_n): \neg(y)\geq k \Bigr\} &\text{if $k$ is even} \\[-10pt]\\
\Bigl\{ y \in \cI_\diamond(\WD_n): \neg(t_0y)\geq k \Bigr\} &\text{if $k$ is odd}.
\end{cases}
\]
Define $\varepsilon_k = (t_0)^k \in \W_n$ so that 
 $\varepsilon_k z \in \I(\W_n)$.
Choose a symmetric $(p,q)$-clan 
\[\gamma
\in \Gamma_{\DI}^{(p,q)} 
 =  \{\text{symmetric $(p,q)$-clans}\}
\quad\text{with $\varepsilon_k z = \overline{\sigma_\gamma}$.}\]
Then we have $z = \psi_{\DI}^{(p,q)}(\gamma)$.
Finally, let $w \in \WD_n$ be any even-signed permutation.

 \begin{definition}[Shapes for type $\DI$] 
 \label{D12-def}
  Let $X=\Neg(\varepsilon_k z)=\Points(\gamma)$. Then:
\bei 
\item (\emph{Possible shapes}) 
Let $\cM_{\DI}^{(p,q)}(z)=  \left\{ M^\div : M \in \NCSP(X:k)\right\}$.

\item (\emph{Shape operator}) 
Let $\sh_{\DI}^{(p,q)}(w)$    be the matching whose blocks are
\bei
\item[]  $\{|a|,-b\}$ and $\{-|a|,b\}$ for each $(a,b) \in \NDes(w_{k+1}w_{k+2}\cdots w_n)$ with $|a|<-b$, and
\item[]  $\{-a\}$ and $\{a\}$ for each $a \in \{w_1,w_2,\dots,w_k\}$.
\eei

\eei
\end{definition}

If
$a=a_1a_2\cdots a_n$ is any word then we let 
 $\tilde a = \overline{a_1}a_2\cdots a_n$
 and define
 $\langle a\rangle_{\es}$ to
be whichever of $a$ or $\tilde a$ 
has an even number of negative letters. 
If $a$ is an element of $\W_n$ in one-line notation, then we interpret $\langle a\rangle_{\es} $ as an element of $ \WD_n$.

  \begin{definition}[Generation for type $\DI$] 
  \label{D12-gen-def}
Fix $M \in \cM_{\DI}^{(p,q)}(z)$.
\bei
\item (\emph{Generators}) 
If $\Triv{M} = \left\{  c_1 < c_2  < \dots < c_{k} \right\}$ then let 
$\bot_{\DI}^{(p,q)}(z,M) =\langle uv\rangle_{\es}\in \WD_n$
for 
\[
u =   c_1c_2\cdots c_{k} \quand v = [[\CycPM( \varepsilon_kz,M)]]_{\des}.
\]

\item (\emph{Order}) 
Let $\precsim_{\DI}^{(n,n)}$ be
the transitive closure of the relation on words
that has
\[
u B C A v  \precsim_{\DI}^{(n,n)}  u C A B  v
\quand
\overline{B} C A  v  \precsim_{\DI}^{(n,n)}  \overline{C}  A B  v
\]
if $A,B,C \in \ZZ$ are integers with $A<B<C$ and $u$ and $v$ are any subwords.
When  $p\neq q$ so that $k> 0$ 
define ${\precsim_{\DI}^{(p,q)}}$  to be the transitive closure of 
$\overset{k}\precsim$ and the relation $<$  
with
\[
  uv <  \tilde u \tilde v  = \overline{u_1}u_2u_3\cdots u_k \overline{v_1}v_2v_3\cdots 
\]
whenever $u$ and $v$ are nonempty subwords with $\ell(u) = k$ and $0<v_1<|u_1|$.

\eei
\end{definition}

\subsection{Type DIII}\label{D34-sect}

Again define $\varepsilon_k =(t_0)^k \in \W_n$ and fix some
\[
z \in \cI_{\DIII}^{(n)} =
\begin{cases}
 \Bigl\{ y \in \Ifpf(\WD_n) : \neg(y)>0 \text{ or }  \ell_0(y) \in 4\NN\Bigr\} &\text{if $n$ is even} \\[-10pt]\\
  \Bigl\{ y\in  \WD_n : t_0y \in \Ifpf(\W_n)\Bigr\} &\text{if $n$ is odd}.
\end{cases}
\]
 Then $\varepsilon_n z \in \Ifpf(\W_n).
$
Fix 
$\gamma
\in \Gamma_{\DIII}^{(n)}  =
\{\text{even-strict skew-symmetric $(n,n)$-clans}\}$
with $\varepsilon_nz =  \overline{\sigma_\gamma}$
so that $z = \psi_{\DIII}^{(n)}(\gamma)$.
 Finally, 
 let $X = \Neg(\varepsilon_nz) = \Points(\gamma)$. 
 
 As in type BI, we need to introduce a  variation of the set $\NCSP(X)$. Specifically, 
 suppose $M \in \NCSP(X)$ has symmetric blocks $\{\pm c_1\}$, $\{ \pm c_2\}$, \dots, $\{ \pm c_l\}$ where $0<c_1<c_2<\dots<c_l$.
Set $k=l$ when $l$ is even and $k=l-1$ when $l$ is odd.
Then define 
\be
\regroupD(M)= \Bigl(M \setminus \Bigl\{ \{\pm c_1\},\dots,\{\pm c_k\}\Bigr\}\Bigr) \sqcup \Bigl\{ 
\pm \{ c_1, -c_2\},
%\pm \{ c_3, -c_4\},
\dots,
\pm \{ c_{k-1}, -c_k\}
\Bigr\}.
\ee
In words, we replace each successive pair of symmetric blocks $\{\pm a\}$, $\{\pm b\}$ by $\{|a|,-|b|\}$, $\{-|a|,|b|\}$,
ignoring the outermost symmetric block if $l$ is odd. For example
\[
\regroupD\Bigl(\Bigl\{
\{\pm1\}, \{\pm2\}, \{\pm5\}, \{-4,-3\}, \{3,4\}
\Bigr\}\Bigr) = 
\Bigl\{
\{-1,2\}, \{-2,1\}, \{-5,5\}, \{-4,-3\}, \{3,4\}
\Bigr\}.
\]
As $\regroupD$ gives a symmetric perfect matching that is not always noncrossing, we write  
\be
\NCSPD(X) = \left\{ \regroupD(M) : M \in \NCSP(X) \right\}.
\ee
Notice that $\triv{\regroupD(M)} = \triv{M}$ for $M \in \NCSP(X)$.
Also, if $M \in\NCSPD(X)$ then
\be\label{ncspd-eq-eq}
\text{$n$ even}
\quad\Leftrightarrow\quad
\text{$|X|$ even}
\quad\Leftrightarrow\quad
\text{$\triv{M}$ even}.
\ee

 \begin{definition}[Shapes for type $\DIII$] 
 \label{D34-def}
 Let $X = \Neg(\varepsilon_nz) = \Points(\gamma)$. 
 Then:
\bei
\item (\emph{Possible shapes}) 
Define
\[
\cM_{\DIII}^{(n)}(z)=
\begin{cases}
\left\{ M\in\NCSPD(X):  \triv{M} \equiv \ell_0(z) \modu 4)\right\} &\text{if $n$ is even}\\[-10pt]
\\
\NCSPD(X) &\text{if $n$ is odd}.
\end{cases}
\] 

\item (\emph{Shape operator}) 
If $w \in \WD_n$ has
one-line representation
\[w= \begin{cases} 
b_1c_1b_2c_2\cdots b_jc_j &\text{when $n=2j$ is even} \\ 
c_0b_1c_1b_2c_2\cdots b_jc_j &\text{when $n=2j+1$ is odd},
\end{cases}
\]
then let $\sh_{\DIII}^{(n)}(w)$  
be the perfect matching with blocks
\bei
\item[] $\{-c_0,c_0\}$ when $n$ is odd,
\item[]  $\{b_i,-c_i\}$ and $\{-b_i,c_i\}$ for each $i \in [j]$ with $c_i <-b_i$.

\eei

\eei
\end{definition}

   \begin{definition}[Generation for type $\DIII$] 
   \label{D34-def2}
   Fix $M \in \cM_{\DIII}^{(n)}(z)$.
\bei
\item (\emph{Generators}) 
If $\Triv{M} =  \{ a_1<\dots<a_l\}$ then
 let 
 $\bot_{\DIII}^{(n)}(z,M) = \langle uv\rangle_{\es} \in \WD_n  $
  for 
\be\label{uv-eq}
u =   \overline{a_l\cdots a_2a_1}
\quand  v = [[\CycPM( \varepsilon_nz,M)]]_{\asc}.
\ee

\item (\emph{Order}) 
Let $\precsim_{\DIII}^{(n)}$ be the restriction to $\WD_n$ of ${\overset{0}\precapprox}$ if $n$ is even or  ${\overset{1}\precapprox}$ if $n$ is odd.
\eei
\end{definition}

\begin{proposition}\label{d2-uv-prop}
If $n$ is even and $M \in \cM_{\DIII}^{(n)}(z)$ then
$\bot_{\DIII}^{(n)}(z,M) = uv$ for $u$, $v$ as in \eqref{uv-eq}.
\end{proposition}

\begin{proof}
Assume $n$ is even and $M \in \cM_{\DIII}^{(n)}(z)$.
Then, as noted in \eqref{ncspd-eq-eq} before Definition~\ref{D34-def},
the length $\ell(u)= |\Triv{M}|=\triv{M} $ is even,
so it suffices to show that $ v = [[\CycPM( z,M)]]_{\asc}$ has an even number of negative letters.
Consulting \eqref{cyc-z-m-eq}, we see that $v$
has one negative letter for each block $\{a,b\} \in M$ with $0<a<b$ (say there are $B$ of these) and one negative letter for each pair of cycles $(-b,a)(-a,b)$ in $z$ with $0<a<b$ (say there are $C$ of these).  
Then we have 
\[\neg(z) = |X|=2B+\triv{M}\quand \ell_0(z) = \neg(z) + 2C\]
so $B+C 
=\tfrac{1}{2} \( \ell_0(z) - \triv{M}\)$
which is even as $\triv{M} \equiv \ell_0(z) \modu 4)$.
\end{proof}

\section{Proofs of the main theorems in type D}\label{D-proof-sect}

This section contains the proofs of Theorems~\ref{main-thm} and \ref{main-thm2} for
the classical types with Weyl group $W=\WD_n$.
We start with some general results about even-signed permutations in Sections~\ref{esi-sect} and \ref{esg-sect}.
Then in Sections~\ref{tDI-sect}, \ref{tDI-sect2}, and \ref{tDIII-sect} we prove
a sequence of propositions that establish the different parts of our main theorems in types DI and DIII.

\subsection{Even-signed relations}\label{esi-sect}

Fix $ z \in \I(\WD_n)$ and write $\ell$ for the Coxeter length function of $ \WD_n$.
For the rest of this section
\be
\cED(z) = \left\{ \text{minimal-length elements }w \in \WD_n\text{ with }w \circ w^{-1}= z\right\}
\ee
where $\circ$ is the Demazure product for $\WD_n$.
If $w \in \WD_n$ and $i \in [n]$ then we write $w_i =w(i)$.

\begin{lemma}\label{D1-lem0}
Suppose $w \in \cED(z)$. Then
\ben
\item[(a)] neither $|w_1| > w_2>w_3$
nor $w_i>w_{i+1}>w_{i+2}$ holds for any $i \in [n-2]$, and 
\item[(b)] if $w_1w_2w_3 =CAB$ where $A<B<|C|$ then  $|B| < |C|$.
\een
\end{lemma}

\begin{proof}
Recall that  
$\ell(w t_i) < \ell(w)$ if and only if $w_i > w_{i+1}$ and $\ell(wt_{-1}) < \ell(w)$ 
if and only if $-w_1>w_2$.
Hence, if $w_i>w_{i+1}>w_{i+2}$ for some $i \in[n-2]$
then $w$ has a reduced expression ending in $t_i t_{i+1} t_i$
and if $-w_1>w_2>w_3$ then $w$ has a reduced expression ending in $t_{-1}t_2 t_{-1}$.
Since 
\[ (t_i \circ t_{i+1} \circ t_i) \circ (t_i \circ t_{i+1} \circ t_i) = t_{i+1}\circ t_i \circ t_{i+1} =  ( t_{i+1} \circ t_i) \circ (t_i \circ t_{i+1})\]
and
\[ (t_{-1} \circ t_{2} \circ t_{-1}) \circ (t_{-1} \circ t_{2} \circ t_{-1}) = t_{2}\circ t_{-1} \circ t_{2} =  ( t_{2} \circ t_{-1}) \circ (t_{-1} \circ t_{2})\]
no such $w$ having these patterns can be a minimal-length element satisfying $w\circ w^{-1} =z$.

For part (b), suppose  $w_1w_2w_3 =CAB$ where $A<B<|C|$ and $|B|>|C|$.
If $C>0$ then we must have $A<B<-C<0<C$ and if $C<0$ then $A<B<C<0<-C$.
Using these inequalities, one can check that
 $w$ has a reduced expression ending in $t_{-1}t_{1} t_2 t_1 t_{-1}$
which contradicts our assumption that $w$ is a minimal-length element satisfying $w\circ w^{-1} =z$
since the Demazure products
\[ 
\ba
(t_{-1}\circ t_{1}\circ t_2\circ t_1 \circ t_{-1})\circ (t_{-1}\circ t_{1}\circ t_2\circ t_1 \circ t_{-1})
&= 
%t_{-1}\circ t_{1}\circ t_2\circ t_1 \circ t_{-1}\circ t_2\circ t_1 \circ t_{-1}
%\\&= 
%t_{-1}\circ t_{2}\circ t_1\circ t_2 \circ t_{-1}\circ t_2\circ t_1 \circ t_{-1}
%\\&= 
%t_{-1}\circ t_{2}\circ t_1\circ t_{-1} \circ t_{2}\circ t_{-1}\circ t_1 \circ t_{-1}
%\\&= 
%t_{-1}\circ t_{2}\circ t_1\circ t_{-1} \circ t_{2}\circ t_{-1}\circ t_1
%\\&= 
%t_{-1}\circ t_{2}\circ t_1\circ t_{2} \circ t_{-1}\circ t_{2}\circ t_1
%\\&= 
%t_{-1}\circ t_{1}\circ t_2\circ t_{1} \circ t_{-1}\circ t_{2}\circ t_1
%\\&= 
%t_{-1}\circ t_{1}\circ t_2\circ t_{-1} \circ t_{1}\circ t_{2}\circ t_1
%\\&= 
%t_{-1}\circ t_{1}\circ t_2\circ t_{-1} \circ t_{2}\circ t_{1}\circ t_2
%\\&= 
%t_{-1}\circ t_{1}\circ t_{-1}\circ t_{2} \circ t_{-1}\circ t_{1}\circ t_2
%\\&= 
% t_{1}\circ t_{-1}\circ t_{2} \circ t_{-1}\circ t_{1}\circ t_2
% \\&= 
% t_{1}\circ t_{2}\circ t_{-1} \circ t_{2}\circ t_{1}\circ t_2
%  \\&= 
% t_{1}\circ t_{2}\circ t_{-1} \circ t_{1}\circ t_{2}\circ t_1
%  \\&= 
( t_{1}\circ t_2\circ t_1 \circ t_{-1})\circ (t_{-1}\circ t_{1}\circ t_2\circ t_1)
\ea
  \]
  both give the even-signed permutation $ t_{1} t_{2} t_{-1}  t_{1} t_{2} t_1=1\overline{2}\overline{3}45\cdots n=(-2,2)(-3,3)\in \I(\WD_n)$.
  \end{proof}

\begin{definition}\label{d-cover-def0}
Let $\precsimD $ be the transitive closure of the relation on words
that has
\be\label{d-cover1}
u B C A v  \precsimD   u C A B  v
\quand
\overline{B} C A  v  \precsimD   \overline{C}  A B  v
\ee
if $A,B,C \in \ZZ$ are integers with $A<B<C$ and $u$ and $v$ are any subwords.
\end{definition}

Notice that $\precsimD $ is the same as $\precsim_{\DI}^{(n,n)}$ from Definition~\ref{D12-gen-def}.

\begin{definition}\label{d-cover-def}
Let $\llD$ be the transitive closure of $\precsimD $ and the relation on words
that has
\be\label{d-cover2}
u  A\overline{B}  vC\overline{D}  w \llD   u A\overline{D} v  B \overline{C}  w
\ee
if $A,B,C,D \in \ZZ$ are integers with $0<|A|<B<C<D$ and $u$, $v$, and $w$ are subwords
such that all letters of $u$ and $v$ have absolute value strictly less than $B$.
\end{definition}

We evaluate $\precsimD $ and $\llD$ on elements of $\WD_n$ identified
with their one-line representations.
The following result
 connecting these relations to  $\cED(z)$
 is essentially a reformulation of \cite[Thm.~3.17]{HuZhang2}.

\begin{theorem}[See \cite{HuZhang2}]
\label{d-equiv-lem}
For each involution $z \in \I(\WD_n)$ the set
 $\cED(z)$ is a single equivalence class for the symmetric closure of $\llD$.
\end{theorem}

Figure~\ref{drel-fig} shows an example of $\cED(z)$ partially ordered by $\llD$.
 \begin{proof}
 Fix $z \in \I(\WD_n)$
and let $\cRD(z)$ be the union of the sets of reduced words for all $w \in \cED(z)$.
Hu and Zhang's result \cite[Thm.~3.17]{HuZhang2} (see also \cite[Thm.~4.1]{HH}) 
asserts that $\cRD(z)$ is a single equivalence class for the transitive closure of 
the usual Coxeter braid relations for $\WD_n$ and the extra relations 
\be\label{pcover-eq1}
\cdots t_i t_{i+1} \sim \cdots t_{i+1} t_i \quad\text{for }i \in [n-2]
\quand
\cdots t_{-1} t_{2} \sim \cdots t_{2} t_{-1}
\ee
along with
\be\label{pcover-eq2} 
\cdots t_2t_3t_{-1} t_1 t_2t_{-1} t_1 t_3 \sim \cdots t_3 t_2 t_{-1} t_1 t_2 t_{-1} t_1 t_3.
\ee

It is a straightforward exercise to check that two elements of $\WD_n$
have reduced words related as in \eqref{pcover-eq1} 
if and only if the one-line representations of these signed permutations are related as in \eqref{d-cover1}.
%Hence $\cED(z)$ is preserved by the symmetric closure of $\precsimD $.
Similarly,
two elements of $\WD_n$ have reduced words  related as in \eqref{pcover-eq2} 
if and only if their one-line representations have the form
$
 A\overline{B}  C\overline{D}  \cdots$ and $    A\overline{D}   B \overline{C}  \cdots
$
for some integers $A,B,C,D \in \ZZ$  with $0<|A|<B<C<D$.
We conclude that  $\cED(z)$ is a single equivalence class for  the symmetric closure of 
the sub-relation of
$\llD$ that only allows \eqref{d-cover2} when  $\ell(u)=\ell(v)=0$.

It  only remains to show that $\cED(z)$ is preserved by the 
unrestricted relation \eqref{d-cover2}. 
For this, suppose
$0<|A|<B<C<D$ and
$u$ and $v$ are words whose letters all have absolute value strictly less than $B$.
If  $
u  A\overline{B}  vC\overline{D}  w$
is the one-line representation of an element of $\cED(z)$,
then Lemma~\ref{D1-lem0}(a) and the preceding paragraph imply that  
$u  A\overline{B}  vC\overline{D}  w \precsimD     A\overline{B}  C\overline{D} u vw \in \cED(z)$.
Similarly, if $ u A\overline{D} v  B \overline{C}  w$
 is the one-line representation of an element of $\cED(z)$,
 then $u A\overline{D} v  B \overline{C}  w \precsimD  A\overline{D}   B \overline{C}  uvw \in \cED(z)$.
 As we have already seen that $A\overline{B}  C\overline{D} u vw \in \cED(z)$
 if and only if $A\overline{D}   B \overline{C}  uvw \in \cED(z)$,
 we conclude that 
 it likewise holds that 
  $u  A\overline{B}  vC\overline{D}  w \in \cED(z)$
 if and only if $u A\overline{D} v  B \overline{C}  w \in \cED(z)$, as needed.
  \end{proof}

 \begin{figure}
\[    \begin{tikzpicture}[xscale=1.5,yscale=1.5]%[xscale=2.3, yscale=2.3,>=latex,every node/.style={scale=0.8}]
\node at (0,0) (A1) {\color{blue}$12\bar34\bar5$};
\node at (-3,1) (B1) {\color{blue}$124\bar5\bar3$};
\node at (-1,1) (B2) {\color{blue}$2\bar314\bar5$};
\node at (1,1) (B3) {\color{red}$12\bar53\bar4$};
\node at (3,1) (B4) {\color{blue}$\bar2\bar3\bar14\bar5$};
\node at (-4,2) (C1) {\color{blue}$14\bar52\bar3$};
\node at (-2,2) (C2) {\color{blue}$2\bar34\bar51$};
\node at (0,2) (C3) {\color{red}$2\bar513\bar4$};
\node at (2,2) (C4) {\color{red}$\bar2\bar5\bar13\bar4$};
\node at (4,2) (C5) {\color{blue}$\bar2\bar34\bar5\bar1$};
\node at (-5,3) (D1) {\color{blue}$4\bar512\bar3$};
\node at (-3,3) (D2) {\color{blue}$24\bar5\bar31$};
\node at (-1,3) (D3) {\color{blue}$\bar4\bar5\bar12\bar3$};
\node at (1,3) (D4) {\color{red}$2\bar53\bar41$};
\node at (3,3) (D5) {\color{red}$\bar2\bar53\bar4\bar1$};
\node at (5,3) (D6) {\color{blue}$\bar24\bar5\bar3\bar1$};
\node at (-3,4) (E1) {\color{blue}$4\bar52\bar31$};
\node at (-1,4) (E2) {\color{blue}$\bar4\bar5\bar2\bar31$};
\node at (1,4) (E3) {\color{blue}$\bar4\bar52\bar3\bar1$};
\node at (3,4) (E4) {\color{blue}$4\bar5\bar2\bar3\bar1$};
\draw[->,thick]  (A1) -- (B1);
\draw[->,thick]  (A1) -- (B2);
\draw[->,dotted]  (A1) -- (B3);
\draw[->,thick]  (A1) -- (B4);
\draw[->,thick]  (B1) -- (C1);
\draw[->,thick]  (B2) -- (C2);
\draw[->,dotted]  (B2) -- (C3);
\draw[->,thick]  (B3) -- (C3);
\draw[->,thick]  (B3) -- (C4);
\draw[->,dotted]  (B4) -- (C4);
\draw[->,thick]  (B4) -- (C5);
\draw[->,thick]  (C1) -- (D1);
\draw[->,thick]  (C1) -- (D3);
\draw[->,thick]  (C2) -- (D2);
\draw[->,dotted]  (C2) -- (D4);
\draw[->,thick]  (C3) -- (D4);
\draw[->,thick]  (C4) -- (D5);
\draw[->,dotted]  (C5) -- (D5);
\draw[->,thick]  (C5) -- (D6);
\draw[->,thick]  (D1) -- (E1);
\draw[->,thick]  (D2) -- (E1);
\draw[->,thick]  (D2) -- (E2);
\draw[->,thick]  (D3) -- (E3);
\draw[->,thick]  (D6) -- (E3);
\draw[->,thick]  (D6) -- (E4);
     \end{tikzpicture}\]
     \caption{Hasse diagram of $\cED(z)$ partially ordered by $\llD$ for $z= 1\overline{2345}=w_0 \in \I(\WD_5)$.
     The solid arrows indicate covering relations of the form $\precsimD$.
     Here, there are two possible values of $\shD(w) \in \NCSP(\{2,3,4,5\}:0)$ for $w \in\cED(z)$.
     The blue (respectively, red) vertices are the elements $w \in \cED(z)$ with $\shD(w) $ equal to $ \{\pm \{2,3\}, \pm \{4,5\}\}$
     (respectively, $\{\pm \{2,5\}, \pm \{3,4\}\}$).
     }\label{drel-fig}
\end{figure}

Lemma~\ref{D1-lem0} and Theorem~\ref{d-equiv-lem} immediately imply the following:

  \begin{corollary}\label{well-cor}
The set $\cED(z)$ is a well-nested family for all $z \in \I(\WD_n)$.
 \end{corollary}

For a word $w=w_1w_2\cdots w_n$
write $\NRes(w) = \{ a_1<\dots<a_p < b_1<\dots<b_q\}$ where $p,q \in \NN$ and
$a_i<0$ for all $i \in [p]$ and $0<b_j$ for all $j\in [q]$. 
Then define $\NDes^\pm(w)$ to be the set of  pairs
\ben
\item[] $(c, d)$ for all $(c,d) \in \NDes(w)$,
\item[] $(a_{i}, a_{i+1})$ for each odd $i\in [p-1]$, and
\item[] $(a_p, b_1)$ if $p$ is odd and $q>0$ and $|a_p|>b_1$.
\een
Also define 
\be
\NRes^\pm(w) = \begin{cases}
\{ b_2,\dots,b_q\} &\text{if $p$ is odd and $q>0$ and $|a_p|>b_1$}, \\
\{|a_p|, b_1,b_2,\dots,b_q\}&\text{if $p$ is odd and $q>0$ and $|a_p|<b_1$}, \\
\{b_1,b_2,\dots,b_q\} &\text{otherwise}.
\end{cases}
\ee
We evaluate this construction on even-signed permutations $w \in \WD_n$ in the usual way
by identifying $w$ with its one-line representation $w=w_1w_2\cdots w_n$ where $w_i = w(i)$.
Next, let
\be\label{k0-shape-eq}
\ba
\shD(w) =\sh_{\DI}^{(n,n)}(w) &= \Bigl\{ \pm\{|a|,-b\} : (a,b) \in \NDes(w)\text{ with }|a|<-b\Bigr\}
\\
&= \Bigl\{ \pm\{|a|,-b\} : (a,b) \in \NDes^\pm(w)\text{ with }|a|<-b\Bigr\}.
\ea
\ee
 \begin{lemma}\label{CC-lem}
Suppose $w \in \cED(z)$ has $w_1w_2w_3 = \overline{C}AB$ where $A<B<C$.
Then 
\ben
\item[(a)] it holds that $(-C,A) \in \NDes^\pm(w)$, and 
\item[(b)] if $-C<A$ then 
$\{-C,A\}\subset \NRes(w)$
and
$ \NDes(w) = \NDes(w_3w_4\cdots w_n)$.
\een
 \end{lemma}
 
 \begin{proof}
 If $-C>A$ then $(-C,A) \in \NDes(w)\subseteq \NDes^\pm(w)$. Assume $-C<A$.
 By Theorem~\ref{d-equiv-lem} the signed permutation $v = \overline{B}CA w_4w_5\cdots w_n$ also belongs to $\cED(z)$.
 It suffices to show that $A \in \NRes(w)$ as then $\NRes(w) = \{-C<A <\dots\}$.
If we assume  $A \notin \NRes(w)$, then
it follows from Proposition~\ref{ndes-prop}
that we can successively remove a sequence of descents from $w_4w_5\cdots w_n$ to transform $ w$
to a word of the form
$\overline{C} A B XY\cdots $ where $B>X$ and $A>Y$. 
But then removing the same descents would 
transform $v$ to a word of the form
$
\overline{B} C A  XY\cdots 
$.
One of these words has a consecutive $321$-pattern, contradicting
Lemma~\ref{nest-lem} since $\cED(z)$ is a well-nested family.
 \end{proof}

Let $w=w_1w_2\cdots w_n$ be a
sequence of nonzero integers with distinct absolute values.
Given $a \in \{ \pm w_1, \pm w_2,\dots,\pm w_n\}$
let $w^{-1}(a) = i$ if $a =w_i$ and let $w^{-1}(a) = -i$ if $-a=w_i$. 
Then define   \[
I_R(w) = \left\{ w^{-1}(a) : (b,a) \in \NDes^\pm(w)\right\} \quand I_L(w) =[n]\setminus I_R(w).\] Write these sets as 
$
I_R(w) = \{ r_1<r_2<\dots < r_l\}$ and $ I_L(w) = \{ \ell_1<\ell_2<\dots < \ell_m\}.
$
Then let %$w_R$ and $w^\pm_L$ denote the words
\[\ba w_R &= w(r_1)w(r_2)\cdots w(r_l)
\\
w_L &=w(\ell_1)w(\ell_2)\cdots w(\ell_m)
\ea
 \quand w^\pm_L = w(-\ell_m)\cdots w(-\ell_2)w(-\ell_1)w(\ell_1)w(\ell_2)\cdots w(\ell_m)\]
where if  $i \in [n]$ then $w(i) = w_i$ and $w(-i)=-w_i$.
Finally define
\be\label{rankD-def}
\rankD(w) = \tfrac{1}{2}\Bigl(\inv( w^\pm_L)- \ell_0(w_L)\Bigr) - \inv(w_R) \in\ZZ
\ee
where for any word $a=a_1a_2\cdots a_m$ we set
$\ell_0(a) =|\{i \in [m] : a_i<0\}|$ and 
\be 
\ba\inv(a) &= |\{ (i,j) \in [m]\times [m] : i<j\text{ and }a_i>a_j\}|.
\ea
\ee
Notice that if $(i,j) $ is an inversion for $w^\pm_L$ then so is $(2m+1-j,2m+1-i)$.
Thus $\inv( w^\pm_L)- \ell_0(w_L)$ is even
as $\ell_0(w_L)$ counts the inversions of  
$w^\pm_L$ that are fixed by $(i,j)\mapsto (2m+1-j,2m+1-i)$.
   
\begin{example}

The elements $v=16\overline{52}3487\precsimD  w=\overline{6512}3487 $ are both in $ \cED(z)$ for 
\[ z=(-1,1)(-2,2)(5,-6)(-5,6)(7,8)(-7,-8) \in \I(\WD_8).
\]
We have $\CycPM(z) = \{ (-5,6),(7,8), (3,3),(4,4)\}$ and $\Neg(z) =\{1,2\}$, and it holds that
\[
\ba 
\NDes(v) &= \{(6,-5),(1,-2),(8,7) \},\\
\NRes(v) &= \{3,4\},\\
\NDes^\pm(v)&= \{(6,-5),(1,-2),(8,7) \}, \\
\NRes^\pm(v) &= \{3,4\},
\ea
\qquad
\ba 
\NDes(w) &= \{(-1,-2), (8,7)\},\\
\NRes(w) &= \{-6,-5,3,4\},\\
\NDes^\pm(w) &= \{(-1,-2),(8,7),(-6,-5)\}, \\
\NRes^\pm(w) &= \{ 3,4\}.
\ea
\]
Thus
$ v_R = w_R = \overline{52}7$
and
$v^\pm_L = \overline{84361}16348$
and
$w^\pm_L = \overline{843}16\overline{61}348 $
so we compute 
\[\rankD(v) = \tfrac{1}{2}\(4 -0\) - 0 = 2\quand \rankD(w) = \tfrac{1}{2}\(8 -2\) - 0 = 3.\]
\end{example}

  \begin{theorem}\label{rankD-prop1}
  The relation $\precsimD $ is a graded partial order on $\cED(z)$
with rank function $\rankD$.
  \end{theorem}
  
  \begin{proof}
  Let $w,\hat w \in \cED(z)$.
It suffices to show that if $w \precsimD  \hat w$ is a relation of the form \eqref{d-cover1}
  then $\rankD(\hat w) = \rankD(w)+1$. 
  %This shows that \eqref{d-cover1} is a covering relation.
 %
  First assume that in one-line notation we have
\be\label{dr-eq1}
w = u B C A v \quand  \hat w = u C A B  v
\ee
for some integers  $A<B<C$ and some subwords $u$ and $v$.
Then $\NDes(w) =\NDes(\hat w)$ since $\cED(z)$ is well-nested. 

As 
$(C,A) \in \NDes(w) =\NDes(\hat w)$,
one has $w^{-1}(B) \in I_R(w)$ if and only if $\hat w^{-1}(B)\in I_R(\hat w)$.
 If $w^{-1}(B) \in I_R(w)$ then the subword $BA$ contributes an inversion to $w_R$ but not to $\hat w_R$ so
 we have
$
  \inv(w_R) - 1 = \inv(\hat w_R)
$ while
$  w_L^\pm = \hat w_L^\pm
$.
Alternatively if $w^{-1}(B) \in  I_L(w)$ then the subwords $\overline{BC}$ and $CB$ contribute two inversions to $\hat w^\pm_L$ that are not present in $w_L^\pm$ so
$
\inv(w_L^\pm) +2 = \inv(\hat w_L^\pm)
$ while $ w_R = \hat w_R .
$
In either case $\ell_0(w_L)=\ell_0(\hat w_L)$ so  $\rankD(\hat w) = \rankD(w)+1$. 

Now suppose instead that in one-line notation we have
\be\label{dr-eq2}
w = \overline{B} C A  v\quand \hat w= \overline{C}  A B  v\ee
for some integers $A<B<C$ where $v$ is an arbitrary subword.
Then $|B|<|C|$ by Lemma~\ref{D1-lem0} so we must have $C>0$.

By Lemma~\ref{CC-lem} we have $(C,A) \in \NDes^\pm(w)$ and $(-C,A) \in \NDes^\pm(\hat w)$, 
so $\{1,2\}\subset I_L(w)$ and  $3 \in I_R(w)$ 
while $\{1,3\}\subset I_L(\hat w)$
 and $2 \in I_R(\hat w)$.
This means that $w_R=\hat w_R$ 
while $w^\pm_L$ and $\hat w^\pm_L$ only differ in their middle four letters, which are of the form
\[ 
w^\pm_L = \cdots \overline{C} B \overline{B}C \cdots \quand \hat w^\pm_L = \cdots \overline{B} C \overline{C}B \cdots.
\]
It follows that
\bei 
\item if $0<B<C$ then we have 
$\ell_0(w_L)=\ell_0(\hat w_L)$ and $ \inv(w_L^\pm) +2= \inv(\hat w_L^\pm)$; 
\item if $-C<B<0$ then 
$\ell_0(w_L)+2=\ell_0(\hat w_L)$ and $ \inv(w_L^\pm) +4= \inv(\hat w_L^\pm)$.
\eei
Either way, we have $\rankD(\hat w) = \rankD(w)+1$ as needed.
  \end{proof}
  
  We end this section with one other fundamental property of the partial order $\precsimD$.

 \begin{lemma}\label{same-shape-lem}
 Suppose $v,w \in \cED(z)$ have $v\precsimD  w$. Then $\shD(v) = \shD(w)$.
\end{lemma}

\begin{proof}
We must show that if $w \precsimD  \hat w$ is a covering relation of the form \eqref{d-cover1}
  then $\shD(\hat w) = \shD(w)$.
  This is clear if $w$ and $\hat w$ are as in \eqref{dr-eq1} since then $\NDes(w) = \NDes(\hat w)$.
 Instead suppose $w$ and $\hat w$ are as in \eqref{dr-eq2}.
  Define $\cN(w) = \{ (|a|, b) : (a,b) \in \NDes(w)\}$.
Then \[\shD(w) = \{ \pm \{ a, -b\} : (a,b) \in \cN(w) \text{ with }a < -b\}.\]
It follows from Lemma~\ref{CC-lem} that  $\cN(\hat w)$ is formed from $\cN(w)$ by the following operations:
\bei
\item[(1)] if $C>-A$ then remove $(C,A)$;
\item[(2)] if there is some $X<B$ with $(|B|,X) \in \cN(w)$ and $-B<X$ then
remove $(|B|,X)$;
\item[(3)] if $\NRes(w) = \{ B<X<\dots\}$ for some $X$ with $-B>X$ then add the pair $(|B|,X)$.
\eei
The pair $(C,A) \in \cN(w)$ in step (1) does not contribute any blocks to $\shD(w)$ since $C>-A$.
Neither does the pair $(|B|,X) \in \cN(w)$ in step (2)
since if $-B<X<B$ then $|B|>-X$.
Likewise, the pair $(|B|,X) \in \cN(\hat w)$ in step (3) does not contribute any blocks to $\shD(\hat w)$
since if $B<X<-B$ then $|B|>-X$.
We conclude that $\shD(\hat w) = \shD(w)$.
\end{proof}
 
 \subsection{Even-signed generators}\label{esg-sect}
 
  As in the previous section 
 we fix an even-signed involution $ z \in \I(\WD_n)$. 
 Recall that if we are
  given a word $a=a_1a_2\cdots a_n$ and $a_1\neq 0$  then we let 
 $\tilde a = \overline{a_1}a_2\cdots a_n$
 and define
 $\langle a\rangle_{\es}$ to
be whichever of $a$ or $\tilde a$ 
has an even number of negative letters. 
If $a$ is the one-line representation of an element of $\W_n$ 
then we interpret  $\langle a\rangle_{\es} \in \WD_n$ in one-line notation.

Let $X=\Neg(z)$.
This set has an even number of elements, so we can define
\[
  M_{\min}  =\{ \pm\{c_1,c_2\},\pm\{c_3,c_4\},  \dots, \pm\{c_{k-1},c_k\}\}
 \quad\text{where}\quad X= \{c_1<c_2<\dots<c_k\}. \]
Write $\cMM(z)  = \cM_{\DI}^{(n,n)}(z) =\NCSP(X:0)$.
Note that ignoring all negative blocks gives a bijection to the noncrossing perfect matchings of $X$. 
Let 
 \be \ddelta(z,M)   = \bot_{\DI}^{(n,n)}(z,M)
 \quand \ddelta(z) = \ddelta(z,M_{\min})
 \ee
 for each $M \in \cMM(z)$. Concretely, if $P$ is the set consisting of the pairs 
\ben
\item[] $(a,z(a))$ for $a\in[\pm n]$ with $|a|< z(a)$ or $0<a=z(a)$, and
\item[] $(-b,a)$ for $\{a,b\}\in M$  with $0<a<b$,
\een
then  $\ddelta(z,M) =\langle[[P]]_{\des}\rangle_\es$
in one-line notation.
Notice that 
\be [[P]]_{\des} \in \WD_n 
\quad\text{if and only if}\quad
\ell_0(z) \equiv 0 \modu 4)\ee
since the 
 number of negative letters in $[[P]]_{\des}$ is 
the number of $a \in [n]$ with $z(a) < -a$ plus the number of positive blocks in $M$, and the first number 
is
 $\tfrac{1}{2}(\ell_0(z) - |X|)$
 while the second is $\tfrac{1}{2}|X|$.

   \begin{lemma}\label{111-lem}
 If $i \in \{-1,1,2,\dots,n-1\}$ has $t_i \circ z \circ t_i \neq z$ then 
  some $M \in \cMM(t_i \circ z \circ t_i)$ has
  \[ \ddelta(t_i \circ z \circ t_i, M) \precsimD  t_i \cdot \ddelta(z) .\]
  \end{lemma}

 \begin{proof}
Recall that $\diamond$ is the length-preserving involution of $\WD_n$
with $w^\diamond = t_0  w  t_0$ so $t_{\pm1}^\diamond = t_{\mp 1}$.
 Notice that $z^\diamond \in \I(\WD_n)$ has $\Neg(z^\diamond) =\Neg(z)$,
 $\cMM(z^\diamond) = \cMM(z)$, and $\ddelta(z^\diamond,M) = \ddelta(z,M)^\diamond$ for all  $M$.
 Also if $v ,w \in \WD_n$ then $v \precsimD  w$ if and only if $v^\diamond \precsimD  w^\diamond$.
Given these observations, it suffices to prove the lemma when $i \in [n-1]$
 since if $t_{-1} \circ z \circ t_{-1} \neq z$ then we have 
 \[ t_{1} \circ z^\diamond \circ t_{1} = (t_{-1} \circ z \circ t_{-1} )^\diamond   \neq z^\diamond\]
in which case if some $M \in \cMM(t_{1} \circ z^\diamond \circ t_{1}) = \cMM(t_{-1} \circ z \circ t_{-1} )$
exists with 
\[ \ddelta(t_1 \circ z^\diamond \circ t_1, M) \precsimD  t_1 \cdot \ddelta(z^\diamond) 
\quad\text{then also}\quad \ddelta(t_{-1} \circ z \circ t_{-1}, M) \precsimD  t_{-1} \cdot \ddelta(z) .\]
Thus we may assume $i \in [n-1]$. If $z(i) = i$ and $z(i+1) = i+1$ then it follows from \eqref{BC-dem-eq} that
\[ \ba
\Cyc^{\pm}(t_i \circ z \circ t_i) &=  \(\Cyc^{\pm}(z) \setminus\{(i,i), (i+1,i+1)\} \) \sqcup \{(i,i+1)\}
\quand\\
\Neg(t_i\circ z \circ t_i) &= \Neg(z).
\ea
\]
In this case $\ddelta(t_i \circ z \circ t_i) =t_i \cdot \ddelta(z)$.
If $z(i) =- i-1$ and $z(i+1) = -i$ then we likewise have
\[
\ba
 \Cyc^{\pm}(t_i \circ z \circ t_i) &=  \Cyc^{\pm}(z) \setminus\{(-i,i+1)\}
 \quand\\
\Neg(t_i\circ z \circ t_i) &= \Neg(z)\sqcup\{i,i+1\}.
\ea
\]
In this case   $t_i \cdot \ddelta(z) =\ddelta(t_i\circ z\circ t_i,M)$
for $M = M_{\min}(X:0) \sqcup \{\pm \{i,i+1\}\} \in \cMM(t_i\circ z \circ t_i)$.

Assume we are not in the preceding cases so that $t_i \circ z \circ t_i = t_i \cdot z \cdot t_i$.
Then we have
\[
\ba
 \Cyc^{\pm}(t_i \circ z \circ t_i) &=\{ (t_i(a), t_i(b)) : (a,b)\in  \Cyc^{\pm}(z) \}
 \quand\\
\Neg(t_i\circ z \circ t_i) &= \{ t_i(a) : a \in \Neg(z)\}.
\ea
\]
If there  are   $A,B \in [\pm n]$ with $(A,i),(B,i+1) \in \Cyc^{\pm}(z)$ and $A\neq i$ and $B\neq i+1$,
then $z(i)=A<z(i+1)=B <i$ so we can express
$\ddelta(t_i\circ z \circ t_i) $ and $t_i \cdot \ddelta(z)$ in one-line notation as
\[\ba
\ddelta(t_i\circ z \circ t_i) = \langle \cdots  i, B,(i+1),A \cdots \rangle_\es 
&\precsimD 
  \langle \cdots  i, (i+1),A,B \cdots \rangle_\es
  \\&\precsimD 
  \langle \cdots  (i+1),A,i,B \cdots \rangle_\es = t_i \cdot \ddelta(z).
  \ea
\]
Similarly,
if we have 
$(A,i),(i+1,i+1) \in \Cyc^{\pm}(z)$ for some $i\neq A \in [\pm n]$ then $z(i)=A<i$ so
\[\ba
\ddelta(t_i\circ z \circ t_i) = \langle \cdots  i,(i+1),A \cdots \rangle_\es 
&\precsimD 
  \langle \cdots  (i+1),A,i \cdots \rangle_\es  = t_i \cdot \ddelta(z).
  \ea
\]

Now suppose $ \Neg(z)=\{ c_1<c_2<\dots<c_k\}$ and $i=c_j$ for some $j \in [k]$ (note that $k$ must be even)
while $(B,i+1) \in \CycPM(z)$ for some $B \in [\pm n]$.
If $j$ is even then $\ddelta(t_i\circ z \circ t_i)  = t_i \cdot \ddelta(z)$
while if $j$ is odd then
$\{c_j < c_{j+1}\} \in M_{\min}(X:0)$
and
\[
-c_{j+1} < z(i) = -i < z(i+1) = B < i\] so we can express
$\ddelta(t_i\circ z \circ t_i) $ and $t_i \cdot \ddelta(z)$ in one-line notation as
\[\ba
\ddelta(t_i\circ z \circ t_i) = \langle \cdots  i, B,(i+1),\overline{c_{j+1}} \cdots \rangle_\es 
&\precsimD 
  \langle \cdots  i, (i+1),\overline{c_{j+1}},B \cdots \rangle_\es
  \\&\precsimD 
  \langle \cdots  (i+1),\overline{c_{j+1}},i,B \cdots \rangle_\es = t_i \cdot \ddelta(z).
  \ea
\]
We cannot have 
$(i,i),(B,i+1) \in \Cyc^{\pm}(z)$ for some $i+1\neq B\in [\pm n]$ as this leads to the contradiction
 $B<i=z(i) < z(i+1)=B$. Likewise, we cannot have
$(A,i) \in \Cyc^{\pm}(z)$ for some $A \in [\pm n]$ when $i+1 \in \Neg(z)$ as
this leads to the contradiction $-i <A \leq z(i)<z(i+1)=-i-1$.

In all remaining cases
at most one of $i$ or $i+1$ is in $\{B: (A,B) \in \CycPM(z)\}\sqcup \Neg(z)$ 
and then the definition of $\ddelta$ implies that $\ddelta(t_i\circ z \circ t_i) =t_i \cdot \ddelta(z)$.
\end{proof}
 
 Recall the definition of the relation $\lessdot $ 
 %on $\cMM(z)$ 
 from Proposition~\ref{lessdot-prop}.
 \begin{lemma}\label{lessdot-lem2}
 Suppose $M,N \in \cMM(z)$ have $M\lessdot N$. Then $ \ddelta(z,M)\llD \ddelta(z,N)$.
   \end{lemma}

 \begin{proof}
In this case 
%the definition of $\ddelta(z,\cdot)$ implies that 
we can write 
$ \ddelta(z,M)\overset{0}\precsim  uA\overline{B}  vC\overline{D}  w
\llD
 u A\overline{D} v  B \overline{C}  w
 =
\ddelta(z,N)
$
for some $A,B,C,D \in \ZZ$ with $0<|A|<B<C<D$ where $u$, $v$, and $w$ are subwords
such that all letters of $u$ and $v$ have absolute value strictly less than $B$. %We will have   $A=|A|$ if $\ell(u)>0$.
%Thus $ \ddelta(z,M)\llD \ddelta(z,N)$.
 \end{proof}

 The next theorem is the main result of this section.

   \begin{theorem}\label{d1-prop1}
Suppose $z \in \I(\WD_n)$, $M\in\cMM(z)$,  and  $w \in \cED(z)$. Then:
\ben
\item[(a)] $\ddelta(z,M) \in \cED(z)$ and $ \shD(\ddelta(z,M)) = M$.
\item[(b)] $\shD(w)  \in \cMM(z)$.
\item[(c)]  $\CycPM(z)= \left\{ (b,|a|) : (a,b) \in \NDes^\pm(w)\text{ with }|a|>-b\right\} \sqcup \{ (c,c) : c \in \NRes^\pm(w)\}$.
\item[(d)]
 $ \Neg(z) $ consists of all numbers $|a|$ and $|b|$ for $ (a,b) \in \NDes^\pm(w) $ with $|a|<-b$.
\een
\end{theorem}

  \begin{proof}
We first prove part (a). If $z = 1$ then $\cMM(z) = \{\varnothing\}$ and $\ddelta(z,M)= \ddelta(1,\varnothing) = 1 \in \cED(1)$.
Assume $1\neq z \in \I(\WD_n)$.
Then some $i \in \{-1,1,2,\dots,n-1\}$ and   $y \in \I(\WD_n)$ have $y \neq t_i \circ y \circ t_i = z$.
For this choice of $i$ and $y$ we must also 
have $t_i  w \in \cED(z)$ for all $w \in  \cED(y)$.
We may assume by induction that $\ddelta(y,M) \in \cED(y)$ for all $M \in \cMM(y)$, and in particular that $\ddelta(y) \in \cED(y)$.
Theorem~\ref{d-equiv-lem} and Lemma~\ref{111-lem} imply that $\ddelta(z,M) \in \cED(z)$ for some $M \in \cMM(z)$.
We therefore have $\ddelta(z,M) \in \cED(z)$ for all $M \in \cMM(z)$
by Proposition~\ref{lessdot-prop}, Theorem~\ref{d-equiv-lem}, and Lemma~\ref{lessdot-lem2}.

Now let $ w=[[\CycPM(z,M)]]_{\des}$
and
recall that $\tilde w = \overline{w_1}w_2\cdots w_n$.
Then $\ddelta(z,M) \in \{w,\tilde w\}$ and 
it holds by definition that $\shD(w)=M$.
This is also the shape of $\tilde w$
since 
if we define $\cN(w)$ as in the proof of Lemma~\ref{same-shape-lem},
then
$\cN(\tilde w)$  is obtained from $\cN(w)$ by removing 
$(w_1,w_2)$ when $-w_1 < w_2$, but in this case   
 $w_1>-w_2$ so the removed pair $(w_1,w_2)$ does not contribute to $\shD(w)$.
  
For part (b), suppose $v \in \cED(z)$ has $v \llD w$ or $w \llD v$.
If $v$ and $w$ are related by a covering relation of the form \eqref{d-cover1}
then $\shD(v)=\shD(w)$ by Lemma~\ref{same-shape-lem}.
If $v$ and $w$ are related by a covering relation of the form \eqref{d-cover2} then 
Lemma~\ref{ijkl-lem} implies that $\shD(v)\lessdot \shD(w)$ or $\shD(w)\lessdot \shD(v)$.
We conclude from Proposition~\ref{lessdot-prop} that 
$\shD(v) \in \cMM(z)$ if and only if $\shD(w)\in \cMM(z)$.
Part (b) follows from this property given  Theorem~\ref{d-equiv-lem} and  part (a).

Similarly, to prove (c) and (d),
one may observe that both claims hold when $w = \ddelta(z)$,
and then use Lemma~\ref{CC-lem} to check that the sets 
described 
in each part do not change when $w$ is replaced by a signed-permutation
related to $w$ by a covering relation of the form \eqref{d-cover1} or \eqref{d-cover2}.
 \end{proof}
 
  \begin{corollary}\label{k0-precsim-cor}
 Let $z \in \I(\WD_n)$, $w \in \cED(z)$, and  $M=\shD(w)$. Then $\ddelta(z,M) \precsimD  w$.
 \end{corollary}
 
    \begin{proof}
   Let $v \in \WD_n$ be a minimal element with $v \precsimD w$.
   Then $v \in \cED(z)$ and $\shD(v) =M$ by Theorem~\ref{d-equiv-lem} and Lemma~\ref{same-shape-lem}. 
Recall that we define $\tilde v =\overline{v_1}v_2\cdots v_n \in \W_n$. Let
\[
\ba
 P &= \{ (a,b) : (b,a) \in \NDes( v)\}\sqcup \{(c,c) : c \in \NRes( v)\}\quand \\
\tilde P &= \{ (a,b) : (b,a) \in \NDes(\tilde v)\}\sqcup \{(c,c) : c \in \NRes(\tilde v)\}.
\ea
\] 
Then it follows from Proposition~\ref{Pdes-prop} and Lemma~\ref{D1-lem0} that either $v=[[P]]_{\des}$ and every $(a,b) \in P$ has $b>0$,
or   $\tilde v = [[\tilde P]]_{\des}$ and every $(a,b) \in \tilde P$ has $b>0$.
In either case we can deduce that $v = \ddelta(z,M)$ from  Theorem~\ref{d1-prop1}.
 \end{proof}

As another corollary, we can slightly strengthen Theorem~\ref{d-equiv-lem}.
    \begin{corollary}
If $z \in \I(\WD_n)$ then $\llD$ is a partial order on $\cED(z)$ with unique minimum $\ddelta(z)$.
  \end{corollary}
  
  \begin{proof}
  As noted in the proof of Theorem~\ref{d1-prop1}, if 
 $v\llD w $ is a covering relation for $v,w\in \cED(z)$
 then either $v\precsimD  w$ or $\shD(v) \lessdot \shD(w)$.
 Thus $\llD $ is a partial order
  by Proposition~\ref{lessdot-prop} and Theorem~\ref{rankD-prop1}.
  One can check directly that $\ddelta(z)$ is minimal under this order.
  This is the only minimal element in $\cED(z)$ in view of 
  Proposition~\ref{lessdot-prop},
  Lemma~\ref{lessdot-lem2}, and Corollary~\ref{k0-precsim-cor}.
    \end{proof}
  
  We close this section with two technical lemmas that follow from the results above.

   \begin{lemma}\label{cabd-lem}
 Suppose $z \in \I(\WD_n)$ and $w \in \cED(z)$. If there exists a pair $(a,b) \in \NDes(w)$ with $0<-a<-b$
 then no pairs $(c,d) \in \NDes(w)$ exist with  $|c| < |a| < -b < -d$.
 \end{lemma}
 
 \begin{proof}
 Let $M\in \cMM(z)$ and suppose $\{a,b\},\{c,d\} \in M$ are blocks with $0<c<a<b<d$.
Consulting the definition of $\shD(w)$ in \eqref{k0-shape-eq}, we see that it is equivalent to prove that 
if $\shD(w) = M$ then $a \in \{w_1,w_2,\dots,w_n\}$
rather than $-a \in \{w_1,w_2,\dots,w_n\}$.
This is true when $w = \ddelta(z,M)$
since 
 either $c\overline{d} a\overline{b}$ or  $\overline{c}\overline{d} a\overline{b}$ 
is a %(not necessarily consecutive) 
subword of the one-line representation of $\ddelta(z,M)$.

Now assume that $\shD(w) = M$  and $a \in \{w_1,w_2,\dots,w_n\}$. 
Suppose $\hat w \in \WD_n$ and $w \precsimD  \hat w$ is either covering relation in \eqref{d-cover1}.
Given Corollary~\ref{k0-precsim-cor}, it suffices to show that 
we also have $a \in \{\hat w_1,\hat w_2,\dots,\hat w_n\}$.

To prove this, notice that our hypothesis implies that
 $\NDes(w)$ contains $(a,-b)$ and either $(c,-d)$ or $(-c,-d)$.
Since $-c<c<a$ and $-d < -b$, it follows from Lemma~\ref{ijkl-lem} that 
$a\in \{w_3,w_4,\dots,w_n\}$.
Finally, it is clear from \eqref{d-cover1} that $\{w_3,w_4,\dots,w_n\} \subset  \{\hat w_1,\hat w_2, \dots,\hat w_n\}$.
 \end{proof}

Let $y,z \in \I(\WD_n)$ and suppose $v \in \cED(y)$ and $w \in \cED(y,z)$ so that $wv \in \cED(z)$.
Choose a reduced expression $w=t_{i_1}t_{i_2}t_{i_3}\cdots t_{i_m}$
so that we may write 
$z =z^0 \xrightarrow{t_{i_1}} z^1 \xrightarrow{t_{i_2}} z^2 \xrightarrow{t_{i_3}}  \cdots \xrightarrow{t_{i_m}} z^m = y$
for some $z^i \in \I(\WD_n)$,
where each arrow has the same meaning as in \eqref{wsv-eq}. 
Recall that 
\[C(t_i) =  \{ \{i,|i|+1\},\{-|i|-1,-i\}\} 
\quand \ccMD(z) = C(\overline{z})
\]
where $C(u)$ is the set of nontrivial cycles of a signed permutation $u$. 
Define 
\be
\Upsilon_{m+1} = \left\{ \pm \{ |w(a)|, |w(b)|\} : \{a,b\} \in \shD(v) \right\}.
\ee
Then, for each $j \in [m]$  let  
$u=t_{i_1}t_{i_2}\cdots t_{i_{j-1}}$ and define
\be\label{2Lam-eq}
\Upsilon_j
=\begin{cases}
\left\{ \pm\left\{|u(a)|, |u(b)|\right\} : \{a,b\} \in C (t_{i_j})\right\} &\text{if $C(t_{i_j}) \subseteq \ccMD(z^{j})$}\\
\varnothing &\text{if $C(t_{i_j}) \not\subseteq \ccMD(z^{j})$}.
\end{cases}
\ee 
In terms of this notation, the following holds:

\begin{lemma}\label{Lambda-lem}
We have $\shD(wv) =   \Upsilon_1 \sqcup  \Upsilon_2\sqcup \cdots \sqcup  \Upsilon_{m+1}$.
\end{lemma}

\begin{proof}
Since $ |u(|a|)|=|u(a)|$ when $u \in \WD_n$ and $a \in [\pm n]$,
it suffices to prove the lemma when $m=1$.
In this case, what we need to show is that if $z=t_i \circ y \circ t_i\neq y$ and $v \in \cED(y)$ then 
$\shD(t_iv)$  is the disjoint union of
$
 \left\{ \pm \{ |t_i(a)|, |t_i(b)|\} : \{a,b\} \in \shD(v) \right\}
$
and
\[
\begin{cases}
\left\{ \left\{i,i+1\right\}, \left\{-i,-i-1\right\}  \right\} &\text{if $i>0$ and $y(i) = -i-1$}\\
\left\{ \left\{1,2\right\}, \left\{-1,-2\right\}  \right\} &\text{if $i=-1$ and $y(1) = 2$}\\
\varnothing &\text{otherwise}.
\end{cases}
\]
This can be checked using Theorem~\ref{d1-prop1}.
\end{proof}

 \subsection{Proofs in type DI for even parity cases}\label{tDI-sect}

The propositions in this section will establish Theorems~\ref{main-thm} and \ref{main-thm2} for half of type DI. 
To begin, we prove a general lemma
before specializing to the case when $\frac{|p-q|}{2}$ is even.
Continue to fix an involution $z \in \I(\WD_n)$ and let $0\leq k \leq n$ be any integer.
Recall from Definition~\ref{AA-def} that 
\be\label{AA-def2}
  \hat \sigma_k^n =\overline{1}\hs \overline{2}\cdots \overline{k}(k+1)\cdots n
 \text{ if $k$ is even}
 \quand
 \hat \sigma_k^n =1\hs \overline{2}\cdots \overline{k}(k+1)\cdots n
 \text{ if $k$ is odd},
 \ee

 \begin{lemma}\label{inc-lem1}
The injective map  $\inc_{\DI}^{n,k} :\WD_n \to \WD_n$ given by
\[\inc_{\DI}^{n,k} : w \mapsto 
\begin{cases}
w_1\overline{w_2} w_3\overline{w_4}w_5\overline{w_6}\cdots w_{k-1}\overline{w_k}w_{k+1}w_{k+2}w_{k+3}\cdots w_n &\text{if $k\equiv 0\modu 4)$} \\
w_1w_2\overline{w_3}w_4\overline{w_5}w_6\cdots w_{k-1}\overline{w_k}w_{k+1}w_{k+2}w_{k+3}\cdots w_n &\text{if $k\equiv 1\modu 4)$} \\
\overline{w_1}\hs\overline{w_2} w_3\overline{w_4}w_5\overline{w_6}\cdots w_{k-1}\overline{w_k}w_{k+1}w_{k+2}w_{k+3}\cdots w_n &\text{if $k\equiv 2\modu 4)$} \\
\overline{w_1}w_2\overline{w_3}w_4\overline{w_5}w_6\cdots w_{k-1}\overline{w_k}w_{k+1}w_{k+2}w_{k+3}\cdots w_n &\text{if $k\equiv 3\modu 4)$}
\end{cases} 
\]
is a bijection 
$\cA(\hat\sigma_k^n, z) \to  \left\{ w \in \cED(z) :  |w_1| < (-1)^{k+1} w_2<  (-1)^{k+2}w_3<  \dots  <  w_{k-1} < -{w_k}\right\}.$
\end{lemma}

The inequalities defining the preceding set 
are considered vacuous when $k\in \{0,1\}$.
\begin{proof}
Let $v^{(k)} \in \WD_n$ be the element with $v^{(0)} = v^{(1)} = 1$, $v^{(2)} = t_1\cdot t_{-1}$, and  
\[v^{(k)} =  t_{k-1}\cdots t_3\cdot t_2\cdot  t_1 \cdot t_{-1}\cdot t_2\cdot t_3 \cdots t_{k-1} \cdot v^{(k-2)}\quad\text{ for }3\leq k \leq n.\]
It is easy to check by induction on $k$ that $\inc_{\DI}^{n,k}(w) = wv^{(k)}$ for all $w \in \WD_n$.
Since $\inc_{\DI}^{n,k}$ is an involution, the same is true of the element $v^{(k)}$.
Using this fact, 
and
recalling the descent conditions in \eqref{d-descent-eq},
  we can likewise deduce by induction that 
$
\ell(v^{(k)}) = \ell(v^{(k-2)}) + 2k-2 $ when $k\geq 2$.
Hence, our inductive formula for $v^{(k)}$ provides a reduced expression for this element.

Next, using \eqref{szs-eq} and \eqref{szs-length-eq}, one can show by induction in a similar way that
 $v^{(k)} \in \cED( \hat\sigma_k^n)$ for all $k$.
Lemma~\ref{dem-lem} therefore implies that $\inc_{\DI}^{n,k}$ is a bijection
\[ \cA(\hat\sigma_k^n,z) \to  \left\{ w \in \cED(z) : \ell( w v^{(k)}) = \ell(w) - \ell(v^{(k)})\right\}.\]
Let $w \in \cED(z)$. To finish the proof, we will show that 
$ \ell( w v^{(k)}) = \ell(w) - \ell(v^{(k)})$
if and only if \be
\label{chain-eq}
 |w_1| < (-1)^{k+1} w_2< (-1)^{k+2} w_3<  \dots  < w_{k-1} < -{w_k}.
 \ee
It follows from \eqref{d-descent-eq}  that 
$
\ell(w\cdot  t_{j-1}\cdots t_3\cdot t_2\cdot  t_1 \cdot t_{-1}\cdot t_2\cdot t_3 \cdots t_{j-1} ) = \ell(w) - (2j-2)
$
if and only if 
\be\label{inc-cond-eq} 
 \max\{|w_1|, |w_2|,\dots,|w_{j-1}|\} < -w_j.
\ee
This fact makes it easy to see that \eqref{chain-eq} implies $ \ell( w v^{(k)}) = \ell(w) - \ell(v^{(k)})$.

Conversely suppose that $ \ell( w v^{(k)}) = \ell(w) - \ell(v^{(k)})$.
Then \eqref{inc-cond-eq} must hold for all $1<j \leq k$ with $j \equiv k \modu 2)$ so
for all such indices $j$ we have
$(w_{j-1},w_j) \in \NDes(w)$ and $ \{ |w_{j-1}|, -w_j\} \in \shD(w)
$
along with $0 < \dots < -w_{k-4} < -w_{k-2} < -w_{k}. $ We deduce from
Lemma~\ref{ijkl-lem} that $w_{j-3} < w_{j-1}$ whenever $3<j \leq k$ has $j\equiv k \modu 2)$.

If $k \in \{0,1,2\}$ then \eqref{chain-eq} follows immediately.
Assume $k\geq 3$ is odd. Then we have $|w_1| < -w_3$ and $|w_{2}| < -w_3$ so $w_2>w_3$.
As we cannot have $|w_1| > w_2>w_3$ by Lemma~\ref{D1-lem0}
it follows that 
\[|w_1| < w_2 < -w_3 < -w_5 <-w_7< \dots < -w_k
\quand 0<w_2<w_4<w_6<\dots<w_{k-1}.\]
Since  $ \shD(w)$ is noncrossing
by Theorem~\ref{d1-prop1},
 the   chain of inequalities \eqref{chain-eq} must hold.
 
Finally assume $k\geq 4$ is even. Then $|w_1| < -w_2<-w_4$ and $|w_3| < -w_4$ and $w_1<w_3$.
If $w_1$ is negative then Lemma~\ref{cabd-lem} implies that $|w_1|<w_3$.
Thus $|w_1|<w_3$ so  $\{ |w_1|, -w_2\}$ and $\{w_3,-w_4\}$ are both blocks
of the noncrossing matching $\shD(w)$,
which means that we must have
\[|w_1| < -w_2< w_3 < -w_4 < -w_6 < -w_8<\dots <-w_k
\quand
|w_1| <w_3 < w_5 <\dots < w_{k-1}.
\]
The desired property \eqref{chain-eq} now follows again because $\shD(w)$ is noncrossing.
\end{proof}

\begin{corollary}\label{D1-well-cor}
For any $0\leq k\leq n$ the map $w=w_1w_2\cdots w_n \mapsto w_{k+1}w_{k+2}\cdots w_n$ is a bijection
from $\cA(\hat\sigma_k^n,z)$ to a well-nested family of partial permutations.
\end{corollary}

\begin{proof}
By Lemma~\ref{inc-lem1} this map is 
injective, and its image is well-nested by Corollary~\ref{well-cor}.
\end{proof}

For the rest of this section, choose $p,q \in \NN$ with $2n=p+q$ such that $k = \frac{|p-q|}{2} \in 2\NN$.
Assume 
$ z \in \I_{\DI}^{(p,q)} = \{ y \in \I(\WD_n): \neg(y) \geq k\}$
and recall that $\cA_{\DI}^{(p,q)}(z) =  \cAD(z:k) = \cA(\hat\sigma_k^n,z)$.
Below, we write $\precsim$ to denote the relation $\overset{0}\precsim$ from Section~\ref{word-rel-sect}.
\begin{lemma}\label{D1-t-lem}
Assume that $k >0$ is even and 
 $w \in \cAD(z:k)$.
 \ben
 \item[(a)] If $w_{k+1}<0$ then $|w_{k+1}| < |w_{1}| $.

\item [(b)] If $(a,b) \in \NDes(w_{k+1} \cdots w_n)$ has $|a|<-b$ then no $i\in [k]$ exists with $|a| < |w_i|<-b$.
\een
\end{lemma}

\begin{proof}
Let $v = \inc_{\DI}^{n,k}(w)$. 
For part (a), suppose $w_{k+1}<0$ and $|w_{k+1}| > |w_{1}| $. 
Then $v_{k+1} <0$ and $v_{k+1} <- |v_1|$.
Since 
$
 |v_1| < -v_2< v_3<  \dots  < v_{k-1} < -v_k
 $ by Lemma~\ref{inc-lem1},
 it follows from Corollary~\ref{well-cor} that
 there exists an element $u\in \cED(z)$ with $u\precsim v$
and $u_1u_2u_3 = v_1v_2 v_{k+1}$.
But then either $u_1>u_2>u_3$, or   $u_2<u_3<|u_1|$ while 
$|u_3|=|v_{k+1}|=|w_{k+1}|>|w_1| = |v_1|=|u_1|.$
Both cases contradict Lemma~\ref{D1-lem0},
so if $w_{k+1}<0$ then we must have $|w_{k+1}|<|w_1|$.

For part (b), let $M = \shD(v)$.
By Lemma~\ref{inc-lem1}, this matching consists of the blocks 
\[\pm \{ |w_1|<|w_2|\},\ \pm \{ |w_3|< |w_4|\},\ \dots,\ \pm\{|w_{k-1}|< |w_{k}|\}\]
together with $\pm \{|a|,-b\}$ for each $(a,b) \in \NDes(v_{k+1}\cdots v_n)= \NDes(w_{k+1}\cdots w_n)$
with $|a|<-b$.
Suppose $(a,b)\in  \NDes(w_{k+1}\cdots w_n)$ is a pair with $|a| < |w_i| < -b$ for some $i \in [k]$.
Since $M$ is noncrossing by Theorem~\ref{d1-prop1}(b),
the index $i$ must be odd and
  $|a| < |w_i| < |w_{i+1}|<-b$.
Then, since $|v_i|=|w_i|$ and $
 |v_1| < -v_2< v_3<  \dots  < v_{k-1} < -v_k
 $,
it follows that 
$a<|v_i|$ and $b < v_{i+1}$. The last two
inequalities contradict  Lemma~\ref{ijkl-lem}
unless $i=1$ and $v_1<0$, since outside this case we would have  $v_i=|v_i|$
and $(a,b),(v_i,v_{i+1}) \in \NDes(v)$.

We are left to consider the situation when $v_1<0$ and some
 $(a,b) \in \NDes(v_{k+1}\cdots v_n)$ has
$|a| < |v_1| < |v_2| = -v_2 < -b$. However, Lemma~\ref{cabd-lem} says precisely
that this case cannot occur.
\end{proof}

%In the following proofs 
Let $X=\Neg(z)$ and recall that $ \cM_{\DI}^{(p,q)}(z)$ is the image of $ \NCSP(X:k)$
under the map $M\mapsto M^\div$
that replaces each symmetric block $\{\pm i\}$ by the trivial blocks $\pm\{i\}$.

\begin{proposition}\label{d1-prop2}
%Suppose $z \in\I_{\DI}^{(p,q)}$ and
If $w \in \cA_{\DI}^{(p,q)}(z)$ when $k=\frac{|p-q|}{2}$ is even then $\sh_{\DI}^{(p,q)}(w)  \in \cM_{\DI}^{(p,q)}(z).$
\end{proposition}

\begin{proof}
When $k=0$ this result holds by Theorem~\ref{d1-prop1}, so assume $k>0$.
Let $v = \inc_{\DI}^{n,k}(w)$ and $M = \sh_{\DI}^{(n,n)}(v)$.
Then $M \in \NCSP(X:0)$ by Theorem~\ref{d1-prop1},
and $\sh_{\DI}^{(p,q)}(w) $ is formed from $M$ by replacing 
the blocks $\pm \{ |w_1|, w_2\}, \pm \{ w_3, w_4\},\dots, \pm \{w_{k-1}, w_k\}$
with the trivial blocks $\pm\{ w_1\}, \pm\{ w_2\},\dots, \pm\{ w_k\}$.
Lemma~\ref{D1-t-lem}(b) implies that this shape belongs to $ \cM_{\DI}^{(p,q)}(z)$
as all other blocks in $M$ have the form $\pm\{|a|,-b\}$ for pairs $(a,b) \in \NDes(w_{k+1}\cdots w_n)$ with $|a|<-b$.
\end{proof}

\begin{proposition}\label{d1-prop3}
Suppose $z \in\I_{\DI}^{(p,q)}$ and $M\in\cM_{\DI}^{(p,q)}(z)$  when $k = \frac{|p-q|}{2}$ is even. Then
\[\bot_{\DI}^{(p,q)}(z,M) \in  \cA_{\DI}^{(p,q)}(z)
\quand
\sh_{\DI}^{(p,q)}\(\bot_{\DI}^{(p,q)}(z,M)\) = M.\]
\end{proposition}

\begin{proof}
By Theorem~\ref{d1-prop1}(a) we may assume $k>0$.
The second identity is clear from the definitions. 
Write $\Triv{M} = \{i_1<i_2<\dots<i_k\}$. We can 
form an element $\tilde M \in \NCSP(X:0)$ from $M$ by replacing its trivial blocks with
$\pm \{i_1,i_2\}, \pm \{i_3,i_4\},\dots,\pm\{i_{k-1},i_k\}$.
Since every block $\{a,b\} \in M$ with $0<a<b$ and $a<i_j$ for some $j \in [k]$
must have $a<b<i_j$, it follows that 
\[ \inc_{\DI}^{n,k}\(\bot_{\DI}^{(p,q)}(z,M)\) \precsim^{(n,n)}_{\DI} \bot_{\DI}^{(n,n)}(z,\tilde M) \in \cED(z).\]
Thus $w=\inc_{\DI}^{n,k}\(\bot_{\DI}^{(p,q)}(z,M)\)$
 belongs to $\cED(z)$ by Theorem~\ref{d-equiv-lem}.
As by definition
$ |w_1| < - w_2<  w_3< \dots  < w_{k-1} < -{w_k}$,
 Lemma~\ref{inc-lem1} implies that $\bot_{\DI}^{(p,q)}(z,M) \in \cAD(z:k)$.
\end{proof}

Propositions~\ref{d1-prop2} and \ref{d1-prop3} imply that
$\sh_{\DI}^{(p,q)} : \cA_{\DI}^{(p,q)}(z) \to \cM_{\DI}^{(p,q)}(z)$ is a surjective map when $k = \frac{|p-q|}{2}$  is even.
The following result establishes Theorem~\ref{main-thm} in type DI.

 \begin{theorem}\label{D1-alignment-prop}
Suppose $\gamma \in \Gamma_{\DI}^{(p,q)}$ and 
$z=\psi_{\DI}^{(p,q)}(\gamma) \in \cI_{\DI}^{(p,q)}$
 when $k = \frac{|p-q|}{2}$ is even. Then
 \[
 \cW_{\DI}^{(p,q)}(\gamma) = \left\{ w \in \cA_{\DI}^{(p,q)}(z) : \sh_{\DI}^{(p,q)}(w) \text{ is $\gamma$-aligned}
 \right\}.
 \]
\end{theorem}

 \begin{proof} 
 Fix $w \in \cA_{\DI}^{(p,q)}(z) $ with a reduced expression $w=t_{i_1}t_{i_2}t_{i_3}\cdots t_{i_m}$ and consider the sequence
 \[
 P = \(  z=z^0 \xrightarrow{t_{i_1}} z^1 \xrightarrow{t_{i_2}} z^2 \xrightarrow{t_{i_3}}  \cdots \xrightarrow{t_{i_m}} z^m = \hat\sigma_k^n\).
 \]
 Lemma~\ref{weak-order-lem} asserts that $w \in \cW_{\DI}^{(p,q)}(\gamma) $ if and only if 
\be\label{align-equiv0-eq}
 |S_+(\gamma) \cap \{a,b\}| = |S_-(\gamma) \cap \{a,b\}| = 1\text{ for all $\{a,b\} \in \Lambda(P)$}.
 \ee

Now let $v  = \inc_{\DI}^{n,k}(1)$  
 so that $wv \in \cED(z)$, and define $\Upsilon_j $ for $j \in [m+1]$ as in \eqref{2Lam-eq}.
 Lemma~\ref{inc-lem1} implies that $|w_1| < w_2 <\dots<w_k$ and that  $\sh_{\DI}^{(p,q)}(w)$
 is obtained from the matching $ \shD(wv) = \sh_{\DI}^{(n,n)}(wv)$ 
 by replacing 
 $\pm \{ |w_1|, w_2\}, \pm \{ w_3, w_4\},\dots, \pm \{w_{k-1}, w_k\}$
with the $2k$ trivial blocks $\pm \{w_i\}$ for $i \in [k]$.
Since
$ \shD(v) = \{ \pm \{1,2\}, \pm \{3,4\},\dots,\pm\{k-1,k\}\}$ we have
 \[
\Upsilon_{m+1}=  \{ \pm \{ |w(a)|, |w(b)|\} : \{a,b\} \in \shD(v) \} = \{ \pm \{ |w_1|, w_2\}, \pm \{ w_3, w_4\},\dots, \pm \{w_{k-1}, w_k\}\}.
  \]
Thus, the union $\Upsilon := \Upsilon_1\sqcup\Upsilon_2\sqcup \cdots \sqcup\Upsilon_m$ 
is just $\sh_{\DI}^{(p,q)}(w)$ with its trivial blocks removed,
so the matching
$ \sh_{\DI}^{(p,q)}(w)$ is $\gamma$-aligned
if and only if \eqref{align-equiv0-eq} holds with $\Upsilon$ in place of $\Lambda(P)$.

These two versions of \eqref{align-equiv0-eq}  are equivalent,
since 
$\Upsilon = \{ \pm \{ |a|, |b|\} : \{a,b\} \in \Lambda(P)\}
$
and the clan $\gamma$ is symmetric.
Hence, $w \in \cW_{\DI}^{(p,q)}(\gamma) $ 
if and only if $ \sh_{\DI}^{(p,q)}(w)$ is $\gamma$-aligned.
 \end{proof}

Recall the definition of $\rankD$ in \eqref{rankD-def}.
For any $w \in \WD_n$ and $0\leq q \leq n$ let
\be\label{rankD-def2}
\rankD^{q}(w) =
\begin{cases}
 \rankD(w_{q+1}w_{q+2}\cdots w_n) + \ell_0(w_{q+1}w_{q+2}\cdots w_n) &\text{if }q>0
 \\
  \rankD(w) &\text{if }q=0.
  \end{cases}
\ee

\begin{proposition}\label{D1-rank-prop}
Assume $z \in \cI_{\DI}^{(p,q)}$ when $k=\frac{|p-q|}{2}$ is even.

\ben
\item[(a)] If $v,w\in \WD_n$ have $v \precsim_{\DI}^{(p,q)} w$ then
$v \in \cA_{\DI}^{(p,q)}(z)$ if and only if $w \in  \cA_{\DI}^{(p,q)}(z)$,
and in this case
\[\sh_{\DI}^{(p,q)}(v) = \sh_{\DI}^{(p,q)}(w).\]

\item[(b)] The relation $\precsim_{\DI}^{(p,q)}$ is a graded partial order on $ \cA_{\DI}^{(p,q)}(z) $
for the rank function $  \rankD^k$.
\een
\end{proposition}

\begin{proof}
By Theorem~\ref{rankD-prop1} and Lemma~\ref{same-shape-lem} we may assume that $k>0$.
In view of earlier results like
 Theorem~\ref{d-equiv-lem}, Lemma~\ref{inc-lem1} implies that if $v \overset{k}\precsim w$ then 
$v \in \cA_{\DI}^{(p,q)}(z)$ if and only if $w \in  \cA_{\DI}^{(p,q)}(z)$.
If $v$ and $w$ are related by a covering relation of the form 
$
\cdots BCA\cdots \overset{k}\precsim \cdots CAB\cdots 
$
 with $A<B<C$  
then
it is straightforward  to check that 
$
 \rankD^k(w) =  \rankD^k(v) + 1.
 $ The argument is similar to the one in the proof of Theorem~\ref{rankD-prop1}.

Suppose instead that $v$ and $w$ are related by the other covering relation for $\precsim_{\DI}^{(p,q)}$ so that
$0<-w_{k+1}<|w_1|$ and $ v=\overline{w_1}w_2\cdots w_k \overline{w_{k+1}} w_{k+2}\cdots w_n.$
 Then Theorem~\ref{d-equiv-lem}, Lemma~\ref{inc-lem1},  and Corollary~\ref{well-cor} 
imply that some $u \in \WD_n$ has
$
u \precsim_{\DI}^{(n,n)} \inc_{\DI}^{n,k}(v)$ and $ u \precsim_{\DI}^{(n,n)} \inc_{\DI}^{n,k}(w).
$ Specifically, this element is given in one-line notation as
\[u =
\begin{cases}
w_{k+1} w_1\overline{w_2} w_3\overline{w_4}w_5\overline{w_6}\cdots w_{k-1}\overline{w_k}w_{k+2}w_{k+3}\cdots w_n 
&\text{if $k\equiv 0 \modu 4)$ and $w_1>0$} 
\\
\overline{w_{k+1}} \overline{w_1}\overline{w_2} w_3\overline{w_4}w_5\overline{w_6}\cdots w_{k-1}\overline{w_k}w_{k+2}w_{k+3}\cdots w_n 
&\text{if $k\equiv 0 \modu 4)$ and $w_1<0$} 
\\
\overline{w_{k+1}} w_1\overline{w_2} w_3\overline{w_4}w_5\overline{w_6}\cdots w_{k-1}\overline{w_k}w_{k+2}w_{k+3}\cdots w_n 
&\text{if $k\equiv 2 \modu 4)$ and $w_1>0$} 
\\
w_{k+1} \overline{w_1}\overline{w_2} w_3\overline{w_4}w_5\overline{w_6}\cdots w_{k-1}\overline{w_k}w_{k+2}w_{k+3}\cdots w_n 
&\text{if $k\equiv 2 \modu 4)$ and $w_1<0$} .
\end{cases}
\]
By Theorem~\ref{d-equiv-lem} 
we have $\inc_{\DI}^{n,k}(v) \in \cED(z)$ if and only if $\inc_{\DI}^{n,k}(w) \in \cED(z)$
 so by Lemma~\ref{inc-lem1} 
we again have $v \in \cA_{\DI}^{(p,q)}(z)$ if and only if $w \in  \cA_{\DI}^{(p,q)}(z)$. 

When this occurs, it follows using Corollary~\ref{D1-well-cor}
that $(v_{k+1}\cdots v_n)_R = (w_{k+1}\cdots w_n)_R$
and that $(w_{k+1}\cdots w_n)_L^\pm$ is obtained from $(v_{k+1}\cdots v_n)_L^\pm$
by changing its middle two letters from $w_{k+1} \overline{w_{k+1}}$ to $\overline{w_{k+1}}w_{k+1}$,
which increases the number of inversions by one since $w_{k+1} < 0$. The negativity of $w_{k+1}$ also implies that
 $\ell_0(w_{k+1}\cdots w_n) = \ell_0(v_{k+1}\cdots v_n) +1$, so
we conclude from   \eqref{rankD-def} and \eqref{rankD-def2} that
$ \rankD^k(w) =  \rankD^k(v) + 1$.

Finally,  when $v,w \in \cA_{\DI}^{(p,q)}(z)$ have $v \precsim_{\DI}^{(p,q)} w$ 
 it is easy to see that $\sh_{\DI}^{(p,q)}(v) = \sh_{\DI}^{(p,q)}(w)$
 from the definitions using Proposition~\ref{ndes-prop} and Corollary~\ref{D1-well-cor}.
\end{proof}

The only part of Theorem~\ref{main-thm2} left to prove in type $\DI$ when $\frac{|p-q|}{2}$ is even is the following.

\begin{theorem}\label{DI-last-prop}
Suppose $z  \in \cI_{\DI}^{(p,q)}$ and $M \in \cM_{\DI}^{(p,q)}(z)$ when $k=\frac{|p-q|}{2}$ is even.
Then 
\[
 \Bigl\{w \in \cA_{\DI}^{(p,q)}(z)  :\sh_{\DI}^{(p,q)}(w) =M\Bigr\} 
 = \Bigl\{w \in \WD_n : \bot_{\DI}^{(p,q)}(z,M) \precsim_{\DI}^{(p,q)} w\Bigr\}.
 \]
\end{theorem}

\begin{proof}
The right hand set is contained in the left by Propositions~\ref{d1-prop3} and \ref{D1-rank-prop},
and the reverse inclusion holds when $k=0$ by Corollary~\ref{k0-precsim-cor}.

Assume $k>0$ and suppose $w \in \cA_{\DI}^{(p,q)}(z)$ has $\sh_{\DI}^{(p,q)}(w) =M$.
Let $v \in \WD_n$ be a minimal element with $v \precsim_{\DI}^{(p,q)} w$.
Then it follows from  Proposition~\ref{Pdes-prop} and Corollary~\ref{D1-well-cor}
that  $v_{k+1}\cdots v_n = [[P]]_{\des}$ for $
P = \{ (a,b) : (b,a) \in \NDes(v_{k+1}\cdots v_n)\}\sqcup \{(c,c) : c \in \NRes(v_{k+1}\cdots v_n)\}.
$
It then follows from Lemma~\ref{D1-t-lem}(a) that every $(a,b) \in P$ must have $b>0$.
Given these properties, we deduce from Theorem~\ref{d1-prop1} and Lemma~\ref{inc-lem1}
that $v = \bot_{\DI}^{(p,q)}(z,M)$.
We conclude that even when $k>0$ our two expressions 
are  
contained in each other and hence equal.
\end{proof}

 \subsection{Proofs in type DI for odd parity cases}\label{tDI-sect2}
 
Below, we describe a way of embedding type DI of rank $n$ into   rank $n+1$.
This will let us quickly derive our main theorems for the other half of type DI from the results in the previous section.
  Throughout this section, we fix $p,q \in \NN$ with $2n=p+q$ such that $k = \frac{|p-q|}{2} $ is   odd.
Let 
\[z \in \I_{\DI}^{(p,q)} = \{ y \in \I_\diamond(\WD_n): \neg(t_0y) \geq k\}\] and recall
 that 
 $\cA_{\DI}^{(p,q)}(z) =  \cAD(z:k)  = \cA_\diamond(\hat \sigma_k^n, z)$
for the element $\hat \sigma_k^n=1\hs \overline{2}\hs\overline{3}\cdots \overline{k}(k+1)\cdots n \in \WD_n$.
For any $y \in \I_\diamond(\WD_n)$ define $y^\vee\in \I(\WD_{n+1})$ to be the even-signed permutation 
 with
 \be\label{iotaD-eq}
y^\vee : i \mapsto  \begin{cases} 
 y_i &\text{if $i \in [n]$ and $|y_i| \neq 1$} \\ 
- y_i &\text{if $i \in [n]$ and $|y_i| = 1$}   \\
-n-1 &\text{if $i=n+1$}
\end{cases}
\quad\text{for }i \in [n+1].
\ee

\begin{lemma}\label{inc-lem-D-pre}
If $w \in \WD_n$ and $ y \in \I_\diamond(\WD_n)$
then 
$ (w^\diamond \circ y \circ w^{-1})^\vee = w \circ y^\vee \circ w^{-1}.$
\end{lemma}

\begin{proof} It suffices to check that
$ (t_i^\diamond \circ y \circ t_i)^\vee = t_i \circ y^\vee \circ t_i$ for $i \in \{-1,1,2,\dots,n-1\}$
and this follows from the formulas in Section~\ref{d-conj-sect}.
\end{proof}

Additionally, for each 
$
\gamma=(S_+,S_-,M) \in \Gamma_{\DI}^{(p,q)} = \{\text{ symmetric $(p,q)$-clans } \}
$ 
define 
\be
\gamma^\vee = 
\begin{cases}
(S_+\sqcup \{-n-1,n+1\}, S_-,M)&\text{if }p\geq q \\
(S_+,S_-\sqcup \{-n-1,n+1\},M)&\text{if }p< q.
\end{cases}
\ee
Then we have $\gamma^\vee \in \Gamma_{\DI}^{(p^\vee,q^\vee)}$ 
when we define
\be
p^\vee = \begin{cases} p+2 &\text{if }p \geq q \\
p&\text{if } p < q
\end{cases}
\quand
q^\vee = \begin{cases} q &\text{if }p \geq q \\
q+2&\text{if } p < q.
\end{cases}
\ee
Moreover, if $\psi_{\DI}^{(p,q)}(\gamma) = z$ then 
$
\psi_{\DI}^{(p^\vee,q^\vee)}(\gamma^\vee) = z^\vee.
$

  \begin{lemma}\label{inc-lem-DI-odd}
  Assume $k=\frac{|p-q|}{2}$ is odd.
Then the injective map  $\iota_k^n :\WD_n \to \WD_{n+1}$ given by
\[\iota_k^n : w \mapsto 
w_1 w_2\cdots w_k (n+1) w_{k+1} w_{k+2}\cdots w_n
\]
restricts to a bijection 
$\cA_{\DI}^{(p,q)}(z) =\cAD(z:k) \to  \left\{ w \in\cA_{\DI}^{(p^\vee,q^\vee)}(z^\vee)= \cAD(z^\vee: k+1) :  w_{k+1} = n+1\right\}.$
Moreover, if $w \in \WD_n$ and $\gamma \in  \Gamma_{\DI}^{(p,q)}$
then
$
w \in \cW_{\DI}^{(p,q)}(\gamma)
$ if and only if
$\iota_k^n(w) \in \cW_{\DI}^{(p^\vee,q^\vee)}(\gamma^\vee).
$
\end{lemma}

\begin{proof}
Identify $\WD_n$ with the set of elements in $\WD_{n+1}$ that fix $n+1$. 
Then 
$\iota_k^n (w) = wv=w\circ v$ for $v= t_n \cdots t_{k+3}t_{k+2} t_{k+1}$ 
and this map is   a bijection $\WD_n \to \{ u \in \WD_{n+1} : u_{k+1} = n+1\}$
with \[\ell(\iota_k^n (w)) = \ell(w) + \ell(v) = \ell(w) + n-k.\]
One can check that $v$ is a minimal-length element with
$ v\circ \hat \sigma_{k+1}^{n+1} \circ v^{-1} = (\hat\sigma_{k}^{n})^\vee 
$.
Since $y \mapsto y^\vee$ is an injective operation and since
$
w\circ  (\hat\sigma_{k}^{n})^\vee \circ w^{-1} =  (w^\diamond \circ \hat\sigma_{k}^{n}\circ w^{-1})^\vee 
$ for all $w \in \WD_n
$
by Lemma~\ref{inc-lem-D-pre},  
Lemma~\ref{dem-lem0} implies
that 
$w \in \cAD(z:k) $ if and only if $\iota_k^n(w) \in  \cAD(z^\vee: k+1)$.

This proves the first claim.
For the second claim, we first note that Theorem~\ref{weak-order-thm} and Lemma~\ref{inc-lem-D-pre} imply
that $\gamma \xrightarrow{t_i} \delta$ is an edge in the weak order graph for type DI if and only if 
 $\gamma^\vee \xrightarrow{t_i} \delta^\vee$ is also an edge (in higher rank).
 Since, in the notation of Section~\ref{weak-order-sect},
$
 (\hat\sigma_{k}^{n})^\vee \xrightarrow{t_{n}}\cdots  \xrightarrow{t_{k+1}}
\hat \sigma_{k+1}^{n+1}
$
always lifts to a path in the weak order graph for type DI,
we deduce from Lemma~\ref{weak-order-lem}
that if $\gamma \in  \Gamma_{\DI}^{(p,q)}$ then
$w \in \cW_{\DI}^{(p,q)}(\gamma)
$ if and only if 
$wv \in \cW_{\DI}^{(p^\vee,q^\vee)}(\gamma^\vee)$.
\end{proof}

Let $X = \Neg(t_0 z)$. Recall that $\cM_{\DI}^{(p,q)}(z) $ is the image of $ \NCSP(X : k)$ under $M\mapsto M^\div$.
 The following result is equivalent to Theorem~\ref{main-thm} in type $\DI$ when 
$\frac{|p-q|}{2}$ is odd.

\begin{theorem}\label{D2-prop1}
Let $\gamma \in \Gamma_{\DI}^{(p,q)}$,
$z =\psi_{\DI}^{(p,q)}(\gamma)$,
 and 
 $w \in \WD_n$ when $k = \frac{|p-q|}{2}$ is odd.  
 Then :
 \ben
 \item[(a)] If $w \in \cA_{\DI}^{(p,q)}(z)$ then  $\sh_{\DI}^{(p,q)}(w)  \in \cM_{\DI}^{(p,q)}(z).$
 
 \item[(b)] If $M\in\cM_{\DI}^{(p,q)}(z)$ then $\bot_{\DI}^{(p,q)}(z,M) \in  \cA_{\DI}^{(p,q)}(z)
$ and $
\sh_{\DI}^{(p,q)}\(\bot_{\DI}^{(p,q)}(z,M)\) = M$.

\item[(c)] One has $ \cW_{\DI}^{(p,q)}(\gamma)= \left\{ w \in  \cA_{\DI}^{(p,q)}(z):  \sh_{\DI}^{(p,q)}(w) \text{ is $\gamma$-aligned}\right\}$.
 \een
 \end{theorem}
 
 \begin{proof}
 Note that  $z =t_0 \cdot \overline{\sigma_\gamma} \in \cI_{\DI}^{(p,q)}.$
 For part (a), suppose $w \in \cA_{\DI}^{(p,q)}(z)$. 
Then
 \be\label{be-eq}
 \sh_{\DI}^{(p^\vee,q^\vee)}(\iota_k^n(w)) = \sh_{\DI}^{(p,q)}(w) \sqcup \Bigl\{ \{-n-1\},\{n+1\}\Bigr\}.
 \ee
Since by Proposition~\ref{d1-prop2} and Lemma~\ref{inc-lem-DI-odd} this matching  belongs to 
 \be\label{cMDI-eq}
 \cM_{\DI}^{(p^\vee,q^\vee)}(z^\vee) = \Bigl \{ M^\div : M \in \NCSP(X \sqcup \{n + 1\} : k+1)\Bigr\},
 \ee
 it is evident that $\sh_{\DI}^{(p,q)}(w) \in \cM_{\DI}^{(p,q)}(z)$.
 
 For part (b), notice that if $M\in\cM_{\DI}^{(p,q)}(z)$ and $M^\vee \defequals M   \sqcup \{ \{-n-1\},\{n+1\}\}$ then 
 \be\label{inc-gen-DI-odd}
 \iota_k^n\(\bot_{\DI}^{(p,q)}(z,M)\) = \bot_{\DI}^{(p^\vee,q^\vee)}(z^\vee,M^\vee)
.\ee This element is in $\cA_{\DI}^{(p^\vee,q^\vee)}(z^\vee)$
by Proposition~\ref{d1-prop3}
since $M^\vee $ belongs to \eqref{cMDI-eq},
so it follows from Lemma~\ref{inc-lem-DI-odd} that $\bot_{\DI}^{(p,q)}(z,M) \in  \cA_{\DI}^{(p,q)}(z)$.
We can compute that 
$
\sh_{\DI}^{(p,q)}\(\bot_{\DI}^{(p,q)}(z,M)\) =M
$ either 
 by combining Proposition~\ref{d1-prop3} with \eqref{be-eq} or by a direct computation.
 
 To prove part (c), it suffices by Theorem~\ref{D1-alignment-prop} and Lemma~\ref{inc-lem-DI-odd}
 to show for $w \in  \cA_{\DI}^{(p,q)}(z)$ that 
 $\sh_{\DI}^{(p,q)}(w)$ is $\gamma$-aligned
if and only if
$\sh_{\DI}^{(p^\vee,q^\vee)}(\iota_k^n(w))$
 is $\gamma^\vee$-aligned.
 This is clear from \eqref{be-eq} since singleton blocks have no impact on alignment. 
 \end{proof}

  The next result is equivalent to Theorem~\ref{main-thm2} in type $\DI$ when $\frac{|p-q|}{2}$ is odd.
  
\begin{theorem}\label{D2-rank-prop}
Assume $z \in \cI_{\DI}^{(p,q)}$ when $k=\frac{|p-q|}{2}$ is odd.
\ben
\item[(a)] If $v,w\in \WD_n$ have $v \precsim_{\DI}^{(p,q)} w$ then
$v \in \cA_{\DI}^{(p,q)}(z)$ if and only if $w \in  \cA_{\DI}^{(p,q)}(z)$,
in which case
\[\sh_{\DI}^{(p,q)}(v) = \sh_{\DI}^{(p,q)}(w).\]

\item[(b)] The relation $\precsim_{\DI}^{(p,q)}$ is a graded partial order on $ \cA_{\DI}^{(p,q)}(z) $
for the rank function $  \rankD^k$.

\item[(c)] For all $M \in \cM_{\DI}^{(p,q)}(z)$ one has
\[
 \Bigl\{w \in \cA_{\DI}^{(p,q)}(z)  :\sh_{\DI}^{(p,q)}(w) =M\Bigr\} 
 = \Bigl\{w \in \WD_n : \bot_{\DI}^{(p,q)}(z,M) \precsim_{\DI}^{(p,q)} w\Bigr\}
.\]
\een
\end{theorem}

\begin{proof}
If we order
$\WD_n $ by $ \precsim_{\DI}^{(p,q)}$ and 
$\WD_{n+1} $ by $ \precsim_{\DI}^{(p^\vee,q^\vee)}$ 
then the map
 $
 \iota_k^n$
is order-preserving. 
Part (a) follows by combining this observation with
Proposition~\ref{D1-rank-prop} and Lemma~\ref{inc-lem-DI-odd},
using \eqref{be-eq} for the identity $\sh_{\DI}^{(p,q)}(v) = \sh_{\DI}^{(p,q)}(w)$ when $v \precsim_{\DI}^{(p,q)} w$.

The same results imply that if $w\in \cA_{\DI}^{(p,q)}(z)$ then 
$  \rankD^k(w) =  \rankD^{k+1}\( \iota_k^n(w)\).$
The image of $\cA_{\DI}^{(p,q)}(z)$ under  $
 \iota_k^n$ is a union of upper sets in  $\cA_{\DI}^{(p^\vee,q^\vee)}(z^\vee)$,
 which are graded by $\rankD^{k+1}$ and whose minimal elements have the form \eqref{inc-gen-DI-odd}
 for $M \in \cM_{\DI}^{(p,q)}(z)$ by Proposition~\ref{D1-rank-prop}
 and Theorem~\ref{DI-last-prop}.
Parts (b) and (c) follow from these observations.
\end{proof}

 \subsection{Proofs in type DIII}\label{tDIII-sect}
 
  The propositions in this section will establish Theorems~\ref{main-thm} and \ref{main-thm2} in type DIII.
 When $n$ is even,  our proofs will use an embedding of type DIII into type DI,
 and when $n$ is odd we will use an embedding into even rank $n+1$.
  Throughout, we fix an element $z \in \cI_{\DIII}^{(n)}$.
  
  When $n$ is even, we have $z \in \Ifpf(\WD_n)$ and
 $\cA_{\DIII}^{(n)}(z)=\cAfpfD(z) = \cA(t_1t_3t_5\cdots t_{n-1}, z)$.

 \begin{lemma}\label{inc-lem2}
Assume $n$ is even. Then the injective map $\incDIIIDI :\WD_n\to\WD_n$ given by 
\[
\incDIIIDI : w \mapsto 
w_2w_1w_4w_3\cdots w_n w_{n-1} 
\]
restricts 
to a bijection 
$\cAfpfD(z) \to  
 \left\{ v \in \cED(z) :  v_i > v_{i+1}\text{ for all odd }i \in [n-1]\right\}.
$
\end{lemma}

\begin{proof}
This follows from Lemma~\ref{dem-lem} with $v =y= t_1t_3t_5\cdots t_{n-1}$
as $\cAfpfD(z) = \cA(y,z)$.
\end{proof}

  When $n$ is odd, we have $z \in \cI_\diamond(\WD_n)$ and
$   t_0 z \in \Ifpf(\W_n)$
   and
 \[\cA_{\DIII}^{(n)}(z) = \cAfpfD(z) = \cA_\diamond(t_2t_4t_6\cdots t_{n-1}, z).\]
 For this case, 
define $z^\vee = (-n-1,n+1)\cdot (-1,1)\cdot z \in \WD_{n+1}$ as in \eqref{iotaD-eq}. Then
\[
z^\vee \in \left\{ y \in \Ifpf(\WD_{n+1}) : \neg(y)>0 \right\}
\subseteq  \cI_{\DIII}^{(n+1)} 
\]
and $\cA_{\DIII}^{(n+1)}(z^\vee)= \cAfpfD(z^\vee) = \cA (t_1 t_3  t_5\cdots t_{n}, z^\vee).$
 Now, for each
 \[
\gamma=(S_+,S_-,M) \in \Gamma_{\DIII}^{(n)} = \{\text{ even-strict skew-symmetric $(n,n)$-clans } \}
\] we redefine our earlier notation by setting
\[
\gamma^\vee = 
 (S_+\sqcup\{-n-1\},S_-\sqcup \{n+1\},M)
.\]
When $n$ is odd
 we have $\gamma^\vee \in    \Gamma_{\DIII}^{(n+1)}$ for all $\gamma \in \Gamma_{\DIII}^{(n)} $
 and
if $\psi_{\DIII}^{(n)}(\gamma) = z$  then 
$
\psi_{\DIII}^{(n+1)}(\gamma^\vee) = z^\vee$.

  \begin{lemma}\label{inc-lem-DIV}
  Assume $n$ is odd. Then the injective map  $\incDIIIEven :\WD_n \to \WD_{n+1}$ given by
\[\incDIIIEven : w \mapsto 
(-n-1)\overline{w_1} w_2w_3\cdots  w_n
\]
is a bijection 
$\cAfpfD(z)\to  \left\{ v \in \cAfpfD(z^\vee) :  v_{1} = -n-1\right\}.$
Also, if $w \in \WD_n$ and $\gamma \in  \Gamma_{\DIII}^{(n)}$
then
\[
w \in \cW_{\DIII}^{(n)}(\gamma)
\quad\text{if and only if}
\quad
\incDIIIEven(w) \in \cW_{\DIII}^{(n+1)}(\gamma^\vee).
\]
\end{lemma}

\begin{proof} 
The proof is similar to that of Lemma~\ref{inc-lem-DI-odd}.
We again identify $\WD_n$ with the set of elements in $\WD_{n+1}$ that fix $n+1$. 
Then 
$\incDIIIEven (w) = wv=w\circ v$ for $v= t_n \cdots t_3 t_2 t_{-1}$ 
and this map is   a bijection $\WD_n \to \{ u \in \WD_{n+1} : u_{1} = -n-1\}$
with \[\ell(\incDIIIEven (w)) = \ell(w) + \ell(v) = \ell(w) + n.\]
One can check that $v$ is a minimal-length element with
\[v\circ t_1t_3t_5\cdots  t_n \circ v^{-1} = (t_2t_4t_6\cdots t_{n-1})^\vee 
.\]
It follows using
Lemmas~\ref{dem-lem0} and \ref{inc-lem-D-pre}
that 
$w \in \cAfpfD(z) $ if and only if $\incDIIIEven(w) \in  \cAfpfD(z^\vee)$.

Next, we can deduce for $\gamma \in  \Gamma_{\DIII}^{(n)}$ that $w \in \cW_{\DIII}^{(n)}(\gamma)
$ if and only if 
$wv \in \cW_{\DIII}^{(n+1)}(\gamma^\vee)$
from Lemma~\ref{weak-order-lem},
after 
using Theorem~\ref{weak-order-thm} and Lemma~\ref{inc-lem-D-pre} 
to note
that $\gamma \xrightarrow{t_i} \delta$ is an edge in the weak order graph for type DIII in rank $n$ if and only if 
 $\gamma^\vee \xrightarrow{t_i} \delta^\vee$ is an edge in the weak order graph for type DIII in rank $n+1$,
 and that 
$
 (t_2t_4t_6\cdots t_{n-1})^\vee
 \xrightarrow{t_{n}}  
 \cdots 
 \xrightarrow{t_{2}}
 t_{-1}t_1t_3t_5\cdots t_n
 \xrightarrow{t_{-1}} 
 t_1t_3t_5\cdots t_n
$
always lifts to a path in the weak order graph for type DIII in rank $n+1$.
\end{proof}

Let $X = \Neg(z)$ and $Y=\Neg(t_0 z)$.
Recall from Section~\ref{D34-sect} that when $n$ is even  
\[
 \cM_{\DIII}^{(n)}(z) = \{ \regroupD(M) : M\in\NCSP(X)\text{ with }  \triv{M} \equiv \ell_0(z) \modu 4)\}
 \]
 and when $n$ is odd we have 
\[
 \cM_{\DIII}^{(n)}(z) = \NCSPD(Y)=\{ \regroupD(M) : M\in\NCSP(Y)\}.
\]
Also recall the definitions of ${\overset{k}\precsim}$ and ${\overset{k}\precapprox}$ from Section~\ref{word-rel-sect},
and that ${\precsim_{\DIII}^{(n)}}={\overset{k}\precapprox}$ for $k=\frac{1-(-1)^n}{2}$.

\begin{proposition}\label{D34-prop1a}
Let $z   \in \cI_{\DIII}^{(n)}$. Suppose  $v,w\in \WD_n$
are such that $v  \precsim_{\DIII}^{(n)} w$. Then $v \in \cA_{\DIII}^{(n)}(z)$ if and only if $w \in \cA_{\DIII}^{(n)}(z)$,
and in this case one has $\sh_{\DIII}^{(n)}(v) = \sh_{\DIII}^{(n)}(w).$
 \end{proposition}
 \begin{proof}
First assume that $n$ is even.
If $v,w \in \WD_n$ have one-line representations 
\[v = \cdots BCAD \cdots {\overset{0}\precapprox} \cdots ADBC \cdots = w\]
where $A<B<C<D$ occur in positions $i,i+1,i+2,i+3$ for some odd $i \in [n-3]$,
then it is easy to see that
$\sh_{\DIII}^{(n)}(v) = \sh_{\DIII}^{(n)}(w)$, and 
 the images of $v$ and $w$ under $\incDIIIDI$ have the form
\[
\incDIIIDI(v) = \cdots CB DA \cdots \overset{0}\precsim \cdots C DA B \cdots \overset{0}\precsim \cdots DA CB \cdots = \incDIIIDI(w).
\]
In this case $\incDIIIDI(v) \in \cED(z)$ if and only if  $\incDIIIDI(w) \in \cED(z)$
by Proposition~\ref{D1-rank-prop} (with $p=q=n$),
so  Lemma~\ref{inc-lem2} implies that $v \in \cAfpfD(z)=\cA_{\DIII}^{(n)}(z)$ if and only if $w \in \cAfpfD(z)=\cA_{\DIII}^{(n)}(z)$.

Alternatively, if $n$ is odd and $v,w\in \WD_n$ have $v  \overset{1}\precapprox w$, then
 it is obvious that
 \[\sh_{\DIII}^{(n)}(v) = \sh_{\DIII}^{(n)}(w) 
\quand 
\incDIIIEven(v)  \overset{0}\precapprox \incDIIIEven(w),\]
so 
  Lemma~\ref{inc-lem-DIV} and the even rank case 
 imply that $v \in \cA_{\DIII}^{(n)}(z)$ if and only if $w \in \cA_{\DIII}^{(n)}(z)$.
\end{proof}

\begin{proposition}\label{D34-prop1b}
Let $z   \in \cI_{\DIII}^{(n)}$ and $w \in \cA_{\DIII}^{(n)}(z)$. Then  $\sh_{\DIII}^{(n)}(w)  \in \cM_{\DIII}^{(n)}(z).$
 \end{proposition}
 
 \begin{proof}
We again first assume that $n$ is even.
Suppose $v=\incDIIIDI(w) \in \cED(z)$
and $N=\sh_{\DI}^{(n,n)}(v)$.
Form $M$ from $N$ 
by replacing all blocks of the form $\pm \{b, c\}$ for $(b,c) \in \NDes(v)$
with $0 < -b < -c$ by the symmetric blocks $\{ \pm b\}$ and $\{\pm c\}$.
Since $N \in \cM_{\DI}^{(n,n)}(z)=\NCSM(X:0)$ by Proposition~\ref{d1-prop2},
it follows by Lemma~\ref{cabd-lem} that $M \in \NCSM(X)$.
Moreover, it is evident from Definition~\ref{D34-def}
that $\sh_{\DIII}^{(n)}(w) = \regroupD(M)$.

Let $l=\triv{M}$ be the number of symmetric blocks in $M$.
We must show that $l \equiv \ell_0(z) \modu 4)$.
Since $v_i>v_{i+1}$ and $w_i < w_{i+1}$ for all odd $i \in [n]$
by Lemma~\ref{inc-lem2},
we have 
$
 l =2| \{ \text{odd }i\in[n]: w_{i} < 0\text{ and }w_{i+1} < 0\}|
$
 along with
$\NDes(v) = \NDes^\pm(v)$.
It follows from Theorem~\ref{d1-prop1} that 
\[
\ba
\ell_0(z) &= 2 |\{ (a,b) \in \Cyc^{\pm}(z) : a<0\}| + \neg(z) 
\\&= 2 | \{ (a,b) \in \NDes(v) : b < 0\}|
= 2| \{ \text{odd }i\in[n]: w_{i} < 0\}|.
\ea
\]
As there are no odd indices $i \in [n]$ with $w_i > 0 > w_{i+1}$, we conclude that
\[ \tfrac{1}{2}\(\ell_0(z) - l\) =| \{ \text{odd }i\in[n]: w_{i} < 0<w_{i+1}\}|
=| \{ \text{odd }i\in[n]: w_{i} w_{i+1}<0\}|.\]
This number must be even since $n$ is even and $w \in \WD_n$, 
so $\ell_0(z) - l \equiv 0 \modu 4)$. 

Hence, the result holds 
if $n$ is even.
Assume $n$ is odd. 
Then  
$ \sh_{\DIII}^{(n+1)}\(\incDIIIEven(w)\)$ is 
 \be\label{be-eq3}
\( \sh_{\DIII}^{(n)}(w) \setminus \Bigl\{ \{-w_1,w_1\}\Bigr\} \)\sqcup \Bigl\{ \{ -n-1,w_1\}, \{-w_1,n+1\}\Bigr\}
.
\ee
It follows that 
having
\[ \sh_{\DIII}^{(n+1)}\(\incDIIIEven(w)\) \in \NCSPD(Y\sqcup \{n+1\})\quad\text{implies that}
\quad\sh_{\DIII}^{(n)}(w) \in \NCSPD(Y).\]
As $\Neg(z^\vee) = \Neg(t_0z) \sqcup \{n+1\} = Y\sqcup\{n+1\}$,
we deduce that  $\sh_{\DIII}^{(n)}(w)  \in 
\cM_{\DIII}^{(n)}(z)$ from Lemma~\ref{inc-lem-DIV} and the even rank case handled above.
\end{proof}

 \begin{proposition}\label{D34-prop1c}
Let $z   \in \cI_{\DIII}^{(n)}$ and 
 $M\in\cM_{\DIII}^{(n)}(z)$.
 Then 
 \[\bot_{\DIII}^{(n)}(z,M) \in  \cA_{\DIII}^{(n)}(z)
\quand
\sh_{\DIII}^{(n)}\(\bot_{\DIII}^{(n)}(z,M)\) = M.\]
 \end{proposition}
 
\begin{proof}
Let $w=\bot_{\DIII}^{(n)}(z,M)$.
In view of \eqref{ncspd-eq-eq} and Proposition~\ref{d2-uv-prop}, the claim that $\sh_{\DIII}^{(n)}(w) = M$ is immediate from Definitions~\ref{D34-def} and \ref{D34-def2}.
The harder claim to prove is that $w\in  \cA_{\DIII}^{(n)}(z)$. To show this, we first assume $n$ is even.

Write $\Triv{M} = \{c_1<c_2<\dots<c_l\}$
and note that $l $ is even and congruent to $\ell_0(z)$ modulo four. 
By Proposition~\ref{d2-uv-prop}, $w$ is given in one-line notation as the concatenation $uv$ where
$
u = \overline{c_l\cdots c_2c_1}
$
and
$
v = [[\CycPM(z,M)]]_\asc.
$
Consider the modified words
\[
\hat u = \begin{cases} c_1 \overline{c_2} c_3 \overline{c_4}\cdots c_{l-1} \overline{c_l} &\text{if $l\equiv 0 \modu 4)$} 
\\
\overline{c_1} \overline{c_2} c_3 \overline{c_4}\cdots c_{l-1} \overline{c_l}&\text{if $l\equiv 2 \modu 4)$}
\end{cases}
\quand
\hat v = [[\CycPM(z,M)]]_\des.
\]
Since $l\equiv \ell_0(z)\modu 4)$ and since the number of negative letters in $\hat v$ is 
\[ \tfrac{1}{2} ( \ell_0(z) - \neg(z)) + \tfrac{1}{2}(\neg(z) - l) = \tfrac{1}{2} ( \ell_0(z)-l),
\]
the number of negative letters in the word $\hat w = \hat u \hat v$  is whichever of $\tfrac{1}{2} \ell_0(z)$ or  $\tfrac{1}{2} \ell_0(z)+1$ is even.  
Thus $\hat w \in \WD_n$. 

Let $\hat M=\{ \pm \{|a|,|b|\} : \{a,b\} \in M\}\in \cMM(z)$ be the matching formed from $M$ by replacing the crossing blocks $\pm \{c_i,-c_{i+1}\}$ for each odd $i \in [l-1]$ by 
$\pm \{c_i,c_{i+1}\}$.
It is straightforward to check that $\bot_{\DI}^{(n,n)}(z,\hat M) \precsimD  \hat w \precsimD \incDIIIDI(w)$ so $ \incDIIIDI(w) \in \cED(z)=\cA_{\DI}^{(n,n)}(z)$
by Theorem~\ref{DI-last-prop}.
 As
$w_i < w_{i+1}$ for all odd $i \in [n]$,
 Lemma~\ref{inc-lem2} gives $w \in  \cA_{\DIII}^{(n)}(z)$.

Now assume $n$ is odd.
In this case,
to prove that $w=\bot_{\DIII}^{(n)}(z,M) $ is in $   \cA_{\DIII}^{(n)}(z)$ it suffices by
 Lemma~\ref{inc-lem-DIV}
to show that the element $\tilde w= \incDIIIEven(w) $ belongs to $ \cA_{\DIII}^{(n+1)}(z^\vee)$.
Let $N = \sh_{\DIII}^{(n+1)}(\tilde w)$,
and recall that this has the formula  \eqref{be-eq3}.
Notice that we respectively have
 \[  \bot_{\DIII}^{(n+1)}(z^\vee,N)=\tilde w\text{ when $w_1>0$}
\quand \bot_{\DIII}^{(n+1)}(z^\vee,N) \overset{0}\precapprox \tilde w \text{ when $w_1<0$}.\] 
We just need to check that $N \in \cM_{\DIII}^{(n+1)}(z^\vee)$,
as then it will follow from the even case that 
$\bot_{\DIII}^{(n+1)}(z^\vee,N)\in   \cA_{\DIII}^{(n+1)}(z^\vee)$,
from which Proposition~\ref{D34-prop1a} will imply that $\tilde w \in  \cA_{\DIII}^{(n+1)}(z^\vee)$.

We verify this containment as follows.
As $|w_1|=\max( \Triv{M})$,
we see that  $N$ is obtained by applying $\regroupD$ to some element of $ \NCSP( Y \sqcup\{n+1\})$. 
Observe that
$ \triv{N}=\triv{M} + 1 $ if $w_1>0 $ and 
$ \triv{N}=\triv{M}-1$ if $w_1<0$.
As 
$ \ell_0(z^\vee) = \ell_0(t_0 z) + 1$, we have 
\[\ell_0(z^\vee) - \triv{N} =\ell_0(t_0z)-\triv{M}+ 2|\{ w_1\} \cap \ZZ_{<0}|.
\]
Consulting the definition of $\bot_{\DIII}^{(n)}(z,M)$ shows that $w_1<0$ if and only if $\triv{M}$
plus the number of negative letters in
  $v= [[\CycPM( t_0z,M)]]_{\asc}$ from \eqref{uv-eq} is even. Counting the latter number 
  as in the proof of Proposition~\ref{d2-uv-prop}, we see that this occurs precisely when
\be\label{parity-consult-eq}
 \triv{M} +
\tfrac{\ell_0(t_0z) - |Y|}{2} +   \tfrac{|Y|- \triv{M}}{2} = \tfrac{\ell_0(t_0z) +\triv{M}}{2}    \equiv 0 \modu 2).\ee
As $\triv{M}$ is odd by \eqref{ncspd-eq-eq},
this is equivalent to requiring that
$
\ell_0(t_0z) -\triv{M}  \equiv 2 \modu 4).$
Since $\ell_0(t_0z)$ is also odd, we conclude that
$ \triv{N} \equiv \ell_0(z^\vee)\modu 4)$
so  $N \in \cM_{\DIII}^{(n+1)}(z^\vee)$.
 \end{proof}

The next proposition requires one additional lemma.

\begin{lemma}\label{ADBC-lem}
Assume $n$ is even, $z \in \Ifpf(\WD_n)$, and $w \in \cAfpfD(z)$.
Choose an odd index  $i \in[n-3]$ and suppose $w_{i+1} > w_{i+3}$. 
Then $w_iw_{i+1}w_{i+2}w_{i+3}=ADBC$ where $A<B<C<D$.

\end{lemma}

\begin{proof}
Lemma~\ref{inc-lem2} implies that  $w_i<w_{i+1}$ and $w_{i+2}<w_{i+3}$
so $A<D$ and $B<C<D$.
Since $\incDIIIDI(w) = \cdots DACB\cdots \in \cED(z)$ by the same lemma,
it follows from Corollary~\ref{well-cor}
 that $A<B$,
since if $A>C$ then $A>C>B$, and if $B<A<C$ then 
 $ DACB  \overset{0}\precsim DCBA$ and $D>C>B$.
\end{proof}

  The following result implies Theorem~\ref{main-thm2} in type DIII.

\begin{theorem}\label{D34-rank-prop}
Assume $z \in \cI_{\DIII}^{(n)}$ and let $M \in \cM_{\DIII}^{(n)}(z)$. Then
  $\precsim_{\DIII}^{(n)} $ is a graded partial order on the set $\cA_{\DIII}^{(n)}(z)$
and  
$
  \Bigl\{w \in  \cA_{\DIII}^{(n)}(z) : \sh_{\DIII}^{(n)}(w) =M\Bigr\}
 = \Bigl\{w \in \WD_n : \bot_{\DIII}^{(n)}(z,M) \precsim_{\DIII}^{(n)} w\Bigr\}
.$
\end{theorem}

\begin{proof}
Proposition~\ref{D34-prop1a} shows that 
 ${\precsim_{\DIII}^{(n)}}$ is a partial order on $\cA_{\DIII}^{(n)}(z)$.
 As noted in Section~\ref{word-rel-sect} this relation is always graded.
For the remaining claim, first observe that the right hand set is contained in the left by Propositions~\ref{D34-prop1a} and \ref{D34-prop1c}.
To prove the reverse containment,  suppose $w \in \cA_{\DIII}^{(n)}(z)$
has 
$\sh_{\DIII}^{(n)}(w) =M$.

First assume $n$ is even.
Then  by Lemma~\ref{ADBC-lem}   there is a unique $v\in \cA_{\DIII}^{(n)}(z)$ with $v \precsim_{\DIII}^{(n)} w$
and $v_2<v_4<\dots<v_n$, and we have $\sh_{\DIII}^{(n)}(v) = M$.
Let $u = \incDIIIDI(v) \in \cED(z)$.
Then \[\NDes(u)=\NDes^\pm(u)  = \{ (v_2,v_1), (v_4,v_3), \dots, (v_{n},v_{n-1})\}\]
by Lemma~\ref{inc-lem2}, so it follows from
Theorem~\ref{d1-prop1} that $v = \bot_{\DIII}^{(n)}(z,M)$.
This shows that our two expressions 
 are contained in each other and hence equal.

Now assume $n$ is odd. By applying Lemma~\ref{ADBC-lem} to $\incDIIIEven(w) \in \cA_{\DIII}^{(n+1)}(z^\vee)$
we deduce that
 there is a unique element $v\in \cA_{\DIII}^{(n)}(z)$ with $v \precsim_{\DIII}^{(n)} w$
and $v_3<v_5<v_7<\dots<v_n$, and we have $\sh_{\DIII}^{(n)}(v) = M$.
To see that $v=\bot_{\DIII}^{(n)}(z,M)$,
we consider its image $u =\incDIIIEven(v) \in \cA_{\DIII}^{(n+1)}(z^\vee)$.
As observed in \eqref{be-eq3}, the shape $N = \sh_{\DIII}^{(n+1)}(u)$ is always 
\[
\( M \setminus \Bigl\{ \{-v_1,v_1\}\Bigr\} \)\sqcup \Bigl\{ \{ -n-1,v_1\}, \{-v_1,n+1\}\Bigr\}.
\]
As $N$ is the image of a noncrossing matching under $\regroupD$, 
$|v_1|$ is the largest element of $\Triv{M}$,
and if $i \in [n-1]$ is even then 
$v_i < v_{i+1}$ and if these numbers are both negative then they are consecutive elements of $\Triv{M} \setminus\{|v_1|\}$.
As $v_3<v_5<\dots<v_n$ and $\sh_{\DIII}^{(n)}(v) = M$, we get that $v=\bot_{\DIII}^{(n)}(z,M)$ as needed.
\end{proof}

 Propositions~\ref{D34-prop1b} and \ref{D34-prop1c} imply
that $\sh_{\DIII}^{(n)} : \cA_{\DIII}^{(n)}(z) \to \cM_{\DIII}^{(n)}(z)$ is a surjective map. 
The following result establishes Theorem~\ref{main-thm} for type DIII.

 \begin{theorem}\label{D34-alignment-prop}
Suppose  $\gamma \in \Gamma_{\DIII}^{(n)}$ and $z =\psi_{\DIII}^{(n)}(\gamma) \in \cI_{\DIII}^{(n)}$. Then
\[
\cW_{\DIII}^{(n)}(\gamma) = \left\{ w \in \cA_{\DIII}^{(n)}(z) : \sh_{\DIII}^{(n)}(w) \text{ is $\gamma$-aligned}\right\}.
\]
\end{theorem}

\begin{proof}
First assume $n$ is even so that $z=\overline{\sigma_\gamma}$.
 Let $j = n/2$ and 
  fix an element $w \in \cA_{\DIII}^{(n)}(z)$ with a reduced expression $w= t_{i_1}t_{i_2}\cdots t_{i_m}$.
  Express $w$ in one-line notation as $w=b_1c_1b_2c_2\cdots b_jc_j$ where each $b_i <c_i$.
  In view of Proposition~\ref{d2-uv-prop} and the fact $z \in \Ifpf(\WD_n)$,  Theorem~\ref{D34-rank-prop} implies that 
  \be\label{cyc-fpf-d3-eq}
  \CycPM(z) = \{ (b,c) \in [\pm n]\times[n] : -c<b<c=z(b)\} = \{ (b_i, c_i) : i\in [j]\text{ with }{-b_i} < c_i \}
  \ee
 On the other hand, by definition  $\sh_{\DIII}^{(n)}(w) $ is the set of unordered pairs $ \{ b_i, -c_i\} $ and $ \{ -b_i, c_i\} $
for each $i \in [j]$ with $c_{i} < -b_i$.
We claim that in the notation of Lemma~\ref{weak-order-lem},
$\sh_{\DIII}^{(n)}(w) = \Lambda(P)$ for 
\[
 P = \( 
  z=
  z^0
  \xrightarrow{t_{i_1}} 
z^1
   \xrightarrow{t_{i_2}}
\cdots
   \xrightarrow{t_{i_m}}
   z^m=t_1t_3t_5\cdots t_{n-1}\)
.\]
This claim will imply the theorem
by  Lemma~\ref{weak-order-lem}.

To prove our claim,
suppose $  t_i \circ z \circ t_i \neq z$ for some $i \in \{-1,1,\dots,n-1\}$. Let $Q$ be the sequence 
\[
Q=\(t_i\circ z \circ t_i
\xrightarrow{t_i}
 z=
  z^0
  \xrightarrow{t_{i_1}} 
z^1
   \xrightarrow{t_{i_2}}
\cdots
   \xrightarrow{t_{i_m}}
   z^m=t_1t_3t_5\cdots t_{n-1}\).
\]
By \eqref{Lambda_i-eq}, the set partition
$\Lambda(Q) $ is the disjoint union of 
$\{ \{ t_i(a), t_i(b)\} : \{a,b\} \in \Lambda(P)\}$
and
\be\label{case-set-eq}
\begin{cases}
\left\{ \{i,i+1\}, \{-i,-i-1\}\right\}  &\text{if $i>0$ and $z(i) = -i-1$}\\
\left\{ \{-1,2\}, \{1,-2\}\right\}  &\text{if $i=-1$ and $z(1) = 2$}\\
\varnothing &\text{otherwise},
\end{cases}
\ee
Hence, we can deduce our claim by induction just
by checking that
$\sh_{\DIII}^{(n)}(t_i w) $ is the disjoint union of 
$\left\{ \{ t_i(a), t_i(b)\} : \{a,b\} \in \sh_{\DIII}^{(n)}(w)\right\}$ and the same set \eqref{case-set-eq}.
This is straightforward using \eqref{cyc-fpf-d3-eq} and the Demazure conjugation formulas in Section~\ref{d-conj-sect}.

Finally assume $n$ is odd
so that $z=t_0\cdot \overline{\sigma_\gamma}$.
Let $w \in \cA_{\DIII}^{(n)}(z) $ and define $M = \sh_{\DIII}^{(n)}(w)$ and $N = \sh_{\DIII}^{(n+1)}(\incDIIIEven(w))$.
  Theorem~\ref{D34-rank-prop} implies that $|w_1|$ is the largest element of $\Triv{M}$ 
  and so is in $S_+(\gamma)$ or $S_-(\gamma)$.

By the even rank case and Lemma~\ref{inc-lem-DIV}
it suffices to show that 
$M $ is $\gamma$-aligned
if and only if 
$ N$ is $\gamma^\vee$-aligned.
 In view of \eqref{be-eq3}, the reverse direction of this equivalence 
 is immediate. 
 To prove the forward direction we just need to check that if
 $M $ is $\gamma$-aligned then
\[
w_1\in S_-(\gamma)\text{ when $w_1>0$}
\quand 
{-w_1} \in S_+(\gamma)\text{ when $w_1<0$}
 \]
 since then we will have 
$\{w_1,n+1\} \subset S_-(\gamma^\vee)$ when 
$\{|w_1|,n+1\} \subset \Triv{N}$,
along with 
\[| \{|w_1|,n+1\} \cap S_+(\gamma^\vee)| = | \{|w_1|,n+1\} \cap S_-(\gamma^\vee)| = 1
\quad\text{when}\quad \{|w_1|,n+1\} \in N,
\]
which is enough to deduce that $N$ is $\gamma^\vee$-aligned.

To this end, recall that the even number $h(\gamma)$ from Definition~\ref{h-def} is given by
\be\label{hh-eq1} h(\gamma)= \tfrac{\ell_0(t_0z) -|X|}{2}+ |S_+(\gamma) \cap [n]| \ee for $X=\Neg(t_0z)=\Points(\gamma)$.
Since $M $ is $\gamma$-aligned we have 
\be\label{hh-eq2}
|S_+(\gamma) \cap [n]| \equiv \tfrac{|X|-\triv{M}}{2} + |\{ |w_1|\} \cap S_+(\gamma)| \modu 2).
\ee
On the other hand,
Theorem~\ref{D34-rank-prop} implies
that $w_1<0$ if and only if  $v=\bot_{\DIII}^{(n)}(z,M)$ also has $v_1<0$,
which by
 \eqref{parity-consult-eq} occurs if and only if 
$\tfrac{\ell_0(t_0z) +\triv{M}}{2} 
$ is even.
Thus if $w_1<0$ then 
\[
0 \equiv h(\gamma) +\tfrac{\ell_0(t_0z) +\triv{M}}{2}  \equiv 1+ |\{ |w_1|\} \cap S_+(\gamma)| \modu 2)
\]
since
 $\ell_0(t_0z)=\ell_0(z)\pm 1$ is odd, 
which implies that ${-w_1} \in S_+(\gamma)$.
Similarly,
 if $w_1>0$ then 
 \[
1 \equiv h(\gamma) +\tfrac{\ell_0(t_0z) +\triv{M}}{2}  \equiv 1+ |\{ |w_1|\} \cap S_+(\gamma)| \modu 2),
\]
which means that $w_1 \notin S_+(\gamma)$
 so $w_1 \in S_-(\gamma)$ as needed.
\end{proof}

\section{Applications}\label{app-sect}

This final section presents a few classification results that can be derived from our main theorems.

\subsection{Multiplicity-free orbit closures}

It is interesting to identify when the $K$-orbit closure $Y_\gamma$ is \defn{multiplicity-free} in the sense that the map $d_\gamma(w) = 0$
for all $w \in \cABrion(\gamma)$. 
(See \cite{Knutson2009} for an application of this property.)
The following results give a classification of such orbit closures.
Recall how $d_z$ was defined in Proposition~\ref{dgamma-prop}.

\begin{proposition}\label{dz-prop}
Suppose $(G,K)$ is of classical type and $z \in \cI^G_K$.
For the types AIII, BI, CII, and DI that depend on parameters $p$ and $q$, define $k$ and $n$ as in Table~\ref{extended-brion-tbl}.
Then:
\ben
\item[(a)] 
In types AII, AIII, CII, and DIII, one has $d_z=0$ for all $z \in \cI^G_K$.

\item[(b)] In type AI,
one has $d_z=0$ if and only if $z =1$.

\item[(c)] In type BI one has $d_z=0$
if and only if  $\neg(z)=k$ and $|z(i)| =i$ for all $i \in [n]$.

\item[(d)] In type CI
one has $d_z=0$ if and only if $\neg(z)\leq 1$ and $|z(i)| =i$ for all $i \in [n]$.

\item[(e)] In type  DI one has $d_z=0$ if and only if  $\neg((t_0)^k z)=k$ and $|z(i)| = i$ for all $i\in[n]$.

\een
\end{proposition}

\begin{proof}
Parts (a) and (b) are immediate from the formulas in Table~\ref{extended-brion-tbl}.
Part (e) likewise follows
after noting from Table~\ref{rs-image-tbl}
that if $z \in \cI^G_K$ in type DI then $\neg((t_0)^kz) \geq k$.

For parts (c) and (d) we use Theorems~\ref{main-thm} and \ref{main-thm2} in types BI and CI.
First, in type BI it follows from the material in Section~\ref{BI-sect}
that any given $w \in  \cEABrion(z)$ has $\ell_0(w) = 0$ if and only if $\neg(z) = k$ (as in this type any $z \in \cI^G_K$ has $\neg(z) \geq k$) and $z(i) \in [n]$ for all $i \in [n] \setminus \Neg(z)$, and if these conditions hold then $|\{ i \in [n] : 0 <z(i) < i\}| =0$ if and only if $|z(i)| =i$.

Similarly, in type CI when $d_z(w)  =|\{i \in [n] :  z(i) < i\}| - \ell_0(w)$,  it follows from Theorems~\ref{main-thm} and \ref{main-thm2} using the definitions in Section~\ref{CI-sect}
that every $w \in \cEABrion(z)$ has 
\[\ell_0(w) \geq \left\lceil \tfrac{\neg(z)}{2}\right\rceil + \tfrac{\ell_0(z) - \neg(z)}{2} = 
\left\lceil \tfrac{\ell_0(z)}{2}\right\rceil
\text{\ \ so\ \ }
d_z(w) \leq | \{ i \in [n] :  0 < z(i)<i\}| + \left\lfloor \tfrac{\ell_0(z)}{2}\right\rfloor.\]
Moreover,  equality is achieved for $w = \bot_{\CI}(z,M)$ when $M\in \NCSP(\Neg(z))$ has $\triv{M}\leq 1$.
The upper bound on $d_z(w)$ is zero precisely when at most one $i \in [n]$ has $z(i) < 0$
and $z$ fixes all other points in $[n]$,
which happens if and only if $\neg(z)\leq 1$ and $|z(i)| =i$ for all $i \in [n]$.
\end{proof}

Combining the previous result with the following tells us precisely when $Y_\gamma$ is multiplicity-free.
\begin{proposition}
Suppose $(G,K)$ is of classical type. Choose $\gamma \in \Gamma^G_K$ and let $z = \RSphi(\gamma) \in \cI^G_K$.
\ben
\item[(a)] In type CI, one has $d_\gamma =0$ if and only if 
\[S_+(\gamma) \subset \PP\text{ or }  S_-(\gamma) \subset \PP
\quand\text{$a+b=0$ for all $\{a,b\}\in M(\gamma)$.}
\] If these conditions hold then $|z(i)| =i$ for all $i>0$.

\item[(b)] In all other types, we have $d_\gamma=0$ if and only if $d_z=0$.
\een
\end{proposition}

\begin{proof}
Assume $(G,K)$ has type CI so that $z = \overline{\sigma_\gamma}$.  
Let $X = \Points(\gamma) =\Neg(z)$.
In view of Definition~\ref{CI-gen-def}, Theorems~\ref{main-thm} and \ref{main-thm2}
imply that $d_\gamma =0$
if and only if 
\[
|\{ i \in [n] : z(i)<i\}| = \ell_0\(\bot_{\CI}^{(n)}(z,N)\)
\quad\text{for all $\gamma$-aligned $N \in \NCSP(X)$.}\]
However, one can check that 
\[\ba \ell_0\(\bot_{\CI}^{(n)}(z,N)\)&=
\triv{N} + \tfrac{\ell_0(z) - \neg(z)}{2} + \tfrac{\neg(z) - \triv{N}}{2} 
 = \tfrac{\ell_0(z) +\neg(z)}{2} + \tfrac{\triv{N}-\neg(z)}{2} 
\\&= |\{ i \in [n] : -i \leq  z(i) < 0 \}| + \tfrac{ \triv{N}-\neg(z)}{2}.
\ea
\]
Since $\triv{N} \leq \neg(z)$, the last expression is equal to $|\{ i \in [n] : z(i)<i\}| $ 
precisely when 
\[\triv{N} = \neg(z)\quand \{ i \in [n] : z(i) < i\} = \{ i\in [n] : -i \leq z(i) < 0\}.\]
As $\gamma$ is skew-symmetric, the first condition holds for all $\gamma$-aligned $N \in \NCSP(X)$
if and only if $S_+(\gamma) \subset \PP$ or $S_-(\gamma) \subset \PP$,
and the second condition holds if and only if $|z(i)| = i$ for all $i \in [n]$,
or equivalently if $a+b=0$ for all $\{a,b\}\in  M(\gamma)$.
This proves part (a).

In type BI, the argument in the proof of Proposition~\ref{dz-prop}
shows that if $d_z(w) =0$ for any $w \in \cEABrion(z)$ then $d_z=0$, so $d_\gamma=0$ if and only if $d_z=0$.
The same is true in all other classical types besides CI
since then $d_z$ takes the same value on all $w \in \cEABrion(z)$. This proves part (b).
\end{proof}

\subsection{Uniform and alternating clans}

Continue to assume $(G,K)$ is of classical type. 
Let $\gamma \in \Gamma^G_K$
and $z = \phiRS(\gamma) \in \I^G_K$.
As our last application, we determine when $ \cW^G_K(\gamma) = \cA^G_K(z)$.
This involves two complementary properties.

\begin{definition} The index $\gamma \in \Gamma^G_K$
is \defn{uniform} if $|\cM^G_K(z)| = 1$.
\end{definition}

Recall that if $\gamma = (S_+,S_-,M)$ is a clan then  $\APoints(\gamma) = S_+ \sqcup S_-$. Let $\points(\gamma) = |\APoints(\gamma)|$.

\begin{proposition}
Suppose $(G,K)$ is of classical type and $\gamma \in \Gamma^G_K$.
For the types AIII, BI, CII, DI, and DIII that depend on extra parameters, let $p$, $q$, and $n$ be as in Table~\ref{extended-brion-tbl}.
Then the index $\gamma$ is uniform if and only if the relevant condition below holds:
\ben
\item[(a)] Types AI and AII: no condition (the index is always uniform).
\item[(b)] Type AIII:  $\points(\gamma)= |p-q|$, or $\points(\gamma)=2$ when $p=q$.

\item[(c)] Type  BI: $\points(\gamma) \in  |p-q| + \{0,2\}$.

\item[(d)] Type CI: $\points(\gamma) \in \{0,2\}$.

\item[(e)] Types CII and DI: $\points(\gamma)= |p-q|$, or $\points(\gamma)=4$ when $p=q$.

\item[(f)] Type DIII: $\points(\gamma) \in \{0,2,4\}$.
\een
\end{proposition}

 \begin{proof}
Let $z = \phiRS(\gamma)$. 
Then $\cM^G_K(z)$ is a set of matchings on $\APoints(\gamma)$,
and the set $X$ appearing in the third column of 
Table~\ref{shapes-summary-tbl}
either has $\APoints(\gamma) = X$ (in type AIII), $ X\sqcup  (-X) \sqcup \{0\}$ (in type BI), or $X\sqcup (-X)$ (in types CI, CII, DI, DIII).
The proposition follows by 
inspecting Table~\ref{shapes-summary-tbl},
noting that  
$
|\NCSP(X:k)| = \tbinom{m}{(m-k)/2}-\tbinom{m}{(m-k)/2-1}= 1$
(for $m=|X|$ and $0\leq k \leq m$ with $m-k$ even) only if $m=k$ or $k=0<m=2$,
and that
$|\NCSPB(X:k)|= \tbinom{m}{\lfloor(m-k)/2\rfloor}=1$
only if $m\in\{k,k+1\}$.
 \end{proof}
 
We can have
$ \cW^G_K(\gamma) = \cA^G_K(z)$ outside the uniform case. 
 
 \begin{definition}
 \label{alt-def}
  Fix an index $\gamma \in \Gamma^G_K$.
 \ben
 \item[(a)] In type DIII, we define $\gamma$ to be \defn{alternating} if $\points(\gamma)\in \{0,2\}$.

\item[(b)] In all other types, we define $\gamma$ to be \defn{alternating} if 
$|S_+(\gamma)\cap\{a,b\}| = |S_-(\gamma)\cap \{a,b\}|=1$
whenever $a$ and $b$ are consecutive nonnegative elements of $S_+(\gamma)\sqcup S_-(\gamma)$.\een
\end{definition}

All signless clans (hence, all indices in types AI and AII) are alternating.
In types BI and CI,
condition (b) means that 
 removing all integers from any one-line representation of $\gamma$
yields the alternating sequence
$
(+,-,+,-,+,\dots)$ or $ (-,+,-,+,-,\dots)
$.
In types CII and DI, condition (b) means that the same operation yields
\[ 
(\dots,+,-,+,-,+,+,-,+,-,+,\dots)\quord (\dots,-,+,-,+,-,-,+,-,+,-,\dots).
\]

\begin{proposition}\label{alternating-prop}
Suppose $(G,K)$ is of classical type. Choose  $\gamma \in \Gamma^G_K$ and let $z = \phiRS(\gamma) \in \I^G_K$.
Then
$ \cW^G_K(\gamma) = \cA^G_K(z)$
if and only if 
$\gamma$ is uniform or alternating.
\end{proposition}

 \begin{proof}
By Theorem~\ref{main-thm},
 we have $ \cW^G_K(\gamma) = \cA^G_K(z)$ if and only if every $M \in \cM^G_K(z)$ is $\gamma$-aligned.
Inspecting Remark~\ref{ncsp-gen-rmk} shows that this holds when  $\gamma$ is alternating. The same holds when $\gamma$ is uniform since then $ \cM^G_K(z)$ has only one element, but the set of $\gamma$-aligned matchings is always nonempty.
Conversely, if $\gamma$ is not uniform or alternating then it is straightforward  
from Table~\ref{shapes-summary-tbl} to construct some $M \in \cM^G_K(z)$ that is not $\gamma$-aligned.
%Thus (b) and (c) are equivalent.
  \end{proof}

\end{document}